\documentclass[11pt,reqno]{amsart}
\usepackage[margin=3.0cm]{geometry}
\usepackage{genpack1}

\dmo{\injo}{inj}
\dmo{\sub}{sub}
\dmo{\aut}{aut}
\nc{\edges}{\mathsf{e}}
\nc{\verts}{\mathsf{v}}
\dmo{\DKL}{D_{KL}}
\dmo{\Meas}{M}
\nc{\fun}{f}

\nc{\sphere}{\mathbb{S}}
\dmo{\St}{St}
\nc{\ball}{\mathbb{B}}
\nc{\bla}{\bs{\lambda}}
\nc{\bmu}{\bs{\mu}}
\nc{\bu}{\bs{u}}
\nc{\bv}{\bs{v}}
\nc{\by}{{\bs{y}}}
\nc{\ba}{\bs{a}}
\nc{\bx}{\bs{x}}
\dmo{\RR}{\mathsf{UT}}
\dmo{\LL}{\mathsf{LT}}

\dmo{\Fluct}{\mathsf{Fluct}}
\dmo{\bd}{\mathsf{Max}}
\dmo{\Mat}{M}
\dmo{\Sym}{Sym}
\nc{\Lab}{L}
\nc{\cstar}{\wt{C}} 

\nc{\phigen}{\phi}
\nc{\psigen}{\psi}

\dmo{\fluct}{{fluct}}
\dmo{\Ecomplex}{E_{complex}}
\dmo{\excep}{{excep}}
\dmo{\pexcep}{P_{\excep}}

\nc{\Gnp}{{\bs{G}}}
\nc{\Anp}{{\bs{A}}}
\nc{\BBB}{{\bs{B}}}
\nc{\tAn}{{\wt{\Anp}}}
\nc{\lamR}{ {\underset{\rightarrow}{\lambda}}}
\nc{\ppower}{\kappa}
\nc{\errRL}{{\eps}}
\nc{\errRLs}{{\delta_o}}

\nc{\Verts}{n}
\nc{\Net}{N}
\dmo{\Cyc}{\mathsf{C}}
\dmo{\jay}{{\sf J}}
\dmo{\id}{{\sf I}}
\dmo{\Pol}{{\sf P}}
\nc{\rnk}{k}
\nc{\subL}[2]{{\{ #1 \le #2 \}}}
\nc{\supL}[2]{{\{ #1 \ge #2 \}}}
\nc{\cLc}{\supL h t}
\nc{\cLcd}{\supL h {t-\delta}}
\nc{\cLl}{\{h \le t\}}

\begin{document}

\title[Large deviations of subgraph counts for sparse random graphs]{Large deviations of subgraph counts \\ for sparse Erd\H os--R\'enyi graphs}

\author[N.\ Cook]{Nicholas Cook$^\ddagger$}\thanks{${}^\ddagger$Partially supported by NSF postdoctoral fellowship DMS-1606310}
 \address{$^\ddagger$Department of Statistics, Stanford University, Stanford, CA 94305}\email{nickcook@stanford.edu}
\author[A.\ Dembo]{Amir Dembo$^{\mathsection}$}\thanks{$
{}^{\mathsection}$Partially supported by NSF grant DMS-1613091}
\address{$^{\mathsection}$Department of Mathematics, Stanford University, 
Stanford, CA 94305}\email{adembo@stanford.edu}

\date{\today}

\begin{abstract}
For any fixed simple graph $H=(V,E)$ 
and any fixed $u>0$, we establish the leading order of the exponential rate function for the probability that the number of copies of $H$ in the Erd\H{o}s--R\'enyi graph $G(\Verts,p)$ exceeds its expectation by a factor $1+u$, assuming $\Verts^{-\kappa(H)}\ll p\ll1$, with $\kappa(H) = 1/(2\Delta)$, 
where $\Delta\ge 1$ is the maximum degree of $H$.
This improves on a previous result of Chatterjee and the second author, who obtained $\kappa(H)=c/(\Delta|E|)$ for a constant $c>0$. Moreover, for the case of cycle counts we can take $\kappa$ as large as $1/2$.
We additionally obtain the sharp upper tail for Schatten norms of the 
adjacency matrix, as well as the sharp lower tail for counts of graphs 
for which Sidorenko's conjecture holds.
As a key step, we establish quantitative versions of Szemer\'edi's regularity lemma and the counting lemma, suitable for the analysis of random graphs in the large deviations regime.
\end{abstract}


\maketitle

\let\oldtocsubsection=\tocsubsection
\renewcommand{\tocsubsection}[2]{\hspace*{.0cm}\oldtocsubsection{#1}{#2}}
\let\oldtocsubsubsection=\tocsubsubsection
\renewcommand{\tocsubsubsection}[2]{\hspace*{1.8cm}\oldtocsubsubsection{#1}{#2}}

\setcounter{tocdepth}{2}

\section{Introduction}
\label{sec:intro}

\subsection{The infamous upper tail for homomorphism counts}

Given a graph $H=(V,E)$, 
the associated \emph{homomorphism counting function} on graphs $G$ over vertex set $[\Verts]$ is given by
\begin{equation}	\label{def:homG}
\hom(H,G) := \sum_{\varphi:V\to [\Verts]} \prod_{\{u,v\}\in E} A_G(\varphi(u),\varphi(v)),
\end{equation}
where $A_G$ denotes the $\Verts \times \Verts$ adjacency matrix for $G$.
That is, $\hom(H,G)$ counts the number of edge-preserving maps from $V$ to $[n]$. 
When $H=\Cyc_\ell$, the cycle on $\ell\ge3$ vertices, we have
\[
\hom(\Cyc_\ell,G) = \Tr A_G^\ell. 
\] 
There are standard relations between homomorphism counts $\hom(H,G)$ and subgraph counts $\sub(H,G)$ (see \cite[Chapter 5]{Lovasz-book}). As an example, we have $\hom(\Cyc_3,G) = 6\sub(\Cyc_3,G)$. While the relationship is more complicated for general $H$, in the regime of sparsity we consider, up to a negligible error, they are related by an easily computed combinatorial factor. 
We thus focus hereafter 
on the (more convenient) homomorphism counts.

For $\Verts$ large and $p\in(0,1)$ possibly depending on $\Verts$, let $\Gnp\sim G(\Verts,p)$ be the Erd\H{o}s--R\'enyi random graph on vertex set $[\Verts]$.
The ``infamous upper tail problem" \cite{JaRu02} is to determine the asymptotic exponential rate function for the probability that $\hom(H,\Gnp)$ exceeds its expectation by a constant factor, that is to estimate
\begin{equation}	\label{def:Rate}
\RR_{\Verts,p}(H,u) := -\log \pr\big(\,\hom(H,\Gnp)\ge (1+u) \Verts^{\verts(H)} p^{
\edges(H)} \,\big), \quad 
 u>0 \,,
\end{equation}
where here and in the sequel $\verts(H):=|V|$, $\edges(H) :=|E| \ge 1$, with $H$ simple.

In the dense regime with $p\in (0,1)$ fixed independent of $\Verts$, Chatterjee and Varadhan established a large deviations principle for $G(\Verts,p)$, viewed as a sequence of measures on the infinite-dimensional space of \emph{graphons} with the cut metric topology \cite{ChVa11}. Since homomorphism counting functions are continuous under this topology, they could consequently
establish $\lim_{\Verts\to \infty} \frac{1}{\Verts^2} \RR_{\Verts,p}(H,u)$ as the solution to an optimization problem over graphons, which was subsequently 
analyzed in the special case $H=\Cyc_3$ by 
Lubetzky and Zhao in \cite{LuZh12}. 

Graphon theory provides a topological reformulation of the classic \emph{regularity method} in extremal graph theory, which rests on two key facts: Szemer\'edi's regularity lemma, which is related to the compactness of graphon space, and the counting lemma, which asserts the continuity of the homomorphism counting functions.

Here we are concerned with the 
sparse regime, with $p=\Verts^{-c}$ for some constant $c\in (0,1)$, which falls outside the purview of graphon theory.
Indeed, as is well known, the regularity lemma is useless in this regime of sparsity. 
In \cite[Section 11]{Chatterjee:survey} (cf.\ Open Problem 5 there), Chatterjee 
asks for a version of the regularity method suitable for the study of large deviations for sparse random graphs. 
A number of extensions of graph limit theory have been developed in recent years to accommodate sparse graphs (with edge density $o(1)$ but growing average degree); see for instance \cite{BCCZ-LP1,BCCZ-Lp2,Frenkel,Szegedy15:sparse,Janson16:sigma-finite,BCCH,NeOs18,BaSz18}. However, it seems that none of these is suitable for studying large deviations of homomorphism counts.

In the present work we develop quantitative versions of the regularity and counting lemmas that are specially tailored for applications to large deviations (see Section \ref{sec:regularity}).
In particular, we use them to obtain sharp asymptotics for the upper tail \eqref{def:Rate} for $p=\Verts^{-c}$ with $c\in (0,\kappa(H))$ for a suitable constant $\kappa(H)>0$ depending only on $H$. For the case that $H$ is a cycle, our methods yield $\kappa(H)$ which is in some sense optimal. 
The results in Section \ref{sec:regularity} are of independent interest and potentially useful for other questions in graph theory.
It is also possible that some features of our approach could be useful for problems outside graph theory, such as large deviations for the number of arithmetic progressions in sparse random sets, which have been studied in \cite{ChDe14,BGSZ16}.

The upper tail problem \eqref{def:Rate} has seen considerable activity in the past few years. Before stating our results, we review what is already known (for additional background see \cite{Chatterjee:survey}), starting with the triangle homomorphism counting function $\hom(\Cyc_3,\cdot)$. 
In this case, one easily computes 
\[
\e \hom(\Cyc_3,\Gnp) = \Verts(\Verts-1)(\Verts-2) p^3 = (1+o(1))\Verts^3p^3.
\]
(Unless stated otherwise, all asymptotic notation is with respect to the limit $\Verts\to \infty$; see Section \ref{sec:notation} for our notational conventions.)
A moment's thought yields upper bounds 
on $\RR(\Cyc_3,\cdot)$
which turn out to be asymptotically tight (at least for some range of $p$). 
Indeed, one way to create on the order of $\Verts^3p^3$ extra triangles is via the event
\begin{equation}	\label{Clique}
\text{ Clique($a$): } \qquad \text{ Vertices $1,\dots, \lf a\Verts p \rf$ form a clique,}
\end{equation}
for fixed $a>0$. The probability of this event is 
\[
\pr(\text{Clique}(a)) = p^{{\lf a\Verts p \rf \choose 2}} \ge \exp \Big\{ - \frac12a^2\Verts^2p^2\log(1/p)  \Big\} \,.
\]
On this event the clique contributes $\sim (a\Verts p)^3$ extra triangle homomorphisms (assuming $\Verts p\to \infty$ and $p=o(1)$).
Thus, taking $a=u^{1/3}$, and intersecting with the high-probability 
(and independent) event that the complement of the clique contains $(1+o(1))\Verts^3p^3$ triangles, we have
\[
\RR_{\Verts,p}(\Cyc_3,u) \le  (1+o(1))\frac{u^{2/3}}2 \Verts^2 p^2\log (1/p).
\]
We get another upper bound on $\RR_{\Verts,p}(\Cyc_3,u)$ by considering the event
\begin{equation}	\label{Hub}
\text{Hub($b$): } \qquad \text{ Vertices $1,\dots, \lf b\Verts p^2 \rf$ are connected to all other vertices}
\end{equation}
for fixed $b>0$ (assuming $p\gg \Verts^{-1/2}$).
On this event, every edge in the complement of the hub $[\lf b\Verts p^2 \rf]$ forms a triangle with every vertex in the hub, giving $\sim 3b\Verts^3p^3$ extra triangle homomorphisms (if $p=o(1)$). Taking $b= u/3$, we obtain
$
\RR_{\Verts,p}(\Cyc_3,u) \le  (1+o(1)) \frac{u}3 \Verts^2 p^2\log (1/p).
$
Thus we have
\begin{equation}	\label{R3.upper}
\RR_{\Verts,p}(\Cyc_3,u) \le (1+o(1))\min\bigg\{ \frac{u^{2/3}}{2}, \frac{u}{3}\bigg\} \Verts^2p^2\log(1/p).
\end{equation}
There is a third natural event to consider, that $\Gnp$ has on the order of $\Verts^2p$ extra edges distributed uniformly across the graph. 
Indeed, this event turns out dominate the tail event for triangle counts in
much of the dense regime (with $p$ fixed) \cite{ChVa11, LuZh12}.
However, a short computation reveals that in the sparse regime $p\to 0$ 
this event can not compete 
with the events Clique and Hub
(though, as seen in Section \ref{sec:lowertails}, it does give the 
leading order contribution for the \emph{lower} tails for certain functions).

Lower bounds on $\RR_{\Verts,p}(\Cyc_3,t)$ (that is, upper bounds on upper tail for triangle counts in $\Gnp$) have a long history in the literature. 
Using the machinery of polynomial concentration, Kim and Vu showed \cite{KiVu04}
\[
\RR_{\Verts,p}(\Cyc_3,u) \,\gs_u\, \Verts^2p^2
\]
for all $p\ge \Verts^{-1}\log\Verts$ and $u>0$, which matches the upper bound \eqref{R3.upper} up to the factor $\log (1/p)$; analogous results for general sub-graphs $H$ were obtained in \cite{JOR04}.
The missing logarithm was found in work of Chatterjee \cite{Chatterjee12-triangles} and DeMarco and Kahn \cite{DeKa12:triangles}, who showed
\[
\RR_{\Verts,p}(\Cyc_3,u) \asymp_u \Verts^2 p^2 \log(1/p)
\]
for all $p\ge \Verts^{-1}\log\Verts$ and $u>0$.

The focus then shifted to the asymptotic dependence of $\RR(\Cyc_3,\cdot)$ 
on $u$, i.e.\ to find a formula for $c_3(u)$ such that for any fixed $u>0$, 
\[
\RR_{\Verts,p}(\Cyc_3,u)\sim c_3(u) \Verts^2p^2\log (1/p) \,.
\]
A breakthrough was made in \cite{ChDe14}, which introduced a general \emph{nonlinear large deviations} framework, and as an application showed that 
\begin{equation}	\label{main-goal}
\RR_{\Verts,p}(H,u) \sim \upphi_{\Verts,p}(H,u),\qquad \Verts^{-\kappa(H)}\ll p\ll1,
\end{equation}
for some constant $\kappa(H)>0$ depending only on $H$, where $\upphi_{\Verts,p}(H,u)$ 
is the solution of the variational problem  \eqref{def:phi}.
While the asymptotic \eqref{main-goal} is expected to hold with $\kappa(H)=1/\Delta$, where $\Delta=\Delta(H)$ is the maximum degree of $H$
(see \cite[Open Problem 4]{Chatterjee:survey}), 
the proof of \cite{ChDe14} gave only $\kappa(H)=c/(\Delta|E|)$
for some absolute constant $c>0$.
Even for $H=C_3$ they only got $\kappa(\Cyc_3)=1/42+\eps$.
The latter was improved to $\kappa(\Cyc_3)=1/18+\eps$ by Eldan,
as a consequence of general advances in the theory of nonlinear large deviations \cite{Eldan-NMF}. Combined with the solution in \cite{LuZh14} of 
\eqref{def:phi} for $H=\Cyc_3$, this
gave a matching lower bound for \eqref{R3.upper} in the range
$\Verts^{-1/18+\eps}\le p\ll1$.

In Theorem \ref{thm:hom} below we obtain \eqref{main-goal} with $\kappa(H)=1/(2\Delta)$ (with a wider range for irregular graphs), which drastically improves the previous bound $\kappa(H)=c/(\Delta|E|)$ for general $H$. Moreover, for the case of cycles we obtain the essentially optimal exponent $\kappa(\Cyc_\ell)=1/2+\eps$ for $\ell\ge 4$, and $\kappa(\Cyc_3)=1/3$; see Corollary \ref{cor:cycles.ut}.
Our general approach also gives bounds for the \emph{lower tail} 
for counts of cycles and of graphs having the ``Sidorenko" property, as well as the Schatten norms of the random adjacency matrix $A_\Gnp$ 
(see \eqref{def:Schatten}).

\subsection{Results for homomorphism counts}

Our main result shows that $\RR_{n,p}(H, u)$ of \eqref{def:Rate} is asymptotically given by the solution of a certain entropic variational problem, which we now formulate. As we also treat other functionals besides homomorphism counts, we begin with the general setup.
For $d\in \N$,  $x\in [0,1]^d$ and $p \in [0,1]$ denote
\begin{align}\label{def:Ipx}
I_p(x) 
& := \sum_{i=1}^d \Big[ x_i \log\frac{x_i}p + (1-x_i) \log \frac{1-x_i}{1-p} \Big] \,,
\end{align}
which is the Kullback--Leibler divergence $\DKL(\mu_x\|\mu_p)$ between the 
(product) Bernoulli measures with centers of mass $x=(x_i)$ 
and $p=(p,\dots,p)$ (we make the convention $0\log 0:=0$).
Set for any $\cE\subseteq \R^d$,
\begin{align}\label{def:IpT}
I_p(\cE) & := \inf\{ I_p(x): x \in \cE\cap [0,1]^d\}\,.
\end{align}
For $h:[0,1]^d\to \R$ and $t\in \R$, the upper- and lower-tail entropic variational problems 
are denoted 
\begin{align}\label{def:phigen}
\phigen_p(h,t)  := I_p(\cLc) &= \inf\big\{ \, I_p(x): x\in [0,1]^d,\, h(x) \ge t\, \big\},\\
\psigen_p(h,t)  := I_p( \cLl) &= \inf\big\{ \, I_p(x): x\in [0,1]^d,\, h(x) \le t\, \big\}.
\label{def:psigen}
\end{align}

Given a graph $H=(V,E)$, the homomorphism counting function of \eqref{def:homG} 
extends to symmetric $\Verts\times\Verts$ matrices $X$ as follows:
\begin{equation}	\label{def:homX}
\hom(H,X) := \sum_{\varphi: V\to [\Verts]} \prod_{e=\{v,w\}\in E} X_{\varphi(v)\varphi(w)}.
\end{equation}
(When $E=\emptyset$ we take the empty product to be 1, so that $\hom(H, X)= \Verts^{\verts(H)}$ in this case.)
We denote by $\cX_\Verts$ the set of all symmetric $\Verts\times\Verts$ matrices with entries in $[0,1]$ and zeros on the diagonal, and by $\cA_\Verts\subset\cX_\Verts$ the set of adjacency matrices for graphs on $[\Verts]$ vertices. 
We will usually (but not always) deal with the objects \eqref{def:phigen}--\eqref{def:psigen} 
with $h=\Verts^{-\verts(H)} p^{-\edges(H)}\hom(H,\cdot)$ for some fixed $H$, where we take $d={\Verts\choose2}$ and identify $[0,1]^d \cong \cX_\Verts$.
In this case we have
\begin{equation}	\label{def:IpX}
I_p(X) = \sum_{1\le i<j\le \Verts} I_p(x_{ij})\,,\qquad X=(x_{ij})\in \cX_\Verts\,,
\end{equation}
and the
variational problems are denoted
\begin{align}
 \upphi_{\Verts,p}(H,u)  &:= \inf \Big\{\, I_p(X): X\in \cX_\Verts, \; \hom(H, X)\ge (1+u) \Verts^{\verts(H)}p^{\edges(H)} \,\Big\}	\,,\quad u>0, \label{def:phi}\\	
\uppsi_{\Verts,p}(H,u) &:= \inf \Big\{\, I_p(X): X\in \cX_\Verts, \; \hom(H, X)\le (1-u) \Verts^{\verts(H)}p^{\edges(H)} \,\Big\}\,, \quad 0<u\le1.	\label{def:psi}
\end{align}

Our results here and in Subsections 
\ref{sec:Schatten1}--\ref{sec:lowertails}
establish the above 
quantities as the large deviation rate
for the upper and lower tails, respectively, of
the associated homomorphism counts in $\Gnp$,
with the corresponding expressions for
Schatten norms of the adjacency matrix. 

Previous works \cite{ChDe14, Eldan-NMF} have shown 
that \eqref{main-goal} holds
with $\ppower(H)=c/(\Delta|E|)$ for some constant $c>0$.
For our main result, we show the above holds with $\ppower(H)$ as small as  
$1/(2\Delta)$, bringing $\ppower(H)$ within a factor two
of the conjectured threshold $1/\Delta$. In fact, our exponent is expressed in terms of the following, 
generally smaller, quantity:
\begin{equation}\label{def:Deltas}
\Delta_\star(H) := 
\frac{1}{2} \max_{\{v_1,v_2\} \in E}
\big\{ \deg_H(v_1) + \deg_H(v_2)
\big\}  \ge 1 \,.
\end{equation}
Note that $\Delta(H)+1\le 2\Delta_\star(H)\le 2\Delta(H)$, where the first bound is tight (for instance) when $H$ is a star on $\Delta+1$ vertices and the second is tight when $H$ is $\Delta$-regular.
\begin{theorem}[Main result]
\label{thm:hom}
For any fixed, non-empty, simple graph $H$ and $u>0$, if $p=p(n)$ satisfies
\begin{equation}	\label{prange.hom}
\Verts^{-1} \log\Verts \ll p^{2\Delta_\star (H)}\ll1,
\end{equation}
with $\Delta_\star(H)$ as in \eqref{def:Deltas}, then 
\begin{equation}	\label{UT.hom}
\RR_{\Verts,p}(H,u) = (1+o(1)) \upphi_{\Verts,p}(H,u +o(1)).
\end{equation}
\end{theorem}
Our next result gives asymptotics for $\RR_{\Verts,p}(H,u)$ in a wider range of $p$ for the case of cycles $H=\Cyc_\ell$, reaching the essentially optimal range $\ppower(\Cyc_\ell) = 1/2+\epsilon$ when $\ell\ge 4$. 
We additionally obtain asymptotics for the lower tail event; analogously to \eqref{def:Rate} we denote the lower tail functions
\begin{align}	\label{def:LL}
\LL_{\Verts,p}(H,u) &:= -\log\pr\big( \hom(H,\Gnp)\le (1-u) \Verts^{|V|} p^{|E|}\big)\,.
\end{align}
Theorem \ref{thm:cycles.quant} provides a quantitative version of the following result.
\begin{theorem}[Large deviations for cycles counts, qualitative version]
\label{thm:cycles}
Fix an integer $\ell > 2$.
If $0<p\ll1$ satisfies
\begin{equation}	\label{prange.cyc.upper}
p \gg \max\bigg(\Verts^{\frac{2}{\ell}-1},\frac{ (\log\Verts)^{\frac{\ell}{2\ell-4} }}{\sqrt{\Verts}}\bigg)\,,
\end{equation}
then for any fixed $u>0$,
\begin{align}\label{eq:upper-tail.cyc}
\RR_{\Verts,p}(\Cyc_\ell,u)\ge \upphi_{\Verts,p}\big(\Cyc_\ell,  u-o(1)\big) +o(\Verts^2p^2\log(1/p)) .  
\end{align}
If
\begin{equation}	\label{prange.cyc.lower}
\left( \frac{\log \Verts}\Verts\right)^{\frac{\ell-2}{2\ell-2}} \ll p\ll1,
\end{equation}
then for any fixed $0<u\le1$, 
\begin{align}\label{eq:lower-tail.cyc}
\LL_{\Verts,p}(\Cyc_\ell,u) \ge \uppsi_{\Verts,p}\big(\Cyc_\ell, u - o(1)\big) +o(\Verts^2p) .
\end{align}
\end{theorem}
\begin{remark}
Asymptotically matching upper bounds in \eqref{eq:upper-tail.cyc} and \eqref{eq:lower-tail.cyc} can be obtained by modification of the tilting argument in \Cref{sec:hom.UB}. Skipping this here, we note that at least for \eqref{eq:upper-tail.cyc} such an upper bound immediately follows from consideration of the clique and hub events and \Cref{thm:BGLZ} below; cf.\ \Cref{cor:cycles.ut}.
\end{remark}
\begin{remark} 
Ignoring the log factors, the exponent $\frac{\ell-2}{2\ell-2} =\frac12 -\frac{1}{2\ell-2} $ of $\Verts$ in \eqref{prange.cyc.lower} asymptotically matches the exponent $1/2$ in \eqref{prange.cyc.upper} as $\ell\to \infty$. 
For the case of even $\ell$, Theorem \ref{thm:Schatten.lower} extends
\eqref{eq:lower-tail.cyc} to hold \emph{for all} $p=p(\Verts)\in (0,1)$.
For the case $\ell=3$, whereas \eqref{prange.cyc.lower} enforces $p\gg ((\log\Verts)/\Verts)^{1/4}$, recent independent work of Kozma and Samotij \cite{KoSa} establishes \eqref{eq:lower-tail.cyc} for $p\gg \Verts^{-1/2}$.
\end{remark}

\begin{remark}	\label{rmk:Augeri}
In the independent work \cite{Augeri18}, posted to arXiv shortly after the first version of this paper, Augeri obtains \eqref{eq:upper-tail.cyc} for all $\ell\ge3$ and $p \gg  (\log\Verts)^2/\sqrt{\Verts}$. Whereas her result is
an outcome of a general advance on large deviations for nonlinear functions on product spaces having the \emph{low-complexity gradient} condition used in \cite{ChDe14, Eldan-NMF}, in the present work 
we avoid this condition. Augeri's
improvement over our result for $\ell=3$ does not stem from the low-complexity gradient approach; rather, she
eliminates the first term in the maximum in \eqref{prange.cyc.upper} by relying on strong concentration of the empirical spectral measure of $\Anp$ around its even, semi-circle, limit. 
Indeed, we can recover
her improvement for $\ell=3$, without appealing 
to a low-complexity gradient,
by replacing our use of the Schatten norm $\|\Anp\|_{S_3}$ to 
control the bulk contribution with
the concentration results from \cite[Proof of Lemma 4.1]{Augeri18}.
\end{remark}

\begin{remark}
As the quantitative version Theorem \ref{thm:cycles.quant} shows, we can allow $\ell=\ell(\Verts)$ to grow at rate $(\log\Verts)^{o(1)}$, and we can further allow $\ell\sim (\log\Verts)^C$ for fixed $C<\infty$ at the expense of increasing the power of the logarithmic corrections in the lower bounds on $p$ by $O(C)$ 
(as can be seen from \eqref{bd:large-ell}, taking $W(\Verts)$ to grow poly-logarithmically).
\end{remark}

Theorems \ref{thm:hom} and \ref{thm:cycles} show the upper and lower tails \eqref{def:Rate} and \eqref{def:LL} are asymptotically given (or bounded) by the respective variational problems \eqref{def:phi} and \eqref{def:psi}.
The following result of \cite{BGLZ16} (extending the earlier work \cite{LuZh14} for the case of cliques), solves the upper tail variational problem \eqref{def:phi} in a wide range of values of $p$.
\begin{theorem}[\!\!\cite{BGLZ16}]
\label{thm:BGLZ}
Fixing a connected graph $H=(V,E)$, let 
$H^\star=H[V^\star]$ denote the induced subgraph on the subset 
$V^\star \subseteq V$ of vertices of maximal degree $\Delta \ge 2$
(so $H^\star=H$ when $H$ is regular). 
For $u>0$ let $\theta_H(u)$ be the unique $\theta>0$ satisfying
$\Pol_{H^\star}(\theta)=1+u$,
for the independence polynomial $\Pol_{H^\star}(x) = 1+ \sum_{k=1}^{|V^\star|} 
a_k x^k$, where $a_k$ counts the
independent sets of size $k$
in $H^\star$. 
For $\Verts^{-1/\Delta}\ll p\ll1$
and any fixed $u> 0$,
\[
\upphi_{\Verts,p}(H,u) = (c_H(u)+o(1))\Verts^2p^\Delta\log(1/p) \,
\]
where
\begin{equation}	\label{def:cH}
c_H(u):=
\begin{cases} 
\min\{ \theta_H(u), \frac12u^{2/|V|}\} & \text{ if $H$ is regular},\\
\theta_H(u)&	\text{ if $H$ is irregular}.
\end{cases}
\end{equation}
\end{theorem}
Combining Theorems \ref{thm:hom} and \ref{thm:BGLZ} yields the asymptotic formula
\begin{equation}	\label{thms.combined}
\frac{\RR_{\Verts,p}(H,u) }{\Verts^2p^{\Delta(H)}\log(1/p)}\longrightarrow c_H(u), \qquad \Verts^{-1}\log \Verts\ll p^{2\Delta_\star(H)}\ll1
\end{equation}
for any fixed, connected graph $H$ with $\Delta(H)\ge2$ and fixed $u>0$. The next corollary details the improved result for the case of cycles.
For $H=\Cyc_\ell$ the functions $c_\ell(u):=c_{\Cyc_\ell}(u)$ from \eqref{def:cH} can be computed using the recursion
\begin{equation}	\label{PCell}
\Pol_{\Cyc_2}(x) = 1+ 2x, \quad \Pol_{\Cyc_3}(x) = 1+3x, \quad \Pol_{\Cyc_\ell}(x) = \Pol_{\Cyc_{\ell-1}}(x) + x
\Pol_{\Cyc_{\ell-2}}(x), \quad \ell \ge 4\,.
\end{equation}
For instance, we have
\[
c_3(u) =\begin{cases} \frac13u & u\le 27/8\\ \frac12 u^{2/3} & u\ge 27/8\end{cases}, 
\qquad 
c_4(u) = \begin{cases} -1+ \sqrt{1+ \frac12u} & u\le 16\\  \frac12\sqrt{u} & u\ge 16.\end{cases}
\]
From Theorems \ref{thm:cycles} and \ref{thm:BGLZ} we have the following.
\begin{cor}[Upper tail for cycle counts]
\label{cor:cycles.ut}
Fix an integer $\ell\ge 3$ and let $0<p\ll1$ be as in \eqref{prange.cyc.upper}.
For any fixed $u>0$, 
\begin{equation}	\label{pcyc.upper}
\RR_{\Verts,p}(\Cyc_\ell,u) = (c_\ell(u)+o(1)) \Verts^2p^2\log(1/p),
\end{equation}
where 
\begin{equation}
c_\ell(u) = \min\bigg\{ \theta_\ell(u), \frac12 u^{2/\ell}\bigg\}
\end{equation}
and $\theta_\ell(u)$ is the unique $\theta>0$
such that
 $\Pol_{\Cyc_\ell}(\theta) = 1+u$ for $\Pol_{\Cyc_\ell}(\cdot)$ of \eqref{PCell}.
\end{cor}

\begin{proof}
The lower bound on $\RR(\cdot)$
in \eqref{pcyc.upper} is an immediate consequence of Theorems \ref{thm:cycles} and \ref{thm:BGLZ}. 
The matching upper bound is established similarly to \eqref{R3.upper} by consideration of the events Clique($a$) and Hub($b$) for appropriate $a=a'(\ell)
u^{1/\ell} $ and $b=b'(\ell)u $, essentially following the lines of the proof of \cite[Proposition 2.4]{BGLZ16}.
\end{proof}

We note that for $\Verts^{-1}\ll p\ll \Verts^{-1/2}$ the upper tail no longer has the form of the right hand side of \eqref{pcyc.upper} -- 
see \cite[Section 1.3]{BGLZ16} for further discussion of this.
(In particular, the event Hub in \eqref{Hub} is no longer viable in this regime of sparsity.) The variational problem \eqref{def:phi} was solved in the regime $\Verts^{-2/\Delta}\ll p\ll \Verts^{-1/\Delta}$ for the case of regular graphs in \cite{LuZh14, BGLZ16}.
For $p \ll \Verts^{-1/\Delta}$ and general $H$, even the order of $\RR_{\Verts,p}(H,u)$ up to constants depending only on $u$ has not been completely settled. Indeed, the conjectured dependence on $H,\Verts,p$ from 
\cite{DeKa12:cliques} has recently been refuted in certain cases, see \cite{SiWa18}
and the references therein on the rich history of this problem.

\begin{remark}[Improvements for the case of regular graphs]
\label{rmk:HMS1}
Since the posting of \cite{Augeri18} and the present work to arXiv, subsequent works \cite{HMS19} and 
its refinement in \cite{BaBa20},
have obtained the asymptotic \eqref{thms.combined} for the upper tail of subgraph counts $\sub(H,\Gnp)$ rather than homomorphism counts, for the optimal range 
$\Verts^{-1}(\log\Verts)^{\frac1{\verts(H)-2}} \ll p^{\Delta/2}\ll 1$ 
(excluding 
$\Verts p^\Delta = \Theta(1)$),
in the case that $H$ is a connected \emph{regular} graph with $\Delta \ge 2$ (a different asymptotic 
was obtained in \cite{HMS19} for such $H$ in the \emph{Poisson regime} 
$\Verts^{-1} \ll p^{\Delta/2}\ll \Verts^{-1}(\log\Verts)^{\frac1{\verts(H)-2}}$).
We note that \eqref{thms.combined} still gives the best range of $p$ for general $H$. 

Roughly speaking, the improved sparsity range obtained in \cite{HMS19,BaBa20} comes from an efficient covering of 
the collection of graphs $G$ for which $\hom(H,G) \ge (1+u) \Verts^{\verts(H)} p^{\edges(H)}$ with simple events, which they achieve for regular $H$ by technical arguments involving many non-trivial facts from graph theory. 
By contrast, our approach is closer in spirit to the regularity method, utilizing efficient coverings for \emph{all}  
graphs 
(see \Cref{sec:approach} for further discussion). 
As a result, our results apply to general $H$, and in fact to any functional of $\Gnp$ that is sufficiently continuous (in a quantitative sense) in an appropriate topology  -- in particular we cover non-polynomial functions such as Schatten norms and top eigenvalues of the adjacency matrix -- whereas the approach of \cite{HMS19} appears to be limited to low-degree polynomials of Bernoulli variables.
On the other hand, while we can allow $\Gnp$ to be much sparser than the classic regularity method can handle, our approach generally leads to sub-optimal ranges of $p$.
See Remark \ref{rmk:HMS2} for further comparison of the methods introduced in \cite{HMS19} with our approach.
\end{remark}

\subsection{Results for Schatten (and operator) norms}
\label{sec:Schatten1}

Denote by $\Anp=A_{\Gnp} \in \cA_\Verts$ the (random) adjacency matrix for $\Gnp\sim G(\Verts,p)$ and recall the Schatten norms 
\begin{equation}	\label{def:Schatten}
\|X\|_{S_{\alpha}} = \Big( \sum_{j=1}^\Verts |\lambda_j(X)|^\alpha \Big)^{1/\alpha}, \quad \alpha\in [1,\infty] \,,
\end{equation}
defined in terms of the eigenvalues of $X$. Clearly,
$\hom(\Cyc_{2\ell},X) = \|X\|_{S_{2\ell}}^{2\ell}$, so
Theorem \ref{thm:cycles} gives large deviations bounds for 
the Schatten norms $\|\Anp\|_{S_\ell}$ of even order $\ell \ge 4$.
An inspection of the proof of Theorem \ref{thm:cycles} reveals that 
with slight modifications our argument applies also
to Schatten norms of any order above two, yielding our next result.

\begin{prop}[Large deviations, Schatten norms]
\label{prop:Schatten}
Replacing $(\Verts p)^{-\ell} \hom(\Cyc_\ell, \Gnp)$ 
in \eqref{def:Rate}
and \eqref{def:LL} by $(\Verts p)^{-1} \|\Anp\|_{S_\alpha}$,
the conclusion of Theorem \ref{thm:cycles} holds with $\ell \in \N$ 
changed to 
$\alpha \in (2,\infty)$,
while $\upphi_{\Verts,p}(\Cyc_\ell, s-1)$ and $\uppsi_{\Verts,p}(\Cyc_\ell, 1-s)$
replaced by $\phigen_p(\|\cdot\|_{S_\alpha},\Verts q)$ and 
$\psigen_p(\|\cdot\|_{S_\alpha},  \Verts q)$,
respectively, for $s=q/p$ fixed.
\end{prop}

\begin{remark} Theorem \ref{thm:Schatten.lower} below
dramatically improves the range of $p$ for the
lower tail in Proposition \ref{prop:Schatten}.
As for the upper tail, while $\|\Anp\|_{S_2}
$ 
reduces to tail estimates for the binomial distribution, note that 
$\E \|\Anp\|_{S_2} \asymp \sqrt{\e\|\Anp\|_{S_2}^2} =\sqrt{\Verts(\Verts-1)p} \gg \Verts p$ and the upper tail exponential decay rate is then $\Verts^2 p$ 
(unlike for $\alpha > 2$). It is also easy
to check that 
$\E \|\Anp\|_{S_\alpha} \ggs 
\Verts^{1/\alpha} \sqrt{\Verts p} \gg \Verts p$ whenever
$\Verts^{-1} \ll p \ll \Verts^{2/\alpha-1}$, with  
the upper tail large deviations of $\|\Anp\|_{S_\alpha}$ exhibiting a qualitative 
transition as $p$ crosses $\Verts^{2/\alpha-1+o(1)}$.
\end{remark}

In \Cref{cor:cycles.ut} the matching lower bound for the upper tail of cycle counts is due to the asymptotic solution of the variational problem $\upphi_{\Verts,p}(\Cyc_\ell, u)$ provided by \Cref{thm:BGLZ}. Whereas the analogous result for $\phigen_p(\|\cdot\|_{S_\alpha}, \Verts q)$ is lacking,
we do get such matching bounds for
$\alpha=\infty$, namely for the upper tail of
the Perron--Frobenius eigenvalue $\lambda_1(\Anp)=\|\Anp\|_\op$
and further bound the upper tail decay for
$\lambda_2(\Anp)$ -- the eigenvalue of second-largest modulus.
\begin{prop}
\label{thm:lam12}
For $\Verts^{-1/2}\ll p\le 1/2$ and fixed $s=q/p>1$,
\begin{equation}	\label{PF.tail}
- \log\pr(\|\Anp\|_\op \ge \Verts q) = (1+o(1)) \phigen_p\big(\|\cdot\|_\op\,, (1+o(1))\Verts q\big).
\end{equation}
Moreover, for $\Verts^{-1}\log \Verts\lls p\le 1/2$ and any
$t\gg \sqrt{\Verts}$,
\begin{align}	\label{lam2.tail}
- \log \pr( \lambda_2(\Anp) \ge t ) &\ge 
- \log \pr\big( \| \Anp - p \1\1^\tran \|_\op \ge t \big) 
\nonumber \\ &\quad 
=  (1+o(1))  \phigen_p\big(\|\cdot\, - p \1\1^\tran \|_\op, \,
t+o(t) \big) \,.
\end{align}
\end{prop}

\begin{remark} The upper bounds on the \abbr{lhs} of
\eqref{PF.tail} and \eqref{lam2.tail}, hold  
up to $p \ggs \Verts^{-1}\log\Verts$ and $t \ggs C\sqrt{\Verts p}$, respectively.
By eigenvalue interlacing 
$\| \Anp - p \1\1^\tran \|_\op \ge \lambda_2(\Anp)$, 
trivially yielding the inequality in \eqref{lam2.tail}, where 
one may further replace $p \1 \1^\tran$ by
$\e \Anp = p (\1\1^\tran -\id)$.
\end{remark}

\begin{remark} In \cite{GuHu18}, Guionnet and Husson establish
a large deviations principle for the largest eigenvalue of 
$\Verts$-dimensional Wigner matrices, 
re-scaled by $\sqrt{\Verts}$, 
whose independent, standardized entries have uniformly 
sub-Gaussian \abbr{mgf}-s
(allowing for Rademacher entries). However, such uniform 
sub-Gaussian domination does not apply to
$\Anp-\e \Anp$ when $p=o(1)$. 
Indeed, \cite{GuHu18} concerns deviations of the largest eigenvalue at the scale $\sqrt{\Verts}$ of the bulk spectral distribution, whereas \eqref{lam2.tail} is about larger deviations (we expect
\eqref{lam2.tail} to fail for $t\asymp \sqrt{\Verts p}$).
\end{remark}

Motivated by the posting of this work on arXiv,  
Bhattacharya and Ganguly \cite{BhGa18} compute the asymptotics of the 
right hand sides of \eqref{PF.tail} and \eqref{lam2.tail}, showing that for $\Verts^{-1/2}\ll p\ll 1$, 
\begin{equation}	\label{BhGa1}
\phigen_p\big(\|\cdot\|_\op\,, \Verts q\big) \sim 
\min\Big\{ \frac12, \, 1-\frac{p}{q} \Big\}\Verts^2 q^2\log(1/p), \qquad 
s=q/p>1 \text{ fixed},
\end{equation}
and
\begin{equation}	\label{BhGa2}
\phigen_p\big( \|\cdot\,-p\1\1^\tran\|_\op, \, \Verts q \big) \sim \frac12 \Verts^2 q^2\log(1/p), \qquad  \Verts^{-1/2} \ll q \ls p\,.
\end{equation}
In particular, the probability for deviations of $\|\Anp\|_\op$ is controlled, up to sub-exponential factors, by the events Clique and Hub from \eqref{Clique}, \eqref{Hub} (as reflected by the two expressions in the minimum in \eqref{BhGa1}).
On the other hand, for deviations of $\|\Anp - p\1\1^\tran\|_\op$ only the clique construction contributes to leading order.
Together with Proposition \ref{thm:lam12}, this completely solves the large deviation problem for the norms of $\Anp$ and $\Anp-p\1\1^\tran$ for these regimes of 
$p,q$.

\subsection{Sharp lower tails for homomorphism counts and Schatten norms} \label{sec:lowertails}

In this subsection we consider two families of functions of $\Gnp$ for which we can obtain the sharp lower tail for a wide range of $p$. Moreover, we provide an explicit formula for the tail, which in both cases is asymptotically given by ${\Verts\choose2}I_p(q)$, the relative entropy of the distribution $G(\Verts,q)$ with respect to $G(\Verts,p)$, for an appropriate $q<p$. In particular, the rate matches the log-probability 
that the edge density is uniformly lowered from $p$ to $q$.
This contrasts with our ``broken symmetry" results for the upper tail, where the rate asymptotically matches the log-probability for a 
small planted structure (see \eqref{Clique} and \eqref{Hub}).

For our first such result we recall some notation from graph limit theory. 
Consider
the space $\cW$ of all bounded symmetric measurable functions $f:[0,1]^2\to \R$, 
and for a simple graph $H=(V,E)$ define the associated \emph{homomorphism density} functional
\begin{equation}	\label{def:tH}
t(H,\cdot): \cW\to \R, \qquad t(H,f) := \int_{[0,1]^V}\prod_{kl\in E} f(x_k, x_l) \prod_{k\in V} dx_k \,.
\end{equation}
This extends to $\cW$
the homomorphism counting functionals \eqref{def:homX}.
Indeed, associating to each 
$X\in \cX_\Verts$ the element
$f_X\in \cW$ with 
\begin{equation}	\label{fX}
f_X(x,y) = X_{\lceil \Verts x\rceil, \lceil \Verts y\rceil}
\end{equation}
it follows that 
$t(H,f_X) = \Verts^{-|V|} \hom(H, X)$.
A simple graph $H=(V,E)$ is 
\emph{Sidorenko} 
if
\begin{equation}	\label{Sidorenko}
t(H,f) \ge t(K_2,f)^{|E|}\,,   \quad \qquad \forall f \in \cW \, , f \ge 0 \, .
\end{equation}
It was conjectured by Erd\H{o}s and Simonovits \cite{Simonovits84} and Sidorenko \cite{Sidorenko} that all bipartite graphs are Sidorenko.
While the conjecture remains open as of this writing, \eqref{Sidorenko} has been established for complete bipartite graphs, trees and even cycles \cite{Sidorenko}, hyper-cubes \cite{Hatami10} and bipartite graphs with a vertex complete to the other side \cite{CFS:Sidorenko}, among others; see the recent works \cite{Szegedy:Sidorenko,CKLL} and references therein for further results.
In the following theorem we provide a lower bound for the lower tail of $\hom(H,\Gnp)$, valid for \emph{any} simple graph $H$, and show that this bound is tight if $H$ is Sidorenko. 
In particular, conditional on Sidorenko's conjecture, \eqref{Sido:lower} and \eqref{Sido:upper} provide the sharp lower tail for homomorphism counts of any bipartite graph.

\begin{theorem}[Lower tail, Sidorenko graphs] 
\label{thm:Sidorenko}
Let $H=(V,E)$ be a finite, simple, graph. 
If $\Verts^{-1/(2\Delta_\star-1)}\ll p\le 1/2$ 
(for $\Delta_\star$ as in Theorem \ref{thm:hom}), 
then 
fixing $q/p \in (0,1)$ and setting ${\widehat q}:=q-q/\Verts$, we have when $\Verts \to \infty$, 
\begin{equation}	\label{Sido:lower}
\pro{ \hom(H,\Gnp) \le {\widehat q}^{~ \edges(H)} \Verts^{\verts(H)}} \ge 
e^{ - (1+o(1)) {\Verts\choose2} I_p(q)} \,. 
\end{equation}
Moreover, if $H$ is Sidorenko, then for
any $0<q<p<1$ and $\Verts \in \N$, 
\begin{equation}	\label{Sido:upper}
\pro{ \hom(H,\Gnp) \le {\widehat q}^{~ \edges(H)} \Verts^{\verts(H)} } \le 
e^{-{\Verts\choose2} I_p(q)} \,.
\end{equation}

\end{theorem}

\begin{remark}
We stress that the upper bound \eqref{Sido:upper} is \emph{non-asymptotic}, applying for 
any fixed $\Verts$ and $0<q<p<1$; thus, if $\Verts$ is an asymptotic 
parameter then $p$ and $q$ can depend in an arbitrary way on $\Verts$. The same goes for \eqref{Schatten.upper} below.
\end{remark}

\begin{remark}
Such bounds for Sidorenko graphs $H$ are derived for the regime of fixed $0< q < p < 1$ 
in \cite{LuZh12}, and in \cite{Zhao:lower} for general $H$,
when $\Verts^{-a_H}\le p\ll1$ and $\bar{s}_H<q/p<1$ for some $\bar{s}_H\in (0,1)$ 
and an extremely 
small $a_H>0$. Moreover, \cite[Prop.\ 3.5]{LuZh12} shows that conditional on 
the event $\{\hom(H,\Gnp)\le q^{\edges(H)} \Verts^{\verts(H)}\}$, the corresponding graphon $f_\Anp$ is close in cut-norm to the constant $q\in \cW$.
\end{remark}

\begin{remark}
Previous works considered the lower tail for \emph{subgraph counts} $\sub(H,\Gnp)$.
For $p\gg \Verts^{-1/\Delta(H)}$, with high probability $\sub(H,\Gnp)$ and $\hom(H,\Gnp)$ differ by a non-random, fixed
factor ($\sim \aut(H)$, the number of graph automorphisms of $H$). 
In contrast, $\sub(H,\Gnp)$ and $\hom(H,\Gnp)$ have substantially different behavior for smaller $p$.
For general $H$, \cite{JaWa} obtains upper and lower bounds for the lower tail $\log\pr(\sub(H,\Gnp)\le (1-\eps)\e\sub(H,\Gnp))$ matching up to a constant factor, whereas Theorem \ref{thm:Sidorenko} obtains the sharp lower tail (with asymptotically matching upper and lower bounds) for Sidorenko graphs. 
Such a sharp lower tail is obtained in \cite[Theorem 3]{JaWa}  
for a wide class of graphs $H$ including \emph{2-balanced} graphs, 
but only in a regime of sufficiently small $p=p(\Verts)$ that does not
overlap with Theorem \ref{thm:Sidorenko}.
\end{remark}

For convex functions of $\Anp$, such as the Schatten norms of Proposition \ref{prop:Schatten}, we can obtain strong results for the lower tail
via the following special case of Proposition \ref{prop:cvx}.

\begin{prop}
\label{prop:lower-cvx}
Fix $\Verts\in \N$, $h:\cX_\Verts\to \R$ and $p\in (0,1)$.
If $t\in \R$ is such that the sub-level set $\{X\in \cX_\Verts: h(X)\le t\}$ is 
convex, then
\begin{equation}\label{non-asymp-convex}
\pro{ h(\Anp)\le t} \le \expo{ - \psigen_p(h,t)} \,.
\end{equation}
\end{prop}


Proposition \ref{prop:lower-cvx} applies to any semi-norm of $\Anp$. Here,
we consider the lower tail for Schatten norms $\|\Anp\|_{S_\alpha}$, showing  
in particular that the leading order is 
the same for all $\alpha\in (2,\infty]$ and $\Verts^{\frac{2}{\alpha}-1}\ll p\le 1/2$ (for smaller $p$ there may be
slack in \eqref{Schatten.upper}).

\begin{theorem}[Lower tail, Schatten norms]
\label{thm:Schatten.lower}
For $0<q<p<1$, $\alpha \in [1,\infty]$ and $\Verts\in \N$, 
\begin{equation}	\label{Schatten.upper}
\pro{ \|\Anp\|_{S_\alpha} \le q (\Verts-1)} \le 
e^{-\psigen_p\big( \|\cdot\|_{S_\alpha},  q (\Verts-1)\big)} 
\le 
e^{- {\Verts\choose2} I_p(q) }
\, . 
\end{equation}
Moreover, if $\alpha \in (2,\infty]$ and $p=p(\Verts)$ satisfies 
$1/2 \ge p \gg \Verts^{\frac{2}{\alpha}-1}$ as $\Verts\to \infty$
(taking $p(\Verts) \gg \log \Verts/\Verts$ for $\alpha=\infty$),
then for fixed $s:=q/p\in (0,1)$, we have
\begin{equation}	\label{Schatten.lower}
\pro{ \|\Anp\|_{S_{\alpha}} \le q(\Verts-1)} \ge e^{-(1+o(1)) {\Verts\choose2} I_p(q)} \,. 
\end{equation}
\end{theorem}

\begin{remark} Whereas even-length cycles are Sidorenko \cite{Sidorenko},
taking $\alpha=2\ell\in2\N$ in \Cref{thm:Schatten.lower}
improves upon the range $p \gg \Verts^{-1/3}$ required for $H=C_{2\ell}$
in \Cref{thm:Sidorenko}.
\end{remark}

\subsection{Organization of the paper}
In Section \ref{sec:ideas} we briefly overview previous results in the theory of nonlinear large deviations based on low-complexity gradient conditions, and describe our new approach based on covering constructions and continuity.
In Section \ref{sec:approach} we motivate our spectral approach to covering and continuity arguments, starting with a very short proof for the sharp upper tail for triangle counts in the regime $((\log\Verts)/\Verts)^{1/8}\ll p\ll1$ (which already surpasses all previous works).  
We then state our versions of the regularity and counting lemmas for sparse random graphs, 
which we prove in Section \ref{sec:reg-proof}, aided by
the preliminary control on the spectrum of $\Anp$ that we establish
in Section \ref{sec:spec}. 
We apply our regularity and counting lemmas to
prove \Cref{thm:hom} in Section \ref{sec:hom}. In Section \ref{sec:cycles} we prove \Cref{thm:cycles} 
as a direct
consequence of the non-asymptotic version, \Cref{thm:cycles.quant};
the necessary modifications to obtain Proposition \ref{prop:Schatten} are given in Section \ref{sec:Schatten}. 
Lastly, we prove \Cref{thm:lam12} in Section \ref{sec:PF} and
Theorems \ref{thm:Sidorenko} and \ref{thm:Schatten.lower} in Section \ref{sec:lower}.

\subsection{Notation and conventions}\label{sec:notation}

\subsubsection*{Asymptotic notation} $~$
Unless otherwise stated, $C,C',C_o,c,$ etc.\ denote universal constants; if they depend on parameters we indicate this by writing e.g.\ $C_\kappa,C(H)$.
The notations $f=O(g)$, $f\ls g$ and $g\gs f$ are synonymous to
having $|f|\le Cg$ for some universal constant $C$, while 
$f=\Theta(g)$ and $f\asymp g$ mean $f\ls g\ls f$.
We indicate dependence of the implied constant on parameters (such as $H$ or $u$) with subscripts, e.g.\ $f\ls_Hg$. The statements
$f=o(g)$, $g=\omega(f)$, $f\ll g$, $g\gg f$ are synonymous to having
$f/g\to 0$ as $\Verts\to \infty$, where
the rate of convergence may depend on fixed parameters 
such as $H$ and $u$ without being indicated explicitly.
While our results use the qualitative $o(\cdot)$ notation, in the proofs we often give quantitative estimates with more explicit dependence on fixed parameters for the sake of clarity.
We assume throughout that $\Verts\ge 2$ (so that $\log \Verts \gs 1$).

\subsubsection*{Matrices and normed spaces}
We endow $\R^\Verts$ with the $\ell^\Verts_q$ norms $\|\cdot\|_q$, $q\in [1,\infty]$ 
and for $q=2$ the Euclidean inner product $\langle\cdot,\cdot\rangle$, denoting by
$\ball_q(r)$ the corresponding closed balls of radii $r \ge 0$, while
$\sphere^{d-1}$ stands for the unit Euclidean sphere in $\R^d$.
We write $\1=\1_\Verts\in \R^\Verts$ for the all-ones vector and $\id=\id_\Verts$ for the $\Verts\times\Verts$ identity matrix.
For a set $\Omega$ we write
$\Sym_\Verts(\Omega)$ for the set of symmetric $\Verts\times\Verts$ matrices with entries in $\Omega$, and $\Sym_\Verts^0(\Omega)\subset \Sym_\Verts(\Omega)$ for the subset of symmetric matrices with zeros along the diagonal. 
For $1\le \rnk \le \Verts$ we write
$
\Sym_{\Verts,\rnk }(\Omega)\subset \Sym_\Verts(\Omega)
$
for the subset of elements of rank at most $\rnk $. 
We abbreviate
\begin{equation}
\cX_\Verts:= \Sym_\Verts^0([0,1]), \qquad \cA_\Verts:= \Sym_\Verts^0(\{0,1\})
\end{equation}
as these sets will appear frequently.
When invoking Corollary \ref{cor:covering} we implicitly identify the above sets with $[0,1]^{{\Verts\choose2}}$ and $\{0,1\}^{{\Verts\choose2}}$, respectively.
Note that $\cA_\Verts$ is the set of adjacency matrices for simple (and undirected) graphs on $\Verts$ vertices.
Throughout we let $\Anp\in \cA_\Verts$ denote the adjacency matrix of $\Gnp\sim G(\Verts,p)$,
with $\mu_p(\cdot)= \pr(\Anp\in \cdot)$ the corresponding product Bernoulli measure on $\cA_\Verts$. We denote the adjacency matrix for the complete graph on $\Verts$ vertices by
\begin{equation}	\label{def:jay}
\jay= \jay_\Verts := \1\1^\tran - \id_\Verts \in \cA_\Verts.
\end{equation}

We label the eigenvalues of an element $X\in \Sym_\Verts(\R)$ in non-increasing order of modulus:
\begin{equation}	\label{eval.order}
|\lambda_1(X)|\ge |\lambda_2(X)|\ge \cdots\ge |\lambda_\Verts(X)| 
\end{equation}
and recall the Schatten norms on $\Sym_\Verts(\R)$ 
as in \eqref{def:Schatten}.
In particular, the spectral norm $\|X\|_{S_\infty}=|\lambda_1(X)|$ equals
the $\ell_2^\Verts\to \ell_2^\Verts$ operator norm 
\[
\|X\|_\op = \sup_{u\in \sphere^{\Verts-1}}\|Xu\|_2 = \sup_{u\in \sphere^{\Verts-1}} \langle u, Xu\rangle \,.
\]
Moreover, $\|X\|_{S_2}$ equals the Hilbert--Schmidt norm for
the inner product
\[
\langle X,Y\rangle_\HS = \Tr(XY), \qquad \|X\|_\HS = (\Tr X^2)^{1/2} \,,
\]
with the closed Hilbert--Schmidt ball in $\Sym_\Verts(\R)$ 
of radius $t$ denoted by $\ball_\HS(t)$. 
By the non-commutative 
H\"older inequality, whenever $1/\alpha+1/\beta=1/\gamma$,
\begin{equation}	\label{Holder}
\|XY\|_{S_{\gamma}} \le \|X\|_{S_\alpha}\|Y\|_{S_\beta} 
\end{equation}
(see \cite[Theorem 2.8]{Sim05b}), 
and in particular
\begin{equation}	\label{Schatten.op}
\|XY\|_{S_\alpha} \le \|X\|_\op\|Y\|_{S_\alpha}\,.
\end{equation}
For $X\in \Sym_\Verts(\R)$ having spectral decomposition
\[
X= \sum_{j=1}^\Verts \lambda_j u_ju_j^\tran\,,
\]
with eigenvalues arranged as in \eqref{eval.order}, and for any
$1 \le \rnk \le \Verts$, we further have that
\begin{equation}	\label{def:cutoff}
X=X_{\le \rnk } + X_{>\rnk }, \qquad
X_{\le \rnk }:= \sum_{j\le \rnk } \lambda_j u_ju_j^\tran, \qquad  X_{>\rnk } 
:= \sum_{j>\rnk } \lambda_ju_ju_j^\tran \,.
\end{equation}

\subsubsection*{Graph theory}
All graphs are assumed to be simple (without self-loops or multiple edges) and finite. 
For a graph $H=(V,E)$ we write $V(H)=V$, $E(H)=E$, $\verts(H)=|V|,$ and $\edges(H)=|E|$. 
We say that a graph $H$ is \emph{nonempty} if $E(H)\ne \emptyset$.
For $v\in V(H)$, $\deg_H(v)$ denotes the degree of $v$, and
$\Delta(H):=\max_{v\in V(H)} \{ \deg_H(v) \}$
denotes the maximum degree of $H$. We often take $V=[\verts]$. We use
$F\le H$ to mean that $F$ is a subgraph of $H$ (obtained by removing 
some of the vertices and/or edges of $H$). We further write
$F \preccurlyeq  H$
when $F$ is an induced subgraph of $H$ (i.e.\ $F=H[V']$ for some 
$V'\subseteq V(H)$), and 
$F \prec H$ if $F \preccurlyeq H$ 
and $F \ne H$.

\section{Relation to previous works and new ideas}
\label{sec:ideas}

Previous work on 
nonlinear large deviations 
focused on 
approximating the partition function $Z$ for Gibbs measures 
on the Hamming cube.
Specifically, given a Hamiltonian $f:\{0,1\}^d\to \R$ with associated Gibbs measure $\mu$  
of
density 
$Z^{-1} e^{f(\cdot)}$ on $\{0,1\}^d$
the aim is to approximate 
\[
Z = \sum_{x\in\{0,1\}^d} \exp(f(x)).
\]
This generalizes the problem of determining the large deviations 
of a function $h$ of a vector $\bx\in \{0,1\}^d$ with i.i.d.\ Bernoulli($p$) components, i.e.\ of approximating 
\begin{equation}\label{eq:tail-apx}
\log \pr( h(\bs{x}) \ge t\e h(\bs{x}) ) \,,
\end{equation}
which corresponds to $\log Z$ for
\begin{equation}\label{eq:f-of-h}
f_h (x) := g(h(x)) + d\log(1-p) + \sum_{i=1}^d x_i \log\frac{p}{1-p} \,,
\end{equation}
where $g(s)\equiv0$ for $s\ge t\e h(\bs{x})$ and $g(s)\equiv-\infty$ for $s<t\e h(\bs{x})$.

The Gibbs variational principle frames the log-partition function (or Helmholtz free energy) as the solution to a variational problem:
\[
\log Z =  \sup_{\nu\in\Meas_1(\{0,1\}^d)} \Big\{ \sum_x f(x)\nu(x) - \sum_x \nu(x) \log \nu(x) \Big\},
\]
where the supremum ranges over all probability measures on the cube.
The feasible region for optimization has dimension exponential in $d$; to reduce dimensionality it is common practice in physics to invoke the \emph{naive mean field approximation}, which is to restrict $\nu$ to range over product measures. When the Hamiltonian has the separable form $f(x)=f_1(x_1)+\cdots + f_d(x_d)$, 
the naive mean field approximation is an exact identity. 

The main idea introduced in \cite{ChDe14} 
(see also the survey \cite{Chatterjee:survey})
is that the naive mean field approximation can be rigorously justified when $f$ has \emph{low-complexity gradient}, meaning that the image of $\nabla f$ can be efficiently approximated using a net (in particular when $f$ is affine, 
so that $\mu$ is a product measure, the image of $\nabla f$ is a single point).
By ``efficient" we mean that the metric entropy of the image of the gradient is small in comparison with the free energy $\log Z$. 
This idea was further developed by Eldan in \cite{Eldan-NMF}
where the complexity of the gradient is quantified 
in terms of the Gaussian width of its image rather than covering numbers.
In addition, he showed a low complexity gradient yields 
an approximation of the Gibbs measure $\mu$ by a 
mixture of tilted measures, each of which is close to a product measure
(see also more recent works \cite{ElGr-decomp,Austin-NMF}).
However, when $h(x)$ stands for subgraph counts, the leading term $\RR(\cdot)$ 
decreases as $p=p(\Verts) \to 0$ and while this approach is relatively general, 
between the required smooth approximation of $g$
we must employ in \eqref{eq:f-of-h},
and the move from $\nabla f_h$ to $f_h$, the ability to recover the 
optimal range of $p(\Verts)$ is completely lost.

To overcome this deficiency we take here a different approach,
better tuned to yield sharper results in specific applications.
As in \cite{ChDe14,Eldan-NMF}, our approach involves a notion of low complexity, now working directly with \eqref{eq:tail-apx} using
nets to approximate the values of the function $h$ rather than 
$\nabla f_h$.
Specifically, we construct efficient coverings of the cube $\{0,1\}^{{\Verts\choose2}}$ (identified with the space of $\Verts\times\Verts$ adjacency matrices in the natural way) by convex bodies $\cB_i$ on which the function $h$ is nearly constant. 
This can be regarded as a quantitative version of the approach from \cite{ChVa11}, which relied on the compactness of the space of graphons; the coverings we construct in Sections \ref{sec:cycles}--\ref{sec:hom} using spectral arguments quantify the compactness of the space of adjacency matrices (see 
Section \ref{sec:approach}
for further discussion of these ideas).

Let $h:[0,1]^d\to \R$, $t\in \R$, and recall the notation \eqref{def:IpT}--\eqref{def:phigen}. 
Our aim is to show that for
the product Bernoulli($p$) measure $\mu_p$ on $\{0,1\}^d$,
\begin{equation}\label{eq:var-bd}
\mu_p( \cLc) \le \expo{ - \phigen_{p}(h,t) + {\rm Error} },
\end{equation}
where {\rm Error} is of lower order than the main term $\phigen_{p}(h,t)$,
and similarly for the lower tail (namely, for 
$\mu_p(\subL h t)$).
It is well known that \eqref{eq:var-bd} holds with ${\rm Error}=0$ 
whenever $h$ is an affine function, namely, for half-spaces $\subL h t$
(hence the exactness of the naive mean field approximation for the associated tilted measure). 
As stated next, thanks to the convexity of $I_p(\cdot)$, such a zero-error, non-asymptotic bound applies for
any closed convex set.
\begin{prop}
\label{prop:cvx}
For $p \in [0,1]$ and closed convex $\cK\subseteq \R^d$,
\[
\mu_p(\cK) \le \expo{ -I_p(\cK)}\,.
\]
\end{prop}
\begin{proof} With $\mu_p(\cdot)$ supported on the convex, compact set $[0,1]^d$,
outside of which $I_p(x)=\infty$, it suffices to consider compact and convex $\cK$.
Moreover, setting $\Lambda_p(\theta) = \log \sum_x e^{\langle \theta,x\rangle} \mu_p(x)$, 
by Markov's inequality
\[
\log \mu_p(\cK) \le 
-\inf_{y \in \cK} 
\langle \theta,y\rangle + \Lambda_p(\theta) \,.
\]
Taking the infimum over $\theta \in \R^d$ and multiplying both sides by minus one, leads to 
\begin{equation}\label{eq:sion}
-\log \mu_p(\cK) \ge \sup_{\theta \in \R^d} \inf_{y \in \cK} \big\{
\langle \theta,y \rangle - \Lambda_p(\theta) \big\} \,.
\end{equation}
By convexity of $\Lambda_p(\cdot)$, the lower-semi-continuous and convex  
functional $y \mapsto \{ \langle \theta,y \rangle - \Lambda_p(\theta) \}$ 
is also concave in $\theta$. Hence, by Sion's min-max theorem, we may change in 
the \abbr{rhs} of \eqref{eq:sion} the order of the supremum over $\theta$ in a
vector space $\R^d$ and the infimum over  $y$ in a compact, convex set $\cK$
(cf.\ \cite[Thm. 4.2']{Sio58}). Consequently,
\[
-\log \mu_p(\cK) \ge \inf_{y \in \cK} \{\Lambda_p^\star(y)\} \,, \qquad 
\Lambda_p^\star(y) = \sup_{\theta \in \R^d}  \big\{
\langle \theta,y \rangle - \Lambda_p(\theta) \big\} \,.
\]
For any product measure $\mu_p(\cdot)$, the functional $\Lambda_p(\theta)$ is of the form 
$\sum_{i} \widehat{\Lambda}_p(\theta_i)$, in which case clearly $\Lambda^\star_p(y) = \sum_{i} 
\widehat{\Lambda}^\star(y_i)$, with $\widehat{\Lambda}_p$ and $\widehat{\Lambda}^\star_p$ corresponding
to the case $d=1$. Finally, for $\mu_p(1)=p=1-\mu_p(0)$ and $d=1$, it is a simple calculus 
exercise to verify that  $\widehat{\Lambda}^\star_p=I_p$.
\end{proof}

The above 
does not at first appear to be useful for proving Theorems \ref{thm:cycles} and \ref{thm:hom}, since super-level
sets for homomorphism counting functions are non-convex (except for the trivial edge-counting function). 
For example, in the case $h=\hom(\Cyc_\ell, \,\cdot)$ for $\ell$ even, 
$\cLc$ is the \emph{complement} of a convex set.
However, our key observation is that such super-level sets can
be efficiently covered by convex sets on which $h$ has small fluctuations, 
thereby utilizing the following easy consequence 
of Proposition \ref{prop:cvx} and the union bound.

\begin{cor}
\label{cor:covering}
Let $h:[0,1]^d\to \R$. 
Suppose there is a finite family $\{\cB_i\}_{i\in \cI}$ of closed convex sets in $\R^d$, an``exceptional" set $\cE\subset \{0,1\}^d$, and $\delta>0$ such that 
\begin{equation} \label{ass:contain}
\{0,1\}^d\setminus \cE  \subseteq \bigcup_{i\in \cI} \cB_i \, 
\end{equation}
and 
\begin{equation}	\label{ass:fluct}
\forall i\in\cI, \;\forall x,y\in\cB_i,\qquad h(y)-h(x)  \le \delta.
\end{equation} 
Then, for any $p \in [0,1]$ and $t\in \R$,
\begin{align}	\label{cube.upper}
\mu_p\big( \cLc \big)
 &\le |\cI| \expo{ - \phigen_{p}(h,t-\delta)}  + \mu_p(\cE)\,,
\\
\mu_p\big(\cLl \big)
 &\le |\cI| \expo{ - \psigen_{p}(h,t+\delta)}  +\mu_p(\cE)\,.
\label{cube.lower}
\end{align}
\end{cor}

\begin{proof} Denoting by
$\cI'\subset \cI$ the set of $i$ for which $(\cB_i \setminus \cE) \cap \cLc \ne \emptyset$,
we have by the union bound, followed by Proposition \ref{prop:cvx} that
\begin{align*}	
\mu_p (\cLc) &\le \mu_p(\cE) + \sum_{i \in \cI'} \mu_p(\cB_i) 
\le \mu_p(\cE) + \sum_{i \in \cI'} e^{-I_p(\cB_i)} \\
& \le |\cI'| \exp\big\{ -\min_{i\in \cI'} I_p(\cB_i)\big\} + \mu_p(\cE) \,.
\end{align*}
From \eqref{ass:fluct} it follows that
$\cB_i \cap[0,1]^d \subseteq \cLcd$ 
for any $i \in \cI'$. Hence,
\[
\min_{i \in \cI'} \{ I_p\big(\cB_i \big) \} 
\ge I_p(\cLcd) = \phigen_{p}(h,t-\delta)
\]
and \eqref{cube.upper} follows. The same line of reasoning yields
also \eqref{cube.lower}.
\end{proof}
\begin{remark} Our proof shows that it suffices for \eqref{cube.upper} 
to have \eqref{ass:fluct} for $x \in \cB_i$ and $y \in \cB_i \setminus \cE$,
whereas for \eqref{cube.lower} it suffices to have \eqref{ass:fluct} for 
$y \in \cB_i$ and $x \in \cB_i \setminus \cE$.
\end{remark}

Proposition \ref{prop:cvx} and Corollary \ref{cor:covering} 
yield only \emph{upper} bounds on tail probabilities.
In some cases tilting arguments can show that these bounds are sharp. 
For our main application 
even this is unnecessary since the 
variational problem \eqref{def:phi} was solved in \cite{LuZh14, BGLZ16}, with sharpness directly verified by considering the events Clique and Hub from \eqref{Clique}, \eqref{Hub}, respectively. 

Whereas 
Corollary \ref{cor:covering} is rather elementary, the \emph{real} technical challenge is in the design of coverings $\{\cB_i\}_{i\in \cI}$ for the space of $\Verts\times\Verts$ adjacency matrices (up to well-chosen exceptional sets) that are efficient enough to allow the sparsity parameter $p=p(\Verts)$ to decay as quickly as in our stated results. 
We obtain suitable coverings via quantitative versions of the regularity and counting lemmas, stated in Section \ref{sec:regularity}.
To establish these, we employ techniques of high-dimensional geometry and spectral analysis.
As detailed in Subsection \ref{sec:regularity}, our approach to the regularity lemma for sparse random graphs (Proposition \ref{prop:RLRG}) is a quantitative strengthening of a well-known spectral proof of the classic regularity lemma \cite{FrKa99,Szegedy11} (see also \cite{Tao-blog-regularity-lemma}).
As a key intermediate step we obtain upper tail bounds for ``outlier" eigenvalues of $\Anp$, i.e.\ eigenvalues of size $\sqrt{\Verts p}\ll |\lambda_j(\Anp)|\ls \Verts p$, which might be of independent interest.
For a slowly growing parameter $\rnk (\Verts)$ this allows us to approximate $\Anp$ by its rank-$\rnk $ projection, which in turn can be approximated by a point $Y\in \cX_\Verts$ in a net of size $O(\rnk \Verts\log\Verts)$. This provides an efficient covering by spectral-norm balls (up to an exceptional event containing matrices with many large outlier eigenvalues). 

For the case that $H$ is a cycle we can take a more refined approach. In particular, we take advantage of the approximate orthogonality of the images of the rank-$\rnk $ approximation $Y$ for $\Anp$ and of the residual matrix $\Anp-Y$ to get improved control on the fluctuation of $\hom(\Cyc_\ell,\,\cdot)$ on a convex body $\cB_Y$ that is specially designed to exploit this orthogonality. See 
Section \ref{sec:ref-approach} for further discussion of these ideas. In particular, 
the probabilistic parts of our arguments are confined to 
\Cref{lem:HSk}, which is used to control the exceptional set $\cE$ in our applications of \Cref{cor:covering}
(and to Lemmas \ref{lem:LB} and 
\ref{lem:HP-lb} for converses of \Cref{prop:cvx} that we utilize for matching 
lower bounds in \Cref{thm:lam12}, and in
Theorems \ref{thm:Sidorenko}, \ref{thm:Schatten.lower}, respectively).


\section{Spectral regularity method for random graphs}
\label{sec:approach}

To establish Theorems \ref{thm:hom} and \ref{thm:cycles} via 
\Cref{cor:covering}, we need to find a covering of ``most" of 
$\cA_\Verts$ 
by convex sets on which the functions $\hom(H, \cdot)$ have small fluctuations. 
In effect, our approach is a quantitative refinement of the argument in \cite{ChVa11}, which  
uses the topological space of graphons with the cut metric to obtain 
such coverings in the dense setting.
In this section, we first motivate our spectral approach to covering constructions, and how it can be optimized towards Theorem \ref{thm:cycles} for cycle counts. 
Then, in Subsection \ref{sec:regularity}, we make the connection with graphon methods more precise by stating quantitative versions of the regularity and counting lemmas tailored for applications to sparse random graphs; along with Corollary \ref{cor:covering}, these are the key ingredients for establishing Theorem \ref{thm:hom}.

\subsection{A simple argument for triangle counts}
\label{sec:C3}

We begin with a short, crude version of our argument
for the normalized homomorphism counting function
\[
h_3(X) := (\Verts p)^{-3} \hom(\Cyc_3, X) \,. 
\]
It yields the upper tail \eqref{eq:upper-tail.cyc} for 
$\ell=3$ and $ ((\log\Verts)/\Verts)^{1/8}\ll p\ll1 $, and motivates 
the derivation of refined estimates on the spectrum of $\Anp$ 
in Section \ref{sec:spec}. Specifically, observe that 
with eigenvalues as in \eqref{eval.order} and $1\le m \le \Verts$,
\begin{equation}\label{bd:HS-k}
\|X_{\le m}\|_\HS^2 :=
\sum_{j=1}^m 
\lambda_j^2(X) \ge m \, \lambda_m^2(X) \,.
\end{equation}
Thus, we have for the projection of $X \in \B_\HS(\Verts)$ to  
$X_{\le \rnk } \in \Sym_{\Verts,\rnk }(\R)$, as in \eqref{def:cutoff},
that 
\begin{equation}	\label{Rapprox.basic}
\|X-X_{\le \rnk }\|_\op = \|X_{>\rnk }\|_\op = |\lambda_{\rnk +1}(X)| \le \frac{\Verts}{\sqrt{\rnk +1}}.
\end{equation}

A \emph{$\delta$-net} for a metric space $(E,d)$ is a subset $\cN$ such that
$\sup_x \{ d(x,\cN) \} \le \delta$.
Lemma \ref{lem:net} constructs efficient $\delta$-nets for the collection
$\Sym_{\Verts,\rnk }(\R)$ of all $\Verts$-dimensional symmetric matrices of rank at most $\rnk$, yielding
the following bound on the size of such nets.
\begin{lemma}	\label{lem:net1}
For any $1\le\rnk\le \Verts$ and any $\delta\in (0,1]$, the set $\Sym_{\Verts,\rnk }(\R) \cap \B_\HS(\Verts)$ 
has a $3\Verts\delta$-net $\cN$ in the Hilbert--Schmidt norm of size 
\[
|\cN|\le \expo{ \rnk(\Verts+2)\log(3\Verts/\delta)} .
\]
\end{lemma}
\begin{proof}
Thanks to \eqref{HS.bd} we have the $3 \Verts\delta$-net $\cN = \Mat(\Sigma\times \cV)$  
for $\Sym_{\Verts,\rnk }(\R)\cap \ball_\HS(\Verts)$, in the Hilbert--Schmidt norm.
Further, from \eqref{bd:Vnet} its size is at most
\[
|\cN|\le |\Sigma| \cdot |\cV| \le \expo{ \rnk(2\Verts^2/\delta) + \rnk\Verts \log(3\sqrt{\rnk}/\delta)}
\le \expo{ \rnk(\Verts+2)\log(3\Verts/\delta)}\,,
\]
as stated.
\end{proof}

From \eqref{Rapprox.basic}, \Cref{lem:net1} and the triangle inequality,
for some $1\le \rnk\le \Verts$ to be chosen later
we have a set $\cN\subset \cX_\Verts$ of size $\exp(O(\rnk\Verts\log\Verts))$ consisting of matrices of rank at most $\rnk$
such that for any $X\in \cX_\Verts$ there exists $Y\in \cN$ with
\begin{equation}	\label{h3:XYop}
\|X-Y\|_\op \le \|X_{> \rnk}\|_\op + \|X_{\le \rnk} -Y\|_\HS 
\le \eps\Verts\,,\qquad \text{ where} \quad \eps \asymp \rnk^{-1/2}.
\end{equation}
This is the key fact behind the
quantitative covering of \cite{ChDe14}; 
incidentally, it also underlies a well-known 
spectral proof of the regularity lemma 
\cite{FrKa99,Szegedy11,Tao-blog-regularity-lemma}.
Note that whereas in \cite{ChDe14} such a net is used to approximate the \emph{gradient} of the functions $\hom(H, \cdot)$, here
we use nets to approximate the values of the functions themselves.

To each $Y\in \cN$ we associate the closed, convex 
set $\cB_Y =\{ X\in \cX_\Verts: \|X-Y\|_\op\le \eps\Verts\}$. 
By Weyl's inequality, 
upon ordering the eigenvalues of $M_1,M_2\in \Sym_\Verts(\R)$ on $\R$ 
(instead of by modulus), we have that
\begin{equation}	\label{lambda.diff}
|\lambda_j(M_1)-\lambda_j(M_2)|\le \|M_1-M_2\|_\op\qquad \forall \; 1\le j\le \Verts \,.
\end{equation}
Since $|a^{\ell} - b^{\ell}| \le \ell |a-b| (|a|^{\ell-1} + |b|^{\ell-1})$ 
for any $a,b\in \R$, $\ell \in  \N$, it follows that
\begin{align}\label{claim2.1}
|\Tr M_1^\ell - \Tr M_2^\ell | &\le 
\sum_{j=1}^\Verts |\lambda_j(M_1)^{\ell} - \lambda_{j}(M_2)^{\ell}|
\notag \\
&\le \ell \|M_1-M_2\|_\op \, (\, \|M_1\|_{S_{\ell-1}}^{\ell-1}+ \|M_2\|_{S_{\ell-1}}^{\ell-1}\, ) \,. 
\end{align}
Considering \eqref{claim2.1}  
for $\ell=3$ and matrices 
$Y$, $X \in \cB_Y	
 \subseteq  \cX_\Verts \subset \B_\HS(\Verts)$, we have by 
\eqref{h3:XYop} that
\begin{equation} \label{133}
|\Tr X^3-\Tr Y^3| \le 3 \|X-Y\|_\op \, ( \|X\|^2_{\HS} + \|Y\|^2_{\HS} ) 
\le 6 \eps \Verts^3 \,.
\end{equation}
Consequently,
when $\eps=o(p^3)$, for which it suffices 
to take $\rnk \gg p^{-6}$ (see \eqref{h3:XYop}), we get by the 
triangle inequality, that uniformly over $Y\in \cN$ and $X_1,X_2\in \cB_Y$, 
\begin{equation}	\label{h3.fluct}
|h_3(X_1)-h_3(X_2)| = (\Verts p)^{-3} |\Tr X_1^3-\Tr X_2^3| =o(1)\,.
\end{equation}
Hence, by Corollary \ref{cor:covering} with $\{\cB_i\}_{i\in \cI}= \{\cB_Y\}_{Y\in \cN}$ and $\cE=\emptyset$, we deduce that 
\begin{align}
\pro{ h_3(\Anp) \ge 1+u} 
&\le |\cN| \expo{ - \phigen_{ p}(h_3,1+u-o(1))}		\notag \\
&= \expo{ -\upphi_{\Verts,p}(\Cyc_3,u-o(1)) + O(\rnk\Verts\log\Verts)} \,.\label{h3:above}
\end{align}
The main term in \eqref{h3:above} dominates the error term 
when $\rnk \Verts\log\Verts \ll \Verts^2 p^2$. We can
satisfy this and our requirement that $\rnk \gg p^{-6}$, provided
$p\gg ((\log\Verts)/\Verts)^{1/8}$.

\subsection{Refined approach}\label{sec:ref-approach}
The element $Y=Y(A)\in \cN$ 
was obtained by approximating the rank $\rnk$ matrix $A_{\le \rnk}$ associated with
each adjacency matrix $A\in \cA_\Verts$. 
In doing so, we can even take 
$\delta = \Verts^{-3 \ell}$ and
the net $\cN$ fine enough to ensure 
$\|A_{\le \rnk} -Y(A)\|_\HS \le 3 \delta\Verts$ while still having 
$\log|\cN|\ls \ell \rnk\Verts \log\Verts$ (cf.\ \Cref{lem:net1}).
Thus, $Y$ is essentially the projection of $A$ onto the eigen-space of its $\rnk$ extremal eigenvalues. 
In particular, the images of the linear operators
 $Y$ and $A-Y$ are nearly orthogonal linear sub-spaces of $\R^n$. 
 This property roughly carries over to any matrix $X$ in the convex hull $\cB'_Y$
of all $A\in \cA_\Verts$ with $\|A_{\le \rnk} - Y\|_\HS \le 3 \delta\Verts$ (see  
\eqref{def:By}). Consequently, 
\begin{equation}\label{eq:fluct-triag}
|\Tr X^3 - \Tr Y^3| \approx |\Tr(X-Y)^3| \le 
\|X-Y\|_{S_3}^3 \approx 
\|X_{>\rnk}\|_{S_3}^3
\qquad \forall X\in \cB'_Y\,,
\end{equation}
thereby reducing the task of controlling the fluctuation of $h_3(X)$ on 
sets $\cB'_Y$ to that of bounding the tail of the (absolute) third moment of the spectrum. Such approximate orthogonality applies to any 
spectral function of $A$ that is dominated by the large eigenvalues
(among 
$\hom(H, \cdot)$ these are precisely 
$\hom(\Cyc_\ell, \cdot)$, but
Schatten norms 
also have this property). The bound $\Verts^3/\sqrt{\rnk}$ of
\eqref{133} is the best we can achieve in \eqref{eq:fluct-triag} with the bound \eqref{Rapprox.basic} on $\{ |\lambda_j|$, $j>\rnk \}$.
While 
it is essentially sharp for \emph{general} elements of $\cX_\Verts$, for \emph{random} elements of $\cA_\Verts$ (under the Erd\H{o}s--R\'enyi measure $\mu_p$) we can do much better. 
Indeed, with probability $1-o(1)$ we have $\lambda_1(\Anp)\sim \Verts p$ and $|\lambda_2(\Anp)|=O(\sqrt{\Verts p})$. In fact, reordering the eigenvalues as $\lambda_1(\Anp)\ge \lambda_2(\Anp)\ge \cdots\ge \lambda_\Verts(\Anp)$, 
we have that $\lambda_2(\Anp)/\sqrt{\Verts p}$ and $\lambda_\Verts(\Anp)/\sqrt{\Verts p}$ ``stick" to the edges $\pm2$ of the support of Wigner's semicircle distribution
(see Lemma \ref{lem:gap}).
However, we are limited to exploiting properties of random elements holding with probability $1-\exp(-\omega(\Verts^2p^2\log(1/p)))$, which do not include the event that $|\lambda_2(\Anp)|=O(\sqrt{\Verts p})$.
Indeed, on the events Clique and Hub from \eqref{Clique}, \eqref{Hub} the Perron--Frobenius eigenvalue $\lambda_1(\Anp)$ is joined by a second ``outlier" eigenvalue at scale $\Verts p$ (on Clique it is a positive outlier, while on Hub it is negative). 
Additional outlier eigenvalues correspond to having a large-scale pattern for the edge distribution which is of rank $\ge 3$ (see \cite{Tao-blog-regularity-lemma} for one formalization of this heuristic).
Fortunately, for $\Cyc_\ell$-counts we only need
\begin{equation}\label{eq:power}
\|\Anp_{>\rnk}\|_{S_\ell} = o(\Verts p)
\end{equation}
with $\rnk=\rnk(\Verts)$ growing poly-logarithmically, in order
to allow $p$ of size $\Verts^{-1/2}(\log\Verts)^C$.
Using the appropriate exceptional sets,
we accomplish this by utilizing Proposition \ref{prop:tail.refined}
(for triangle counts we must further assume $p\gg \Verts^{-1/3}$, though this can be avoided with some extra arguments -- see Remark \ref{rmk:Augeri}).

For the lower tail bound in Theorem \ref{thm:cycles} we can only exclude events of probability $1-\exp( -\omega(\Verts^2p))$, hence the somewhat larger lower limit on $p$ in \eqref{prange.cyc.lower}, but as seen 
in Section \ref{sec:lowertails}, for even $\ell$ we have
no such restrictions (by the convexity of sub-level sets).

A key feature of cycle homomorphism counts is that they can be expressed as functions of the spectrum of $\Anp$ alone, which lets us get sharp control on the fluctuations of these functions on the sets $\cB_Y'$ from \eqref{eq:fluct-triag} via \eqref{eq:power}.
For general $H$ as in Theorem \ref{thm:hom} we lack a spectral representation of $\hom(H, \cdot)$, 
so instead of the sets $\cB_Y'$ we use a covering by spectral-norm balls. In particular, we cannot exploit orthogonality of the images of $Y$ and the residual $A-Y$ as we do for cycles to obtain sharp control on the fluctuations of $\hom(H, \cdot)$.  
Nevertheless, 
after removing improbable events involving 
extremely large values of $\hom(F, \Anp)$ for sub-graphs $F$ of $H$, 
we get strong control on fluctuations of $\hom(H, \cdot)$ by an iterative ``pruning" procedure on $H$,
resulting in the ``counting lemma" of \Cref{prop:countingB} below.

\subsection{Regularity and counting lemmas for random graphs}
\label{sec:regularity}

We first recall the definitions of the space $\cW_0$ of graphons and the cut metric. Denote by $\cW$ the space
of bounded, symmetric, Lebesgue-measurable functions $f:[0,1]^2\to \R$
(as in Section \ref{sec:lowertails}), equipped with the \emph{cut-norm}
\[
\|f\|_\square = \sup_{S,T\subseteq[0,1]} \Big| \int_{S\times T} f(x,y) dxdy \Big|,
\]
where the supremum is taken over measurable subsets of $[0,1]$. 
The \emph{cut-metric} on $\cW$ is then
\[
\delta_\square(f,g):= \inf_{\sigma\in \Sigma} 
\{ \|f-g^{\sigma}\|_\square \} \,, 
\]
where 
$g^\sigma(x,y):= g(\sigma(x), \sigma(y))$ and
the infimum is taken over all measure-preserving bijections $\sigma:[0,1]\to[0,1]$.   
On $\cW$ we have the equivalence relation $f\sim g$ if and only if 
$f=g^\sigma$ for some $\sigma\in \Sigma$, and denote by $\wt{g}$ 
the $\delta_\square$-closure of the corresponding
orbit $\{g^\sigma:\sigma\in \Sigma\}$ of $g \in \cW$.
Setting as $\cW_0$ the collection of elements $f \in \cW$ with $0 \le f \le 1$,
the associated quotient spaces 
$\wt{\cW}=\{\wt{g}: g\in \cW\}$, $\wt{\cW}_0= \{ \wt{g}: g\in \cW_0\}$
are thus $\delta_\square$-metrizeable.

Graphons provide a topological reformulation of the \emph{regularity method} from extremal graph theory, which rests on two key facts: Szemer\'edi's regularity lemma, and the counting lemma. These can be formulated for graphons as follows (cf.\ \cite{Lovasz-book}):

\begin{lemma}[Weak regularity lemma for graphons]
\label{lem:WRL}
For every $f\in \cW_0$ and $k\ge 1$ there exists a step function $g\in \cW_0$ with $k$ steps (i.e.\ a partition $\cP$ of $[0,1]$ into $k$ measurable sets, such that $g$ is constant on $S\times T$ for all $S,T\in \cP$) such that
\[
\|f-g\|_\square \le \frac{2}{\sqrt{\log k}} \,.
\]
\end{lemma}
\begin{lemma}[Counting lemma for graphons]
\label{lem:counting}
For every simple graph $H$ and every $f,g\in \cW_0$, 
\[
|t(H,f) - t(H,g) | \le \edges(H) \delta_\square(f,g).
\]
(Recall the homomorphism density functionals $t(H,\cdot)$ from \eqref{def:tH}.)
\end{lemma}

The weak regularity lemma is closely related to the fact that $(\wt{\cW}_0, \delta_\square)$ is a compact metric space, while the counting lemma says that the functionals $t(H,\cdot)$ are continuous with respect to the cut-metric. 
Taken together, they allow one to cover large deviation events for \emph{dense} Erd\H{o}s--R\'enyi graphs by a bounded collection of graphon neighborhoods on which the functions $\hom(H, \cdot)$ are essentially constant, which was the approach taken in
 \cite{ChVa11}.
Towards proving Theorem \ref{thm:hom} we 
obtain the following quantitative analogues of the regularity and counting lemmas for the probability space $(\cA_\Verts,\mu_p)$. A notable feature is to replace the 
cut norm, which for $f\in \cW$ is equivalent to the $L_{\infty}\to L_1$ norm of the associated operator $g\mapsto \int_0^1 f(\cdot,y)g(y)dy$ (cf.\ \cite[Section 8.2]{Lovasz-book}),
by the (spectral) $\ell_2\to \ell_2$ operator norm on matrices.

\begin{prop}[Spectral regularity lemma for random graphs]
\label{prop:RLRG}
For some absolute constant $C_\star<\infty$, 
any $\errRLs \le 1$, $K,\Delta >0$,  $\Verts^{-1}\log \Verts \le p<1$ and 
$\Verts \ge \rnk \ge  K (p^\Delta/\errRLs^2) \log(1/p)$, there
exists a partition $\cA_n = \bigsqcup_{j=0}^{\Net} \cE_j$ such that 
\begin{align}
\label{RLRG:a} 
\log \Net & \le \rnk(\Verts+2) \log\big(\frac{3\Verts}{\errRLs} \big) 
\,, \\
\label{RLRG:b} 
\mu_p(\cE_0) & \le 4 \exp(- K \Verts^2 p^\Delta \log (1/p)) \,,
\end{align}
and for each $1 \le j \le \Net$, 
there exists $Y_j\in \Sym_{n,\rnk }(\R) \cap \B_\HS(n)$ with
\begin{equation}	\label{RLRG.diam}
\max_{A \in \cE_j} \{ \|A-(p\jay+Y_j)\|_\op \}  
\le C_\star ( \sqrt{\Verts p} + \errRLs n )  \,.
\end{equation}
(Recall the notation $\jay= \1\1^\tran-\id$ for the adjacency matrix of the complete graph.)
\end{prop}
\begin{remark} Note that \Cref{prop:RLRG} is essentially optimal for 
establishing the sharp upper tail for $\hom(H,\Anp)$ in the regime $p\gg \Verts^{-1/\Delta(H)}$. Indeed,
we seek a partition of $\cA_\Verts$ which is fine enough to detect the presence of a clique or hub structure of appropriate size -- for the upper tail event of $\hom(H,\Anp)$ the events are as in \eqref{Clique} and \eqref{Hub} but with $p$ replaced by $p^{\Delta(H)/2}$.
These planted structures are perturbations of $\Anp$ of norm $\Theta(np^{\Delta(H)/2})$.
In view of \eqref{RLRG.diam}, \Cref{prop:RLRG} at 
$\Delta=\Delta(H) \ge 2$ and $\errRLs = O(p^{\Delta/2})$ exhibits a fine enough partition, 
provided $\rnk\gg \log(1/p)$ (where 
choosing a 
sufficiently large $K=O(1)$ yields 
an exceptional event $\cE_0$ which is negligible in comparison 
to the large deviation event). By \eqref{RLRG:a}, the entropy $\log \Net$ of this partition is also negligible in comparison to the rate 
$\RR_{\Verts,p}(H, u) \asymp_{H,u} \Verts^2 p^{\Delta(H)}\log(1/p)$ provided 
$\rnk\ll \Verts p^{\Delta(H)}$, and  we can satisfy both constraints on $\rnk$ 
as long as $\Verts p^{\Delta(H)} \gg \log \Verts$.
\end{remark}
\begin{remark}
To allow for $p\ll \Verts^{-\Delta(H)}$ would require a reduction of the size of the collection of events $\cE_j$ used for covering the large deviation events for $\hom(H, \Anp)$. 
The right hand side of \eqref{RLRG:a} is the metric entropy of the set of rank $k$ matrices in $\ball_\HS(\Verts)$ and we cannot do any better for a spectral norm covering of $\cA_\Verts$. 
However, it is likely that the large deviation events themselves have much lower metric entropy. Indeed, on the events \eqref{Clique} and \eqref{Hub}, $\Anp$ is close to rank-2 matrices whose eigenvectors consist of the nearly-constant Perron--Frobenius eigenvector, together with a second eigenvector that is localized to a set of size $\Theta(\Verts p)$ and $\Theta(\Verts p^2)$ respectively. The set of low-rank matrices with strongly localized eigenvectors has significantly smaller metric entropy than the right hand side of \eqref{RLRG:a}, but we
do not pursue this direction here.
\end{remark}
Hereafter we denote the sub-level sets
\begin{equation}	\label{def:LH}
\cL_H(t):= \{ X\in \cX_\Verts: \hom(H,X) \le t n^{\verts(H)} p^{\edges(H)} \}.
\end{equation}
The following analogue of \Cref{lem:counting} shows that for any graph $H$ and $K=O(1)$, after 
localizing to a region that is near all of the sub-level sets $\cL_F(K)$ with $F\prec H$, the normalized 
homomorphism count $\Verts^{-\verts(H)} p^{-\edges(H)}\hom(H,\cdot)$ is 
$O_H(1/(\Verts p^{\Delta_\star}))$-Lipschitz in the spectral norm.

\begin{prop}[Spectral norm counting lemma]\label{prop:countingB}
$~$

\noindent
For any finite graph $H$ with $\Delta_\star(H) \le \Delta_\star$ and
any convex set $\cB \subseteq \cX_\Verts$ satisfying
\begin{align}	\label{homF-min}
\cB \cap \cL_F(K) \ne \emptyset & \,, \quad  \forall F \prec H  \,, \\
\label{op-diam-cB}
\sup_{X,Y\in \cB} \{ \|X-Y\|_\op \} &\le \eps_0\Verts p^{\Delta_\star} \,,
\end{align}
for some $p \in (0,1)$, $\Verts \in \N$ and $K \ge 1 \ge \eps_0$,
we have that for all $F\preccurlyeq H$, 
\begin{equation}	\label{eq:bd-fluct}
\Fluct(F;\cB) := \sup_{X,Y\in \cB} \{ | \hom(F, X) - \hom(F,Y)| \}  \le C_H
\eps_0 K \Verts^{\verts(F)}p^{\edges(F)}
\end{equation}
for a constant $C_H<\infty$ depending only on $H$.
\end{prop}
\begin{remark}\label{prop:counting}
Note that \eqref{eq:bd-fluct} applies for any spectral-norm ball $\cB$ of 
radius $\eps_0\Verts p^{\Delta_\star(H)}$ which intersects  $\bigcap_{F\prec H} \cL_F(K)$.
Thus, with $\hom(H,\Anp)$ typically of size $\Verts^{\verts(H)} p^{\edges(H)}$, our counting lemma says that a spectral-norm net for $\cA_\Verts$ as in \Cref{prop:RLRG} can detect deviations of size $\eps \Verts^{\verts(H)} p^{\edges(H)}$ as long as we have there $\errRLs = O(\eps p^{\Delta_\star(H)})$. 
This is the only reason for the constraint $p\gg \Verts^{-1/(2\Delta_\star(H))}$ in \Cref{thm:hom}; an improvement of \Cref{prop:countingB} with $\Delta(H)/2$ in place of $\Delta_\star(H)$ would allow us to have 
$p\gg \Verts^{-1/\Delta(H)}$ in \Cref{thm:hom} (recall that $\Delta(H)+1\le 2\Delta_\star(H)\le 2\Delta(H)$). 
We achieve such an improved sparsity range 
in \Cref{thm:cycles}, via a finer  analogue of \eqref{eq:bd-fluct} that exploits the tighter 
relationship between $\hom(C_\ell, \Anp)$ and the spectrum of $\Anp$ (see \Cref{claim2}). 
\end{remark}

\begin{remark}[Comparison with the entropic stability method]
\label{rmk:HMS2}
Taken together, Propositions \ref{prop:RLRG} and \ref{prop:countingB} yield a covering of $\cA_\Verts$ by spectral-norm balls on which $\hom(H,\cdot)$ is essentially constant, allowing us to cover 
large deviation events for $\hom(H,\Anp)$ by a relatively small number of sets where we can apply \Cref{prop:cvx}. However, the covering becomes inefficient when $p$ is too small, and it is natural to ask whether in that case one could still find an efficient covering for the large deviation event only, rather than for all of $\cA_\Verts$.

The approach developed in the more recent work \cite{HMS19} (and also followed in \cite{BaBa20}) uses a refinement of a moment method argument from \cite{JOR04} to show that the large deviation event 
must coincide with the appearance of a small localized structure that they call a \emph{core} (recall from Remark \ref{rmk:HMS1} that these works consider large deviations for subgraph counts $\sub(H,\Gnp)$ rather than homomorphism counts). Morally, cores are approximately cliques and hubs of the appropriate size (as is indeed shown to hold when $H$ is a clique in \cite{HMS19}). 
In this way, they obtain a covering of the large deviation event by events of the form $\{\prod_{\{i,j\}\in E} \Anp(i,j) = 1\}$ with $E\in {[\Verts]\choose 2}$ ranging over the possible edge sets for cores. 
The technical challenge that takes up the bulk of their proof is to show that this covering is indeed efficient, i.e.\ that the number of cores is at most $\exp( o( \upphi_{\Verts,p}(H,u)))$, a property they term \emph{entropic stability}. 
The reduction to verification of this property applies to general low-degree polynomials of sparse i.i.d.\ Bernoulli variables. However, the task of establishing entropic stability is problem specific;  in the case of subgraph-counting functions 
the task involves many nontrivial facts from graph theory, and at present it has only been achieved for counts of regular graphs.
%
%
\end{remark}

\section{Preliminary control on the spectrum}
\label{sec:spec}
We consider for $1\le m\le \Verts$, the norms
\begin{equation}\label{def:HSR}
\|\Anp_{\le m}\|_\HS
= \sup_{{\sf W}:\dim {\sf W} = m} \|\Pi_{\sf W} (\Anp)\|_\HS \,, 
\end{equation}
where $\Pi_{\sf W}$ denotes the operator for projection to the subspace ${\sf W}$,
and link the growth of $m \mapsto \|X_{\le m}\|_\HS$ to
the decay of $\rnk \mapsto \|X_{\ge \rnk}\|_{S_\alpha}$ (when $\alpha > 2$).
\begin{lemma}\label{lem:HSk-Salpha-tail}
Fixing finite $L,D \ge 0$, let
\begin{equation}
\cG(L,D) := \big\{ X \in \Sym_\Verts(\R) \,:\, 
\|X_{\le m}\|_\HS \le L  + \sqrt{m} D \,, \qquad \forall\; 1\le m\le \Verts\, 
\big\} \, .
\end{equation}
Then, for 
$\kappa_\alpha:= (\frac{2}{\alpha-2})^{1/\alpha}$,
any $L,D$, $\alpha \in (2,\infty]$, $1 \le \rnk \le \Verts$ and $X \in \cG(L,D)$,
\begin{equation}\label{tail.refined}
\|X_{>\rnk}\|_{S_\alpha} \le (\Verts-\rnk)^{1/\alpha} D  +
\kappa_\alpha \, L \, \rnk^{1/\alpha-1/2} \,.
\end{equation}
\end{lemma}
\begin{proof} Recall from \eqref{bd:HS-k}, that if
$X \in \cG(L,D)$, then for any $m \in [\Verts]$,
\begin{equation}	\label{lambdaR.bd}
|\lambda_m(X)| = \|X_{\ge m}\|_\op \le m^{-1/2} \|X_{\le m}\|_\HS \le D + L m^{-1/2} \,,
\qquad \forall\; 1\le m\le \Verts \,.
\end{equation}
That is, \eqref{tail.refined} holds at $\alpha=\infty$ (with $\kappa_\infty=1$). Having 
\eqref{lambdaR.bd} at all $m \in (\rnk,\Verts]$, it follows by 
the triangle inequality, that for any finite $\alpha > 2$,
\[
\|X_{>\rnk}\|_{S_\alpha} \le (\Verts-\rnk)^{1/\alpha} D + L \Big( \sum_{m>\rnk} m^{-\alpha/2} \Big)^{1/\alpha}
\,.
\] 
Further, bounding the latter $\ell_\alpha$-norm on $\R^\Verts$, we get 
that for any $\alpha>2$ and $\rnk \ge 1$,  
\[
\Big( \sum_{m=\rnk+1}^\infty m^{-\alpha/2} \Big)^{1/\alpha}
\le 
\Big(\int_{\rnk}^{\infty} u^{-\alpha/2} du\Big)^{1/\alpha}
= \kappa_\alpha \rnk^{1/\alpha-1/2} \,,
\]
thereby establishing \eqref{tail.refined}.
\end{proof}

The main result of this section, used for controlling the 
exceptional set $\cE$ in \Cref{cor:covering},
is as follows.
\begin{prop} 
\label{prop:tail.refined}
For some $C,C', c>0$, any $K\ge 2$ and $\Verts p \ge \log\Verts$,

\begin{equation}\label{def:excep.pm}
\pr\Big(\Anp \notin \cG\big(K \Verts p, C' \sqrt{\Verts p}\,\big)\, \Big) 
\le C \expo{ -cK^2\Verts^2p^2} =: \pexcep (K) \,.
\end{equation} 
Hence, up to probability $\pexcep(K)$, the matrix $\Anp$ 
satisfies \eqref{tail.refined} with $L=K \Verts p$, $D=C' \sqrt{\Verts p}$,  
any $\alpha \in (2,\infty]$ and all $1\le \rnk\le \Verts$.
\end{prop}

For $A\in \cA_\Verts$ and $0 \le \rnk \le n$ we define 
\begin{equation}	\label{def:Ak}
A^{(\rnk)} := p\jay + (A - p\jay)_{\le \rnk}\,, 
\end{equation}
so $A^{(0)}=p\jay$, while $A^{(\Verts)}=A$. Our next lemma shows that $\Anp^{(\rnk)}$ 
approximates $\Anp$ in the spectral norm 
(whereas by \Cref{lem:net}, if $\rnk$ is not too large 
we can approximate $\Anp^{(k)}$ with a net of acceptable size). 
In particular, since $\|\jay\|_{\HS} \le n$, combining \eqref{HS:bound} 
for $t=(K-1) \Verts p$ and $k=m$, with a union bound over $1 \le m \le \Verts$, yields 
\Cref{prop:tail.refined} for $C'=C_1$.
\begin{lemma}
\label{lem:HSk}
For any $\kappa>0$ there exists $C_\kappa<\infty$ such that for
$\frac{\kappa}{n} \log n\le p\le 1$, $t\ge0$ and $1\le \rnk\le n$, 
\begin{align}	\label{HS:bound}
\pr \big(  \| ( \Anp - p\jay)_{\le\rnk} \|_{\HS} \ge  C_\kappa \sqrt{\rnk np} + t  \big) &\le 4e^{-t^2/16} \,, \\\label{approx:op}
\pr \Big( \| \Anp - \Anp^{(\rnk-1)}\|_\op \ge C_\kappa \sqrt{np} + \frac{t}{\sqrt{\rnk}} \Big) &\le 4e^{-t^2/16} \, . 
\end{align}
\end{lemma}

In proving Lemma \ref{lem:HSk} we employ the following well-known 
concentration inequality.
\begin{theorem}[\!\! cf.\ {\cite[Theorem 6.6]{Talagrand-newlook}}]
\label{thm:Talagrand}
Suppose
$F:[-1,1]^d\to \R$ is convex and $L$-Lipschitz with respect to the Euclidean metric for some $L<\infty$ and the random vector  
$\xi\in [-1,1]^d$ has independent components.
Then, for any median $m$ of $F(\xi)$ and $t\ge0$, 
\[
\pr(|F(\xi)-m|\ge t) \le 4\exp \Big\{ -\frac{t^2}{16L^2}\Big\} \, .
\]
\end{theorem}

We further need some control on the spectral gap of $\Anp$, as in the 
following result about the spectral norm of sparse Wigner matrices 
(whose root goes back to
\cite{FuKo81}).
\begin{lemma}[\!\! cf.\ {\cite[Theorem 3.2]{BBK}}, {\cite[Example 4.10]{LHY}}]
\label{lem:gap}
Let $\jay$ be as in \eqref{def:jay}.
For any $\kappa>0$ there exists $C_\kappa < \infty$ such that 
$C_\kappa \to 4$ as $\kappa \to \infty$ and
if $\kappa \log\Verts\le \Verts p\le \Verts/2$, then
\[
\e \|\Anp-p \jay\|_\op \le \frac{C_\kappa}{2}\sqrt{\Verts p} \,.
\]
\end{lemma}

\begin{proof}[Proof of \Cref{lem:HSk}]
Since $\|X_{\le k}\|_\HS \le \sqrt{k} \|X\|_\op$,  \Cref{lem:gap} and Markov's inequality yield
\begin{equation}	\label{HSk:goal}
\pr \big(\, \|(\Anp-p\jay)_{\le k}\|_\HS \le  C_\kappa \sqrt{k \Verts p} \,\big) \ge \frac{1}{2} \,.
\end{equation}
The bound  \eqref{HS:bound} then follows by Theorem \ref{thm:Talagrand}
as the mapping $A\mapsto \|A_{\le m}\|_{\HS}$ is 
convex and 1-Lipschitz with respect to $\|\cdot\|_{\HS}$. Next, with
$A-A^{(\rnk-1)} = (A-p\jay)_{\ge \rnk}$ (see \eqref{def:Ak}),   
by the left inequality in \eqref{lambdaR.bd}, 
\[
\|A-A^{(\rnk-1)}\|_\op = \|(A-p\jay)_{\ge \rnk}\|_\op \le k^{-1/2} \| (A-p\jay)_{\le \rnk} \|_{\HS} 
\]
and \eqref{approx:op} follows from  \eqref{HS:bound}.
\end{proof}

\section{Proof of regularity and counting lemmas for random graphs}\label{sec:reg-proof}

\subsection{Proof of Proposition \ref{prop:RLRG}}\label{sec:RLRG} $~$
Recall \Cref{lem:net1} that for any $1 \le k \le n$ and $\errRLs \in (0,1]$, the set
$\Sym_{\Verts,\rnk }(\R) \cap \B_\HS(\Verts)$ has a $3\Verts\errRLs$-net $\cN$
in the Hilbert--Schmidt norm with
\begin{equation}	\label{logcN.bd}
\log |\cN| \le  \rnk(\Verts+2)\log(3\Verts/\errRLs) \, .
\end{equation}
Hence, in view of \eqref{def:Ak}, for any $A\in \cA_\Verts$ there exists $Y\in \cN$ such that
\begin{equation}	\label{RLRG:above}
\|A^{(k)} - (p\jay+Y)\|_\op\le \|A^{(k)}-(p\jay+Y)\|_\HS \le 3\Verts\errRLs . 
\end{equation}
In particular, by \eqref{RLRG:above} and the triangle inequality, for any  
$C,r>0$, setting
\[
\cE_0 = \Big\{ A\in \cA_\Verts :  \|A-A^{(k)}\|_\op \ge C \sqrt{\Verts p} + r \Verts \errRLs  \Big\}\,,
\]
any enumeration $\{Y_j\}_{j=1}^\Net$ of those $Y\in \cN$ for which 
\[
\cB_Y := \Big\{ X\in \cX_\Verts: \|X-(p\jay+Y)\|_\op \le C \sqrt{\Verts p} + 
(r+3) \Verts \errRLs \Big\}\,,
\]
intersects $\cA_\Verts$, induces a covering of $\cA_\Verts\setminus \cE_0$ by
pairwise disjoint sets $\cE_j\subseteq\cA_\Verts\cap \cB_{Y_j}$
where \eqref{RLRG:a} holds thanks to \eqref{logcN.bd} and
\eqref{RLRG.diam} follows from the definition of $\cB_{Y_j}$.
Setting $C=C_1$ of \Cref{lem:HSk} and $r=4$, it follows from \eqref{approx:op} 
with $\rnk+1$ in place of $\rnk$ and $t=\sqrt{\rnk} r \Verts \errRLs$, that \eqref{RLRG:b} holds whenever
$\rnk \errRLs^2 \ge K p^\Delta \log(1/p)$, as claimed.
\qed

\subsection{Proof of \Cref{prop:countingB}} \label{sec:fluct} 
We begin with the following crude bound on the directional derivatives of $\hom(H, \cdot)$.
\begin{lemma}[Derivatives of homomorphism counts]
\label{lem:derivatives}
For $W,Z\in \Sym_\Verts(\R)$ and a simple graph $H=(V,E)$, 
the directional derivative of $\hom(H, \cdot)$ at $W$ in the direction $Z$ is
\begin{equation}
\cD_H(W,Z) := \big\langle Z \,, \nabla\hom(H, W) \big\rangle_\hs 
=  \sum_{1\le i<j\le \Verts}Z_{ij}\partial_{W_{ij}} \{ \hom(H, W) \} \,.	\label{def:D}
\end{equation}
Fixing a non-empty simple graph $H=([\verts],E)$,  for $v \in V$ 
let $H_{(v)}$ denote the induced subgraph of $H$ on the vertices 
$V\setminus \{v\}$. Then, if $W \in \cX_\Verts$, 
\begin{equation}\label{eq:bd-D}
|\cD_H(W,Z)| \le  \|Z\|_{\op} 
\sum_{\{v_1,v_2\} \in E}  \sqrt{ \hom(H_{(v_1)}, W) \hom(H_{(v_2)},W)} 
\,. 
\end{equation}
\end{lemma}
\begin{proof} 
For $\bs{i}=(i_1,\dots, i_{\verts})\in[\Verts]^{\verts}$, $W\in \Sym_\Verts(\R)$, 
and $E'\subseteq E$, we denote
\[
W_{E'}(\bs{i}) := \prod_{e'=k'l' \in E'} W_{i_{k'},i_{l'}}, \qquad W_{\emptyset}(\bs{i}):= 1\,,
\]
so that
\[
\hom(H, W) = \sum_{\bs{i} \in [\Verts]^{\verts}} W_{E} (\bs{i})
\,. 
\]
All directional derivatives are zero when $E=\emptyset$. 
Thus, assuming \abbr{wlog} that $\verts \ge 2$ and $m=|E| \ge 1$, 
from \eqref{def:D} we can express the directional derivative as a sum over ``labeled" homomorphism counts in which all but one of the edges are labeled by entries of $W$, with the remaining edge labeled by an entry of $Z$:
\begin{align*}
\cD_H(W,Z) = \sum_{e \in E} \hom(H, \Lab^e) \,, \qquad 
\hom(H, \Lab^{\{k,l\}}) := \sum_{\bs{i} \in [\Verts]^\verts}\, Z_{i_ki_l} 
W_{E \setminus \{k,l\}} (\bs{i})
\,.
\end{align*} 
Hence, it suffices to show that for any $W \in \cX_\Verts$ and 
$e = \{v_1,v_2\} \in E$,
\begin{equation}\label{eq:bd-Le}
| \hom(H, \Lab^e) | \le \|Z\|_{\op} \sqrt{ \hom(H_{(v_1)}, W)
 \hom(H_{(v_2)}, W) }  \,.
\end{equation}
To this end, \abbr{wlog} take $e=\{1,2\}$ and partition $E$ to 
$\{e\},E_1,E_2,E_3$, where for $j=1,2$, we denote by $E_j$ the 
set of edges incident to vertex $j$ in the graph $H$, with the exception 
of $e$.  With $W_{E_3}(\bs{i})$ independent of $i_1,i_2$, 
we have that
\[
\hom(H, \Lab^{\{1,2\}}) 
= \sum_{i_3,\dots, i_\verts\in [\Verts]} W_{E_3}(\bs{i}) \sum_{i_1,i_2\in [\Verts]} 
W_{E_1}(\bs{i}) \, Z_{i_1,i_2} \, W_{E_2}(\bs{i}) \,.
\]
Further, for any fixed $i_3,\dots, i_\verts$, the value of $W_{E_1}(\bs{i})$ 
depends only on $i_1$, with $W_{E_2}(\bs{i})$ depending only on 
$i_2$. The inner sum is thus a quadratic form in $Z$, yielding that
\begin{align*}
\bigg| \sum_{i_1,i_2\in [\Verts]} W_{E_1}(\bs{i}) Z_{i_1,i_2} W_{E_2}(\bs{i})\bigg| 
&\le \|Z\|_\op \Big( \sum_{i_1\in [\Verts]} W_{E_1}(\bs{i})^2\Big)^{1/2} 
\Big( \sum_{i_2\in [\Verts]} W_{E_2}(\bs{i})^2\Big)^{1/2}  \\
& \le \|Z\|_\op 
\Big( \sum_{i_1\in [\Verts]} W_{E_1}(\bs{i}) \Big)^{1/2} 
\Big( \sum_{i_2\in [\Verts]} W_{E_2}(\bs{i}) \Big)^{1/2} 
\,,
\end{align*} 
where in the last inequality we used the fact that  
$W_{E'}(\bs{i}) \in [0,1]$ for any $\bs{i}$, $E'$ and $W \in \cX_\Verts$.
Consequently, by the above bound and Cauchy--Schwarz,
\begin{align*}
|\hom(H, \Lab^{\{1,2\}}) |
&\le \sum_{i_3,\dots, i_\verts\in [\Verts]} W_{E_3}(\bs{i}) \bigg| \sum_{i_1,i_2\in [\Verts]}  W_{E_1}(\bs{i}) Z_{i_1,i_2} W_{E_2}(\bs{i}) \bigg| \\
&\le  \|Z\|_\op  
 \Big( \sum_{i_1,i_3,\dots, i_\verts\in [\Verts]} W_{E_3} (\bs{i}) W_{E_1}(\bs{i}) \Big)^{1/2}
\Big( \sum_{i_2,i_3,\dots, i_\verts\in [\Verts]} W_{E_3} (\bs{i}) W_{E_2} (\bs{i}) \Big)^{1/2}  \\
& =  \|Z\|_\op  \Big( \hom(H_{(2)},W) \Big)^{1/2} 
\Big( \hom(H_{(1)},W) \Big)^{1/2} 
\,.
\end{align*}
The same holds for any $e \in E$, resulting with \eqref{eq:bd-Le}
and thereby with \eqref{eq:bd-D}.
\end{proof}

For any set $\cB\subseteq \cX_\Verts$ and any graph $F$ 
(including 
when $\edges(F)=0$),  
we trivially have that
\begin{equation}	\label{MaxBF.trivial}
\bd(F;\cB) := \sup_{X\in \cB} \{ \hom(F, X) \} \le \bd(F;\cX_\Verts) \le \Verts^{\verts(F)} \,.
\end{equation}
We also have the following immediate consequence of 
Lemma \ref{lem:derivatives}.
\begin{lemma}
\label{lem:fluct}
For any non-empty simple graph $F$ and convex set $\cB\subseteq\cX_\Verts$,
\begin{equation}\label{fluct.goal}
\Fluct(F;\cB) \le  \, \sup_{X,Y \in \cB} \{ \|X-Y\|_{\op} \}  \, 
\sum_{\{v_1,v_2\} \in E(F)}  \sqrt{ \bd(F_{(v_1)};\cB) \bd(F_{(v_2)};\cB) 
 \,.}
\end{equation}
\end{lemma}
\begin{proof}
Fixing $X,Y \in \cB$, for $t \in [0,1]$ let $W_t = (1-t) Y + t X$. Note that
\[
\hom(F, X) - \hom(F, Y) =  \int_0^1 \frac{d}{dt} \{ \hom(F, W_t) \} dt = 
\int_0^1 \cD_F(W_t,X-Y) dt \,.
\]
Applying the bound \eqref{eq:bd-D} on the expression on the \abbr{rhs},
\[
|\hom(F, X)-\hom(F, Y)| \le  \|X-Y\|_{\op}  \int_0^1 \sum_{\{v_1,v_2\} \in E(F)} 
 \sqrt{\hom(F_{(v_1)},W_t) \hom(F_{(v_2)},W_t)} \, dt\,.
\]
Since $\cB$ is convex, $W_t\in \cB$ for all $t \in [0,1]$. Hence
\[
\hom(F_{(v)},W_t) \le \bd(F_{(v)};\cB)
\]
and \eqref{fluct.goal} follows by combining the previous two displays.
\end{proof}
We proceed to establish Proposition \ref{prop:countingB} by 
iterating the preceding lemma (thereby sharpening 
the argument from \cite[Lemma 5.4]{ChDe14}).

\begin{proof}[Proof of Proposition \ref{prop:countingB}]
 Set $\fun_{1} (\cdot)=\fun(\cdot)+1$,
$\fun(k)=k$ for $k \le \Delta_\star$ and thereafter set 
$\fun(k)=k \fun_1(k-1)$ recursively, to guarantee that for any subgraph 
$F$ of $H$ with $\edges(F) > \Delta_\star$
\begin{equation}\label{def:c-edges}
\sum_{\{v_1,v_2\} \in E(F)} \, 
\sqrt{\fun_1(\edges(F_{(v_1)})) \fun_1(\edges(F_{(v_2)}))} \le \fun(\edges(F)) \,.
\end{equation}
By \eqref{def:Deltas} we have for any 
$e=\{v_1,v_2\} \in E(F)$ and $F \le H$,
\begin{equation}\label{bd:e-gap}
\Delta_\star + \edges(F_{(v_1)})/2 + \edges(F_{(v_2)})/2 \ge  \edges(F) \,.
\end{equation}
To establish \eqref{eq:bd-fluct} we show that
\begin{equation}	\label{eq:bd-fluct1}
\Fluct(F;\cB) \le \eps_0 f(\edges(F)) K \Verts^{\verts(F)}p^{\edges(F)}
\end{equation}
by induction on 
$\edges(F)$. To this end, note that from
Lemma \ref{lem:fluct} together with 
\eqref{op-diam-cB} and
\eqref{MaxBF.trivial},  
we have for any nonempty graph $F$, 
\[
\Fluct(F;\cB) \le  \eps_0 \Verts p^{\Delta_\star} \sum_{e\in E(F)}  
\Verts^{\verts(F)-1}  = 
\eps_0 \, \edges(F) \Verts^{\verts(F)} p^{\Delta_\star} \,.
\]
This also holds trivially in the case $E(F)=\emptyset$ for which 
$\Fluct(F;\cB)=0$, thereby establishing \eqref{eq:bd-fluct1}
for any $F  \le  H$ having $\edges(F) \le \Delta_\star$. 
Next, let $k\in \{\Delta_\star+1,\ldots,\edges(H)\}$ and assume 
inductively that \eqref{eq:bd-fluct1} holds whenever $F\prec H$ 
has $\edges(F) < k$. For such $F$ we then have
from \eqref{homF-min} and the triangle inequality that
\begin{align*}
\bd(F;\cB) &\le \inf_{X \in \cB} \{ \hom(F, X) \} + \Fluct(F;\cB) \\
& \le K\Verts^{\verts(F)}p^{\edges(F)} + \Fluct(F;\cB) 
\le \fun_1(\edges(F)) K \Verts^{\verts(F)}p^{\edges(F)} \,.
\end{align*}
Considering  
$F\preccurlyeq H$
with $\edges(F)=k$, the preceding applies to all $\{F_{(v)}, v \in V(F)\}$. 
Hence, by Lemma \ref{lem:fluct},
\begin{align*}
\Fluct(F;\cB) &\le \eps_0 \Verts p^{\Delta_\star} \sum_{\{v_1,v_2\}\in E(F)}  
\sqrt{\fun_1(\edges(F_{(v_1)})) \fun_1(\edges(F_{(v_2)}))} K \Verts^{\verts(F)-1} 
p^{\edges(F_{(v_1)})/2} p^{\edges(F_{(v_2)})/2} \\
&\le \eps_0  \, \fun(\edges(F)) \, K \Verts^{\verts(F)} p^{\edges(F)} \,,
\end{align*}
as claimed, where in the second inequality we have used
\eqref{def:c-edges} and \eqref{bd:e-gap}.
\end{proof}

\section{Proof of Theorem \ref{thm:hom}: Upper tail for general homomorphism counts} 
\label{sec:hom}

\subsection{Preliminary lemmas}
\label{sec:hom.prelim} 

We will establish the lower bound on $\RR_{\Verts,p}$ as stated in \Cref{thm:hom} by combining 
Propositions \ref{prop:RLRG} and \ref{prop:countingB} with \Cref{cor:covering}.
In doing so, we shall require the following rough bounds on $\upphi_{\Verts,p}$
for showing that the complement of  
$\bigcap_{F\prec H} \cL_F(K)$ is of negligible probability. 
\begin{lemma}
\label{lem:phiH.rough}
Let $F$ be a graph with $\Delta(F)\ge 2$. For $\theta_F(\cdot)$ of \Cref{thm:BGLZ},
any fixed $u > 0$ and $\Verts^{-1} \ll p^{\Delta(F)} \ll 1$,
\begin{equation}	\label{rough.UB}
\upphi_{\Verts,p}(F,u) \le (1+o(1))\theta_F(u) \Verts^2 p^{\Delta(F)} \log(1/p) \,.
\end{equation}
Further, for some absolute constant $p_0>0$ and all $p\in (0,p_0]$,
\begin{equation}	\label{rough.LB}
\upphi_{\Verts,p}(F,u) \ggs ((1+u)^{1/\edges(F)} -1)^{\Delta(F)} \Verts^2 p^{\Delta(F)} \log(1/p).
\end{equation}
\end{lemma}
\begin{proof} The bound \eqref{rough.UB} is merely \cite[Proposition 2.1(b)]{BGLZ16}. 
Dropping hereafter
the dependence on $F$ in $\verts(F),\edges(F), \Delta(F)$, and fixing
$X\in \cX_\Verts$ with $\hom(F,X)\ge (1+u) \Verts^\verts p^\edges$, it suffices 
to bound $I_p(X)$ below by the right hand side of \eqref{rough.LB}.
Since increasing any entry $X_{ij}\in [0,p)$ to $p$ decreases $I_p(X)$ while increasing $\hom(F, X)$, 
without loss of generality 
$U:= X-p\jay$ has non-negative entries, 
and thereby 
\cite[Corollary 3.5]{LuZh14} implies that for some absolute constant $p_0>0$ and all $p \in (0,p_0]$, 
\[
I_p(X) \ge \frac{1}{2} \|U\|^2_{\HS}
I_p( 1-1/ \log(1/p)) \ggs   \|U\|_{\HS}^2\log(1/p) \,. 
\]
Thus, it only remains to show that 
\begin{equation}	\label{rough.goal}
\|U\|_{\HS}^2 \ge  ((1+u)^{1/\edges} -1)^{\Delta}n^2 p^{\Delta}.
\end{equation}
To this end, recall  the following special case of Finner's generalized H\"older inequality 
\begin{equation}	\label{finner}
\Big[ \int_{\Omega_o^{\verts}} \prod_{\{a,b\}\in E} |f(x_a,x_b)| \mu_o(dx_1)\cdots \mu_o(dx_{\verts}) 
\Big]^{1/\edges} \le \|f\|_{L_{\Delta}(\Omega_o^2, \mu_o^{\otimes 2})}
\end{equation}
for any  graph $F=([\verts],E)$ with $|E|=\edges$ and maximal degree $\Delta$, any
probability space $(\Omega_o,\mu_o)$ and 
$f\in L_{\Delta}(\Omega_o^2, \mu_o^{\otimes 2})$ (cf.\ \cite[Theorem 4.1]{BGLZ16}). 
In particular, taking the uniform measure $\mu_o$ on $\Omega_o:=[\Verts]$ and 
identifying elements $Y\in \cX_\Verts$ with functions 
$\Omega_o^2 \ni (i,j)\mapsto Y_{ij}$, it follows from \eqref{finner} that  
\[
\Big(\frac1{\Verts^\verts}\hom(F,X)\Big)^{1/\edges} \le  \| X\|_{L_\Delta(\Omega_o^2,\mu_o^{\otimes2})} 
\le  p + \|U\|_{L_\Delta(\Omega_o^2,\mu_o^{\otimes2})} 
\le  p + \big(\Verts^{-1} \|U\|_{\HS} \big)^{2/\Delta} \,,
\]
where in the second bound we used the triangle inequality, and in the third that $U_{ij}\in [0,1]$ and $\Delta\ge2$. 
Combining with our assumption $\hom(F,X) \ge (1+u)\Verts^\verts p^\edges$ and rearranging yields \eqref{rough.goal} and hence the claim.
\end{proof}
\begin{remark} Since any independent set in the induced subgraph $F^\star$ 
on the vertices of degree $\Delta(F)$ is of size at most $\edges(F)/\Delta(F)$, 
it follows from \eqref{rough.UB}-\eqref{rough.LB}  that for $u$ sufficiently large,
\[
\upphi_{\Verts,p}(F,u)\asymp u^{\Delta(F)/\edges(F)} \Verts^2p^{\Delta(F)}\log(1/p) \,.
\]
\end{remark}

In preparation for lower bounding the large deviation rates for super-level sets of $\hom(H,\cdot)$, 
we first derive a general lower bound for the product Bernoulli($p$) measure $\mu_p$.
\begin{lemma}	\label{lem:LB}
Identifying $\cA_\Verts$ with $\{0,1\}^{\Verts \choose 2}$, let $\mu_Q$ 
denote the law induced on $\cA_\Verts$ by the product of Bernoulli($Q_{ij}$) measures.
For some $\cstar<\infty$, any $p\in (0,\frac12]$, 
$Q \in \cX_\Verts$ and $\cB\subset\cA_\Verts$, 
\begin{equation}\label{eq:LB-bernoulli}
\log\mu_p(\cB) \ge -I_p(Q) + \log\mu_Q(\cB) -  \frac{\Verts \log (\cstar/p)}{2 \sqrt{2 \mu_Q(\cB)}}  \,.
\end{equation}
\end{lemma}
\begin{proof} 
Since \eqref{eq:LB-bernoulli} trivially holds when $\mu_Q(\cB)=0$, we assume 
that $\mu_Q(\cB)>0$. For any probability measures 
$\Q \ll \pr$ 
with 
$H(\Q \| \pr) = \int_\cX  d\Q \log\frac{d\Q}{d\pr}$ finite, set
$Y=\log \frac{d\Q}{d\pr} - H(\Q \| \pr)$, noting that if $\cB \subseteq \cX$ with 
$\Q(\cB)>0$, then for any $s \ge 1$, by Jensen and H\"older inequalities,
\begin{align}
H(\Q \| \pr) + \log\pr(\cB)  &\ge H(\Q \| \pr) + \log \int_\cB \frac{d\pr}{\d\Q} d\Q   
= \log \int_{\cB} e^{-Y} d\Q \nonumber \\
&\ge  \log \Q(\cB) - \frac1{\Q(\cB)} \int_{\cB} Y d\Q \ge  \log \Q(\cB) 
-  \frac{\|Y\|_{L_s(\Q)}}{\Q(\cB)^{1/s}} \,.
\label{eq:ent-lbd}
\end{align}
In particular, for $\Q=\mu_Q$ and $\P=\mu_p$ one has that at any $A \in \cA_n$,
\begin{align*}
Y  &=  \sum_{i<j} (Q_{ij} - A_{ij})  \Big( \log \frac{(1-Q_{ij})}{(1-p)} - \log \frac{Q_{ij}}p \Big) =  \frac{1}{2} \langle  Q-A, M
\rangle_\HS\,,
\end{align*}
for $M \in \Sym_\Verts^0(\R)$ with entries 
\[
M_{ij} := \gamma_p(Q_{ij}) :=    \log \frac{(1-Q_{ij})}{Q_{ij}} - \log \frac{(1-p)}{p} \,.
\]
Thus, with $H(\mu_Q\|\mu_p)=I_p(Q)$ and
\begin{align*}
\| Y \|^2_{L_2(\mu_Q)}=  \sum_{i<j}  \Var_{Q} (X_{ij}) M_{ij}^2
= \sum_{i<j} Q_{ij}(1-Q_{ij}) \gamma_p(Q_{ij})^2,
\end{align*}
we arrive at \eqref{eq:LB-bernoulli} upon verifying that
\begin{equation}\label{eq:co-bd}
\sup_{x \in [0,1]} \{ 2 \sqrt{x(1-x)} |\gamma_p(x)|\} \le \log (\cstar/p) \,,
\end{equation}
for some $\cstar<\infty$ 
and all $p \in (0,\frac12]$.
\end{proof}
We have the following direct corollary for spectral-norm balls.
\begin{cor}\label{lem:LB.ball}
For some $C_0<\infty$, any $p \in (0,\frac{1}{2}]$ and all $Q \in \cX_\Verts$, 
\begin{equation}\label{lbd-ballQ}
\mu_p( \| A - Q \|_\op  \le C_0 \sqrt{\Verts}) \ge \frac{1}{2} e^{-I_p(Q)} e^{-\frac{\Verts}{2} \log(\cstar/p)} \,.
\end{equation}
Further, for $Q=r \jay_\Verts$ and $\frac{1}{\Verts} \log \Verts \le r \le p \le \frac{1}{2}$, the bound 
\eqref{lbd-ballQ} holds with $C_1 \sqrt{r}$ and $4 \sqrt{2p}$ replacing $C_0$ and  
$\log(\cstar/p)$, respectively.
\end{cor}
\begin{proof} Letting $\tAn$ have distribution $\mu_Q$, the
centered symmetric random matrix $\tAn-Q$
has independent entries bounded in modulus by one. It is standard that
for some $C_0<\infty$ and all $Q \in \cX_n$ 
\begin{equation}\label{eq:mean-op}
\e \|\tAn-Q\|_\op \le \frac{C_0}{2} \sqrt{\Verts}  
\end{equation}
(e.g. see  
\cite[Theorem 4.8]{LHY}).  Thus,
$\mu_Q(\cB^c) \le 1/2$ for $\cB=\{A
: \|A-Q\|_\op \le C_0 \sqrt{\Verts} \}$,
and \eqref{lbd-ballQ} follows from \Cref{lem:LB}. Considering
\Cref{lem:gap} (for $p=r$), we see that when $Q=r \jay_\Verts$ the bound
\eqref{eq:mean-op} holds with $C_0=C_1 \sqrt{r}$, hence the same applies in \eqref{lbd-ballQ}.
In addition, for $Q=r \jay_\Verts$ one needs only consider $x=r$ in \eqref{eq:co-bd}. Thus, 
with $r \le p \le \frac{1}{2}$ yielding $0 \le \gamma_p(r) \le \log (2p/r)$, we can replace the bound \eqref{eq:co-bd} by
\[
2\sqrt{r(1-r)} \gamma_p(r) \le 4 \sqrt{r} \log (\sqrt{2p/r}) \le 4 \sqrt{2p} \,,
\]
with the corresponding improvement in \eqref{lbd-ballQ}.
\end{proof}

\subsection{Proof of Theorem \ref{thm:hom}: lower bound}
\label{sec:hom.LB}
$~$
We proceed
to prove that for any fixed finite graph $H$, $0<\eps<1/2$ and $u \ge 2 \eps$,
assuming $1 \gg p^{2\Delta_\star(H)}\gg \Verts^{-1}\log\Verts$, one has
\begin{equation}	\label{hom.LB.goal}
\RR_{\Verts,p}(H,u) \ge (1-\eps) \upphi_{\Verts,p}(H, u-\eps)
\end{equation}
for all $\Verts$ sufficiently large depending only on $H$, $u$ and $\eps$. 
To this end, note first that if $H$ is the 
disjoint union of graphs $F$ and $F'$ with $\edges(F')=0$, then for any $X\in \cX_\Verts$,
\begin{equation}\label{eq:isolated}
\hom(H, X) = \Verts^{\verts(F')}\hom(F,X),
\end{equation}
so it suffices to establish \eqref{hom.LB.goal} for the graph $F$ (and  \eqref{hom.LB.goal} 
trivially holds if $\edges(H)=0$, i.e. $H=F'$). Assuming hereafter that $H$ has no isolated vertices, 
we prove \eqref{hom.LB.goal} by induction on $\edges(H)$. For 
$\edges(H)=1$, 
necessarily $H={\sf K}_{1,1}$, a single edge. Since $\hom({\sf K}_{1,1},\Anp)= \frac12\1^\tran \Anp\1$ is a linear 
functional on $\Anp$, the set $\{X\in \cX_n: \hom({\sf K}_{1,1},X)\ge (1+u) \Verts^2p\}$ is convex and
\eqref{hom.LB.goal}, even at $\eps=0$, is a consequence of \Cref{prop:cvx}. Upon mapping 
$u \mapsto (1+u)^{1/\edges(H)}-1$, this extends to any such graph with $\Delta(H)=1$, since then 
\begin{equation}	\label{K11tensorize}
\hom(H,X) = \hom({\sf K}_{1,1},X)^{\edges(H)} \,.
\end{equation}
Further, from \eqref{K11tensorize} and standard tail estimates for the Binomial(${n \choose 2},p$) law,
it follows that for some $c>0$, any $K \ge 2$ and all graphs $F$ with $\Delta(F)=1$,
\begin{equation}\label{lem:Delta1}
\pr\big( \Anp \notin \cL_F(K^{\edges(F)}) \big) 
= \pr(\hom({\sf K}_{1,1},\Anp) > K \Verts^2 p) \le e^{-c K \Verts^2 p} 
\end{equation}
(recall the definition \eqref{def:LH} of $\cL_F(\cdot)$).

Now assume $\Delta(H) \ge 2$ and that \eqref{hom.LB.goal} holds, 
with $H$ replaced by $F$, for every fixed graph $F$ with $\edges(F)<\edges(H)$, provided
$p^{2 \Delta_\star(F)} \gg n^{-1} \log n$. For 
$K_0,K_1 \ge 2$ to be chosen sufficiently large depending only on $H$ and $u$,
we invoke \Cref{prop:RLRG} with $K=K_0$, $\Delta=\Delta(H)$ and  
\begin{equation}\label{set.rnk}
\rnk = \lceil K_0 (p^{\Delta(H)}/\errRLs^2) \log(1/p) \rceil \,,
\end{equation}
to obtain a partition 
$\cA_\Verts=\bigsqcup_{j=0}^\Net \cE_j$ with the stated properties \eqref{RLRG:a}--\eqref{RLRG.diam}.
Aiming at \Cref{prop:countingB} with $K=K_1$ and $\eps_0=\eps/(K_1 C_H)$,
we take in turn 
\begin{equation}\label{set.delta}
\errRLs = \frac{\eps_0}{4 C_\star} p^{\Delta_\star(H)}  
\end{equation}
(for $C_\star$ of \Cref{prop:RLRG}). Towards applying  \Cref{cor:covering}
we specify the exceptional set as
\begin{equation}\label{def:cE}
\cE =\cE(H, \eps, K_0,K_1) := \cE_0  \cup \cE_H(K_1) ,
\end{equation}
where $\cE_H(K_1) := \cX_n \setminus \bigcap_{F\prec H}  \cL_F(K_1)$.
Denoting by $\cC_j \subseteq \cX_\Verts$, $j \ge 1$, the closed 
convex hull of $\cE_j$ and taking
\begin{equation}\label{def:cI}
\cI = \big\{ j\in [\Net]: \cC_j
\cap (\cA_\Verts \setminus \cE) \ne \emptyset\big\}\,,
\end{equation}
we have
\begin{equation}\label{claim1.hom}
\cA_\Verts \setminus \cE\subseteq \bigcup_{j\in \cI} \cC_j \,.
\end{equation}
Since $n p^{2\Delta_\star(H)-1} \to \infty$, clearly $\sqrt{n p} \le \errRLs n$ for our choice of  $\errRLs$ 
and all $n$ large enough, in which case, by \eqref{RLRG.diam} and the triangle inequality,  for any $j \ge 1$,
\[
\max_{X,Y \in \cE_j} \{ \|X-Y\|_{\op} \} \le 2 C_\star (\sqrt{np}+ \errRLs n) \le \eps_0 n p^{\Delta_\star(H)} \,.
\]
Further, by \eqref{def:cE} and \eqref{def:cI}, each of the convex sets $\cC_j$, $j \in \cI$,
intersects $\cap_{F\prec H} \cL_F(K_1)$. Thus, from 
\eqref{eq:bd-fluct} and our choice of $\eps_0 \le 1$ we deduce that
\[
\big| \hom(H,X)-\hom(H,Y) \big| \le 
\eps  \Verts^{\verts(H)}p^{\edges(H)}\,, \qquad \forall j\in \cI, \;\;  \forall X,Y\in \cC_j \,.
\]
Hence, we can apply \Cref{cor:covering} with $h=\Verts^{-\verts(H)}p^{-\edges(H)}\hom(H,\cdot)$ to obtain
\begin{align}	\label{from2.2}
\pr(\hom(H,\Anp)\ge (1+u) \Verts^{\verts(H)}p^{\edges(H)}) \le \mu_p(\cE) + \expo{ \log \Net
-\upphi_{\Verts,p}(H, u-\eps) }\,,
\end{align}
where in view of \eqref{set.rnk}, for $\eta_n:=K_0 \Verts^{-1} \errRLs^{-2} \log(\Verts/\errRLs)$, 
\begin{align}	\label{bd.J}
\log \Net \le \rnk (\Verts+2) \log (\frac{3 \Verts}{\errRLs} ) 
\,\ls\, \eta_n \, \Verts^2 p^{\Delta(H)} \log(1/p) \,.
\end{align}
Next, note that $\eta_n \to 0$ thanks to \eqref{set.delta} and
our assumption that $p^{2\Delta_\star(H)}\gg \Verts^{-1}\log\Verts$. Hence, comparing 
the right hand sides of \eqref{bd.J} and \eqref{rough.LB} with $F=H$, we have that 
\begin{equation}\label{bd.Nat}
\log \Net \ll \upphi_{\Verts,p}(H,u-\eps).
\end{equation}
Further, from \eqref{RLRG:b}, if $K_0=K_0(H,u)$ is large enough, then by  \eqref{rough.UB}, 
\[
\mu_p(\cE_0) \le 4 e^{- K_0 \Verts^2p^{\Delta(H)}\log(1/p)} \ll e^{-2 \upphi_{\Verts,p}(H,u)} \,.
\]
Combining this with 
\eqref{from2.2}, \eqref{bd.Nat}, \eqref{def:cE} and the union bound, it suffices to show that 
\begin{equation}	\label{hom.remains}
\pr(\Anp \in \cE_H(K_1)) \le \sum_{F\prec H} \pr(\Anp \notin \cL_F(K_1)) \ll e^{-2\upphi_{\Verts,p}(H,u+\eps)} \,.
\end{equation}
Turning to this task, note that 
$\Delta(F) \le \Delta(H)$ and $\Delta_\star(F)\le \Delta_\star(H)$ 
for each $F\prec H$. Since $H$ has no isolated vertices 
we additionally have $\edges(F)<\edges(H)$. Considering
\eqref{lem:Delta1} at $K=K_1^{1/\edges(F)}$ (in case $\Delta(F)=1$), or the induction hypothesis 
\begin{equation}	\label{LF.bd}
\pr( \Anp\notin \cL_F(K_1)) \le e^{- (1-\eps) \upphi_{\Verts,p}(F,K_1-1-\eps)} \,,
\end{equation}
in case $\Delta(F)\ge2$, together with \eqref{rough.LB} and \eqref{rough.UB},
it follows that 
\[
\pr( \Anp\notin \cL_F(K_1)) \ll e^{-2\upphi_{\Verts,p}(H,u+\eps)} \,,
\] 
for some $K_1 = K_1(F,H,u)$
and all $n$ large enough. Taking the maximal value of $K_1$ 
among the finitely many $F \prec H$, completes the proof of 
\eqref{hom.remains} and thereby also of \eqref{hom.LB.goal}.

\subsection{Proof of \Cref{thm:hom}: upper bound}\label{sec:hom.UB}
We complete the proof of \Cref{thm:hom}  by complementing
\eqref{hom.LB.goal} with an asymptotically matching upper bound. Specifically, under the preceding 
assumptions on $H,u,\eps$ and $p=p(n)$,
for all $\Verts$ sufficiently large depending only on $H,u,\eps$, we have that 
\begin{equation}	\label{hom.UB.goal} 
\RR_{\Verts,p}(H,u) \le (1+\eps) \upphi_{\Verts,p} (H, u+ \eps) \, . 
\end{equation}
To this end, in view of \eqref{eq:isolated}--\eqref{K11tensorize} it again suffices to 
consider $H$ with no isolated vertices and for $\Delta(H)=1$ to deal only with $H={\sf K}_{1,1}$.
Further, by convexity of $I_p(\cdot)$, for any $p \ll 1$
\[
\upphi_{\Verts,p}({\sf K}_{1,1},u) = {\Verts \choose 2} I_p ((1+u)p) = \frac{1}{2}
\{ (1+u) \log (1+u) -u \} \Verts^2 p
(1+o(1)) \,,
\]
whereby \eqref{hom.UB.goal} follows from standard Binomial tail estimates.
Assuming hereafter that $H$ has maximal degree $\Delta(H) \ge 2$ and no isolated vertices, 
fix $K_1=K_1(H,u,\eps)$ 
so \eqref{hom.remains} holds, setting
$\cB_Q= \{ X\in \cX_\Verts: \|X-Q\|_\op \le C_0\sqrt{\Verts}\}$ (for $C_0$ of \Cref{lem:LB.ball}),
and $\cQ=\{Q\in \cX_\Verts: \cB_Q\subseteq \cE_H(K_1)\}$.
For any $Q\in \cQ$, we get by \eqref{lbd-ballQ}, \eqref{hom.remains} and monotonicity that 
\begin{align}\label{lbd-neglig-ball}
e^{-2\upphi_{\Verts,p}(H,u+\eps)} &\gg \pr(\Anp \in \cE_H(K_1)) 
 \ge  \mu_p(\cB_Q) \ge \frac{1}{2} e^{-I_p(Q)-\eta_n}  \,,
\end{align}
where $\eta_n := \Verts \log(\cstar/p) \ll \upphi_{\Verts,p}(H,u+\eps)$, 
since $n p^{\Delta(H)} \gg 1$ (see \eqref{rough.LB}). From \eqref{lbd-neglig-ball} we thus deduce
that $\inf\{I_p(Q): Q \in \cQ\} > \upphi_{\Verts,p}(H,u+\eps)$ and thereby for any $\eps\in (0,1)$,
\begin{equation}	\label{LB.except}
\upphi_{\Verts,p}(H,u+\eps) = \inf\{ I_p(Q) : Q \in \overline{\cL_H(1+u+\eps)^c}\cap\cQ^c\} \,.
\end{equation}
Setting again $\eps_0=\eps/(K_1 C_H)$, recall that $n p^{2\Delta_\star(H)} \gg 1$, hence
$C_0 \sqrt{n} \le \frac{\eps_0}{2} n p^{\Delta_\star(H)}$ for all $n$ large enough, in which case the 
convex set $\cB=\cB_Q$ satisfies \eqref{op-diam-cB}. By definition, for $Q \in \cQ^c$ such 
$\cB_Q$ also satisfies \eqref{homF-min}, hence it follows from \eqref{eq:bd-fluct} 
and our choice of $\eps_0$ that $\cB_Q\subset \overline{\cL_H(1+u)^c}$ for any 
$Q\in \overline{\cL_H(1+u+\eps)^c}\cap\cQ^c$. Thus, combining
\eqref{lbd-ballQ} (as in the right side of \eqref{lbd-neglig-ball}), and \eqref{LB.except}, we arrive at
\[
\pr(\Anp\in \overline{\cL_H(1+u)^c}) \ge \sup\{ \mu_p(\cB_Q) :  Q\in \overline{\cL_H(1+u+\eps)^c}\cap\cQ^c \} 
\ge \frac{1}{2} e^{- \upphi_{\Verts,p}(H,u+\eps) - \eta_n} \,,
\]
from which \eqref{hom.UB.goal} follows upon taking logarithms on both sides.

\section{Proof of Theorem \ref{thm:cycles}: Upper and lower tails for cycle counts}
\label{sec:cycles}

As we show next, Theorem \ref{thm:cycles} is  
a straightforward consequence of 
the following non-asymptotic tail bounds. 
\begin{theorem}[Quantitative large deviations for cycle counts]
\label{thm:cycles.quant} 
There are constants $c>0$ and $C'<\infty$ such that for any 
integer $\ell\ge3$, $\Verts^{-1/2}\le p\le 1/2$, $K\ge 2$,
$1\le \rnk\le \Verts$, we have for any $u>0$,
\begin{align}
&\pro{ \hom(\Cyc_\ell, \Gnp) \ge (1+u) \Verts^\ell p^\ell}  \le	
\expo{ - \upphi_{\Verts,p}(\Cyc_\ell, u - \eps_{\fluct} \big) + \Ecomplex } 
+ \pexcep \,,	\label{cyc.upper}
\end{align}
and for any $0\le u \le 1$,
\begin{align}
&\pro{ \hom(\Cyc_\ell, \Gnp) \le (1-u) \Verts^\ell p^\ell}  \le	
\expo{ - \uppsi_{\Verts,p}(\Cyc_\ell, u - \eps_{\fluct} \big) + \Ecomplex } 
+ \pexcep \,,	\label{cyc.lower}
\end{align}
where the \emph{fluctuation term} is
$\eps_{\fluct} = 3 \eps^\ell$ with
\begin{align}
\eps(K,\rnk ) :=  \frac{C'}{\Verts^{1/2-1/\ell}p^{1/2}} + 
\frac{\kappa_\ell K}{\rnk^{1/2-1/\ell}}	
\label{def:eps+}
\end{align}
and $\kappa_\ell$ as in Lemma \ref{lem:HSk-Salpha-tail},
the \emph{complexity term} is
\begin{align}
\Ecomplex (\rnk) &= O(\ell \rnk\Verts\log\Verts),	\label{def:Ecomplex}
\end{align}
and $\pexcep=\pexcep (K)$ is the 
exceptional probability from \eqref{def:excep.pm}.
\end{theorem}
The bounds \eqref{cyc.upper}--\eqref{cyc.lower} are the result of applying Corollary \ref{cor:covering} with a covering $\{\cB_i\}_{i\in \cI}$ of
$\cA_\Verts \cap \cG(K \Verts p, C' \sqrt{\Verts p})$,
throughout which the corresponding bound \eqref{tail.refined} holds. 
Thanks to Proposition \ref{prop:tail.refined}, the $\mu_p$-probability 
of its complement, exceptional set $\cE$, is at most $\pexcep$. 
The error term $\Ecomplex$ is $\log|\cI|$, which in our case is basically 
the metric entropy of $\Sym_{\Verts,\rnk }([0,1])$.

\begin{proof}[Proof of Theorem \ref{thm:cycles}]
Starting with \eqref{eq:upper-tail.cyc},
the first term in the definition \eqref{def:eps+}
of $\eps(K,\rnk )$ is $o(1)$ as long as $p\gg \Verts^{2/\ell-1}$. 
Fixing an arbitrarily slowly growing function $W=W(\Verts)$, we take 
\begin{equation}	\label{take.KR}
K= (W^2 \log(1/p))^{1/2}\,, \qquad \rnk= (W^4 \log\Verts)^{\ell/(\ell-2)} \, \,.
\end{equation}
Since $\Verts p \ge 1$, with these choices we have that 
\[
\frac{K}{\rnk^{1/2-1/\ell}} = \Big(\frac{W^2 \log(1/p)}{W^4 \log\Verts}\Big)^{1/2} \le W^{-1} =o(1)\,,
\]
hence also $\eps_{\fluct}=o(1)$. Furthermore, for such $K$,
\begin{equation}	\label{cycpf1}
\pexcep=C \exp(-cK^2\Verts^2p^2) = \exp(-\omega(\Verts^2p^2\log(1/p)) \,.
\end{equation}
For
\[
\frac{\big(W^6 \log\Verts\big)^{\frac{\ell}{2\ell -4}}}{\sqrt{\Verts}} \le p\le \Verts^{-1/10}
\]
we have
\begin{equation}\label{bd:large-ell}
\frac{\ell  \rnk\Verts\log\Verts}{\Verts^2p^2\log(1/p)} \asymp \frac{\rnk}{\Verts p^2}  
\le W^{-2} = o(1)\,,
\end{equation}
whereas for $\Verts^{-1/10}\le p\ll1$, 
\[
\frac{\ell  \rnk\Verts\log\Verts}{\Verts^2p^2\log(1/p)} \le\frac{\ell  \rnk\log\Verts}{\Verts p^2}\le \frac{\ell  \rnk\log\Verts}{\Verts^{0.8}} =o(1).
\]
To conclude the proof of \eqref{eq:upper-tail.cyc} it
remains to dominate the error \eqref{cycpf1} by
the first term on the \abbr{rhs} of \eqref{cyc.upper}, for which it suffices to show 
the analogue of \eqref{R3.upper}, namely
\begin{equation}	\label{cycpf2}
\upphi_{\Verts,p}(C_\ell, u-\eps_{\fluct}) \lls \Verts^2p^2\log(1/p)
\end{equation}
for any fixed $u\ge0$. 
While we could appeal to Theorem \ref{thm:BGLZ}, it is easy to verify \eqref{cycpf2} directly. That is, for the projection $\id_{[\Verts_0]}$ 
to the first $\Verts_0$ coordinates, consider
the matrix
\begin{equation}	\label{X.planted}
X_{\star} = p(\1\1^\tran-\id) + (1-p)(\1_{[\Verts_0]}\1_{[\Verts_0]}^\tran-\id_{[\Verts_0]}) \in \cX_\Verts \,.
\end{equation}
As $I_p(X_{  \star}) = {\Verts_0\choose 2} \log(1/p)$, taking $\Verts_0=\lf a\Verts p\rf$ for fixed $a=a(u)>0$ to be chosen gives 
\[
I_p(X_{\star}) \,\ls_{u} \, \Verts^2p^2\log(1/p)\,.
\]
Moreover, for any fixed $\ell \in \N$,
\[
\Tr X_{\star}^\ell \ge \Tr (\1_{[\Verts_0]}\1_{[\Verts_0]}^\tran-\id_{[\Verts_0]})^\ell = 
\frac{\Verts_0!}{(\Verts_0-\ell)!}
 = (a\Verts p-O(1))^\ell \,.
\]
With $p \ll 1$, we can take $a=(2(1+u))^{1/\ell}$, yielding that 
\[
\upphi_{\Verts,p}(\Cyc_\ell,u- \eps_{\fluct}) \le \upphi_{\Verts,p}(\Cyc_\ell, u) \le I_p(X) \,\ls_u\, \Verts^2p^2\log(1/p),
\]
as claimed in \eqref{cycpf2}. 
Turning to prove \eqref{eq:lower-tail.cyc}, let
\begin{equation}\label{KR-minus:def}
K=W p^{-1/2}\,,\qquad \rnk= \big(W^4/p\big)^{\ell/(\ell-2)}\, .
\end{equation}
By \eqref{prange.cyc.lower} we have that $p\gg \Verts^{2/\ell-1}$
(since $\frac{\ell-2}{2\ell-2} \le 1 - \frac{2}{\ell}$). Hence, from
\eqref{def:eps+} and \eqref{KR-minus:def}, 
\[
\eps(K,\rnk ) =  \frac1{\Verts^{1/2-1/\ell}p^{1/2}} + \frac{1}{W}=o(1) \,,
\]
yielding that $\eps_{\fluct}=o(1)$. Further, from  \eqref{def:excep.pm} 
we now have that
\[
\pexcep= C \exp(-cK^2\Verts^2p^2)=\exp(-\omega(\Verts^2p)) \,.
\]
Next, assuming
\[
p \ge \left( \frac{\log\Verts}{\Verts}\right)^{\frac{\ell-2}{2\ell-2}}\, 
W^{3\ell/(\ell-1)} \,,
\]
it follows that
\[
\frac{\rnk\Verts\log\Verts}{\Verts^2p} = \frac{(\log\Verts)\, W^{4\ell/(\ell-2)}}{\Verts p^{(2\ell-2)/(\ell-2)}} 
\le W^{\frac{4\ell}{\ell-2}- \frac{6\ell}{\ell-2}}\le W^{-2} = o(1)
\]
and so it only remains to show that
\begin{equation}	\label{cycpf3}
\uppsi_{\Verts,p}(C_\ell, u - \eps_{\fluct}) \,\ls_{u}\, \Verts^2 p \,.
\end{equation}
For this consider the matrix $X_o = b p \jay_\Verts \in \cX_\Verts$
for some fixed $b=b(u)\in [0,1]$. Clearly
$I_p(X_o) = {\Verts\choose2} I_p(bp)$, whereas since $p=o(1)$,
\[
I_p(bp) \sim p(b\log b - b + 1) \,.
\]
Thus $I_p(X_{o}) \lls_{u} \Verts^2p$. Moreover,
\[
\Tr X_{o}^\ell =  (bp)^\ell \frac{\Verts!}{(\Verts-\ell)!}\,,
\]
which for $b=(1-u)^{1/\ell}$ yields that
$
\uppsi_{\Verts,p}(\Cyc_\ell, u - \eps_{\fluct}) \le \uppsi_{\Verts,p}(\Cyc_\ell, u) \le 
I_p(X_o) 
\lls_{u} \Verts^2p
$
as claimed in \eqref{cycpf3}. 
\end{proof}

\subsection{Constructing a net}

For $\rnk\in \N$ let 
\begin{equation}
\Lambda_\rnk = \big\{ \, \bs{\lambda}=(\lambda_1,\dots, \lambda_\rnk)\in \R^\rnk : |\lambda_1|\ge |\lambda_2| \ge \cdots\ge |\lambda_\rnk|\,\big\}
\end{equation}
and for $L<\infty$ write
\begin{equation}
\Lambda_\rnk(L) = \{ \bs{\lambda}\in \Lambda_\rnk: \|\bs{\lambda}\|_2 \le L \}.
\end{equation}
For $1\le \rnk\le \Verts$ we denote the Stiefel manifold $\St(\Verts,\rnk )$ of ordered orthonormal bases for sub-spaces of $\R^\Verts$ of dimension $\rnk$ by
\[
\St(\Verts,\rnk ) = \big\{\, \bs{u} = (u_1,\dots, u_\rnk) \text{ orthonormal in } \R^\Verts\,\big\}.
\]
We denote a mapping
\begin{equation}	\label{def:matmap}
\Mat: \Lambda_\rnk\times \St(\Verts,\rnk )\to \Sym_{\Verts,\rnk }(\R), \qquad \Mat(\bs{\lambda},\bs{u}) = \sum_{j\le\rnk} \lambda_j u_ju_j^\tran,
\end{equation}
which is a surjection by the spectral theorem (recall from Section \ref{sec:notation} that $\Sym_{\Verts,\rnk }(\R)$ is the set of symmetric $\Verts\times\Verts$ matrices of rank at most $\rnk$).
We equip $\Lambda_\rnk$ and $\St(\Verts,\rnk )$ with the Euclidean metrics (where elements are naturally associated to points in $\R^\rnk$ and $\R^{\Verts\rnk}$, respectively). 
\begin{lemma}
\label{lem:net}
For any $1\le \rnk\le \Verts$ and $\delta\in (0,1]$, there
exist $\delta$-nets $\Sigma\subset \Lambda_\rnk(\Verts)$, $\cV\subset \St(\Verts,\rnk )$ (with respect to the Euclidean metrics) of size
\begin{equation}\label{bd:Vnet}
|\Sigma| \le \exp( \rnk\log(2\Verts^2/\delta)),
\qquad |\cV| \le  \exp(\rnk\Verts\log(3\sqrt{\rnk}/\delta)). 
\end{equation}
Furthermore, if $X=\Mat(\bla,\bu)\in \Sym_{\Verts,\rnk }(\R)\cap \ball_\HS(\Verts)$ and $Y=\Mat(\bmu,\bv)\in \Sigma\times \cV$ is such that $\|\bla-\bmu\|_2,\|\bu-\bv\|_\HS\le \delta$, then
\begin{equation}	\label{HS.bd}
\|X-Y\|_\HS\le 3 \Verts\delta \,.
\end{equation}
\end{lemma}

\begin{remark}\label{rm:delta-HS-net}
From \eqref{HS.bd} we have
that $\Mat(\Sigma\times \cV)$ is a $3 \Verts\delta$-net for $\Sym_{\Verts,\rnk }(\R)\cap \ball_\HS(\Verts)$ in the Hilbert--Schmidt metric.
In the proof of Theorem \ref{thm:cycles.quant} it will be convenient to separately approximate the spectrum and the eigenbasis of rank $\rnk$ projections of matrices $A \in \cA_\Verts$, which is why we have defined the net in terms of $\Sigma$ and $\cV$.
\end{remark}

\begin{proof}
For $\Sigma$ of the specified size we intersect $\Lambda_\rnk(\Verts)$ with 
the $\rnk$-th Cartesian power of an $\delta/(2\Verts)$-mesh of the interval $[-\Verts,\Verts]$.
As for $\cV$, since a $\delta$-separated subset
of a metric space 
which is maximal under set inclusion must be a $\delta$-net,
a standard volume argument yields the existence 
of a $\delta$-net
of size at most $(1+2r/\delta)^d$ for the ball $\ball_2(r) \subset \R^d$.
Recalling that
$\St(\Verts,\rnk )$ is 
a subset of the ball $\ball_2(\sqrt{\rnk})$
 in $\R^{\Verts\times \rnk}$,
yields for $\rnk \ge 1 \ge \delta$
a $\delta$-net for $\St(\Verts,\rnk)$, whose size is bounded as in \eqref{bd:Vnet}.

Turning to show \eqref{HS.bd}, by
the triangle inequality and Cauchy--Schwarz,
\begin{align*}
\|X-Y\|_\HS &\le  \sum_{j\le\rnk} \big\| \lambda_j u_ju_j^\tran -
\mu_j  v_jv_j^\tran \Big\|_\HS \\
& \le \sum_{j\le\rnk}  |\lambda_j-\mu_j |\, \|u_ju_j^\tran \|_\HS 
+ \sum_{j \le\rnk} |\mu_j | \, \| u_ju_j^\tran - v_jv_j^\tran \|_\HS \\
&\le \sqrt{\rnk} \|\bs{\lambda}-\bs{\mu}\|_2 + \|\bs{\mu}\|_2 \Big( \sum_{j\le\rnk} \big\| u_ju_j^\tran - v_jv_j^\tran \big\|_\HS^2 \Big)^{1/2}\\
&\le \sqrt{\rnk} \delta + \Verts \Big( \sum_{j\le\rnk}\big\| u_ju_j^\tran - v_jv_j^\tran \big\|_\HS^2 \Big)^{1/2}\,.
\end{align*}
Next note that for any $u,w\in \R^\Verts$,
\begin{equation}\label{basic-iden}
2\|u-w\|_2^2 - \|uu^\tran - w w^\tran \|_\HS^2 
= 2 (\langle u,w\rangle - 1)^2 - (\|u\|_2^2-1)^2 - (\|w\|_2^2-1)^2 \,,
\end{equation}
which is non-negative for pairs of unit vectors such as $u_j,v_j$. Summing over
$1\le j\le\rnk$ gives
\[
\sum_{j\le\rnk}\big\| u_ju_j^\tran - v_jv_j^\tran \big\|_\HS^2  \le 2\sum_{j\le\rnk} \|u_j-v_j\|_2^2 = 2\|\bs{u}-\bs{v}\|_\HS^2 \le 2 \delta^2\,.
\]
Consequently, 
$\|X-Y\|_\HS \le (1+\sqrt{2}) \Verts\delta \le 3 \Verts\delta$, as claimed.
\end{proof}

\subsection{Proof of Theorem \ref{thm:cycles.quant} }

Fix $\ell > 2$ and $\rnk\in [\Verts]$. For $X\in \Sym_\Verts(\R)$ recall the decomposition 
$X=X_{\le \rnk}+X_{>\rnk}$ of \eqref{def:cutoff}, omitting hereafter 
the subscript $\rnk$, with the induced parameters
\begin{equation}
\bu_\le(X) = (u_1,\dots, u_\rnk)\in \St(\Verts,\rnk), \qquad \bs{\lambda}_\le(X) = (\lambda_1,\dots, \lambda_\rnk)\in \Lambda_\rnk.
\end{equation}
In order to apply Corollary \ref{cor:covering} for
\[
h_\ell(X) = (\Verts p)^{-\ell} \Tr X^\ell \,,
\]
we specify for $\eps = \eps(K,\rnk)$ of \eqref{def:eps+},
the ``exceptional" set 
\begin{equation}
\cE(\eps) := \big\{ X\in \ball_\HS(\Verts) : \|X_>\|_{S_\ell} > \eps \Verts p \big\}\,.
\end{equation}
Then, for the covering by closed convex sets,
let $\Sigma$ and $\cV$ be as in Lemma \ref{lem:net} and
for each $X \in \ball_\HS(\Verts)$ choose any
$\by(X)=(\bs{\mu}(X), \bv(X))\in \Sigma\times \cV$  such that 
\begin{equation}\label{eq:net-dist}
\|\bs{\lambda}_\le(X) - \bmu(X)\|_2 \le \delta \qquad
\mbox{and}
\qquad \|\bu_\le(X)-\bv(X)\|_\HS \le \delta \,.
\end{equation}
Setting $\delta' = 5 \delta \sqrt{\Verts}$,
for each $\by=(\bmu,\bu)\in \Sigma\times \cV$ consider the neighborhoods of $\Mat(\by)$,
\begin{align}
\cB_\by(\eps) :=\big\{
\Mat(\by)+W+Z  : \; W \in \ball_\HS(\delta' n),\;  Z \in \cZ_{\by}(\eps) \, \big\} \cap 
 \ball_\HS(n) \,, 
\label{def:By}
\end{align}
where for each $\eps>0$ the set
\[
\cZ_{\by}(\eps) := \big\{ Z \in \Sym_\Verts(\R) : \Im(Z)\subseteq \ker(\Mat(\by)),\;
 \|Z\|_{S_\ell}\le \eps \Verts p  \,\big\},
\]
consists of symmetric matrices controlled in $S_\ell$-norm, whose image as linear operators 
is orthogonal to that of $\Mat(\by)$.  The set $\cZ_{\by}(\eps)$ is convex, by 
convexity of the $S_\ell$-norm and the linear subspace $\ker(\Mat(\by)) \subset \R^n$,
hence $\cB_\by(\eps)$, being the intersection of the convex set $\ball_\HS(\Verts)$ with the 
translation by $\Mat(\by)$ of the convex sum-set $\ball_\HS(\delta'\Verts) + \cZ_{\by}(\eps)$, 
is also convex. In the following claims we show that the sets $\cB_\by(\eps)$,
with $\by$ ranging over the net $\Sigma\times \cV$, cover all of $\ball_\HS(\Verts)\supseteq\cX_\Verts$ but the exceptional set $\cE(\eps)$, and that the function $h_\ell$ is essentially constant on each $\cB_\by(\eps)$
 when $\eps$ is small. 
The intuition is that $h_\ell(X)$ is most sensitive to perturbations of $X$ in the directions $\bu_\le(X)$, so we approximate $X_\le$ with high precision in the Hilbert--Schmidt norm, while the remainder $X_>$ (approximated by $Z$) only needs to be of size $o(\Verts p)$ in $S_\ell$-norm (one should think of $\eps$ as being arbitrarily small but fixed while $\delta'=\delta'(\Verts)$ goes to zero polynomially fast).
\begin{claim}
\label{claim1}
For any $\ell > 2$, $\eps>0$, $\delta>0$ and 
$X \in \ball_\HS(\Verts) \cap \cE(\eps)^c$ we have 
$X\in \cB_{\by(X)}(\eps)$.
\end{claim}

\begin{claim}
\label{claim2}
For any $\ell > 2$, $\eps>0$,
$\delta' \le \Verts^{-2\ell}$,
$\by\in \Sigma\times\cV$ and $X\in \cB_\by(\eps)$, 
\[
|h_\ell(X)-h_\ell(\Mat(\by))| \le \eps^\ell + O(\Verts^{-\ell}) \, .
\]
\end{claim}

We defer the proofs of these claims to subsequent subsections and conclude the proof of Theorem \ref{thm:cycles.quant}.
From Claim \ref{claim1} we have that for any $\eps>0$, 
\[
\ball_\HS(\Verts) \setminus  \cE(\eps)\subseteq  \bigcup_{\by\in \Sigma\times \cV} \cB_\by(\eps).
\]
From Claim \ref{claim2} and the triangle inequality we have for 
any $\by\in \Sigma\times \cV$, $\eps>0$ and $X,X'\in \cB_\by(\eps)$,
\begin{equation}	\label{hell.fluct}
|h_\ell(X)-h_\ell(X')|\le 2\eps^\ell + O(\Verts^{-\ell}).
\end{equation}
It is easy to check that for $L=K \Verts p$, $D=C' \sqrt{\Verts p}$ and 
$\eps=\eps(K,\rnk)$, 
\[
\Verts^{1/\ell} D  + \kappa_\ell \, L \, \rnk^{1/\ell-1/2} = \eps \Verts p \,.
\]
Hence, by \eqref{tail.refined} we have that 
$\|\Anp_{>}\|_{S_\ell} \le \eps \Verts p$ on the event that
$\Anp \in \cG(L,D)$. 
From \Cref{prop:tail.refined},
the latter holds up to $\mu_p$-probability  
$\pexcep (K)$ of \eqref{def:excep.pm}. In particular,
\begin{equation*}
\mu_p(\cE(\eps)) \le \pexcep (K) \,.
\end{equation*}
Further, for such $\eps$
the \abbr{rhs} of \eqref{hell.fluct} is
controlled by $\eps_{\fluct} = 3 \eps^\ell$. Thus, \eqref{cyc.upper} and \eqref{cyc.lower}
follow by applying Corollary \ref{cor:covering} 
for $h_\ell$, with $t=1\pm u$, $\cE=\cA_\Verts \cap \cE(\eps)$ 
and $\{\cB_i\}_{i\in \cI}=\{\cB_\by(\eps)\}_{\by\in \Sigma\times \cV}$.
\qed

\subsection{Proof of Claim \ref{claim1}}
Fix $X\in \ball_\HS(\Verts) \cap \cE(\eps)^c$ with spectral decomposition 
\[
X = \sum_{j=1}^\Verts \lambda_j u_ju_j^\tran 
\]
and write $Y:= \Mat(\bmu,\bv)$ for $(\bmu,\bv)=\by(X)$
 (with notation as in \eqref{def:matmap}). Consider the matrix
$V$ with columns $v_1,\dots, v_\rnk$ and the corresponding projection matrix
$\Pi=\id-VV^\tran$ onto $\Span(v_1,\dots, v_\rnk)^\perp$. Evidently,
$\Im(Z)\subseteq \Span(v_1,\dots, v_\rnk)^\perp \subseteq\ker(Y)$
for $Z := \Pi X_> \Pi$.
Proceeding to establish \eqref{def:By} for $X$, $Y$ and $Z$, upon
applying \eqref{Schatten.op}, 
our assumption that $X \notin \cE(\eps)$ yields that
\[
\|Z\|_{S_\ell} \le \|\Pi\|_\op^2 \|X_>\|_{S_\ell}\le \|X_>\|_{S_\ell} \le \eps \Verts p \,.
\]
Further, setting $w_j=\Pi u_j$, we have by the triangle inequality and Cauchy--Schwarz, that
\begin{align}
\|X_>-Z\|_\HS 
 = \big\| \sum_{j>\rnk} \lambda_j (u_ju_j^\tran - w_j w_j^\tran ) 
\big\|_\HS  
& \le \sum_{j>\rnk} \, |\lambda_j| \, \|u_ju_j^\tran - w_j w_j^\tran\|_\HS \nonumber \\
&\le 
\|X_>\|_{\HS} \,
\big( \sum_{j>\rnk} \|u_ju_j^\tran - 
w_j w_j^\tran\|_\HS^2\big)^{1/2}.
\label{eq:bd-am1}
\end{align}
Recall \eqref{basic-iden}, that 
$\|uu^\tran -  w w^\tran\|_\HS^2 \le 2 \|u -w \|_2^2$
whenever $\|w\|_2^2 = \langle u, w \rangle$ and $\|u\|_2=1$. 
With $\Pi$ a projection matrix, this applies for $w_j = \Pi u_j$
and since $\|V\|_{\op} = 1$, yields the bound 
\begin{align}\label{eq:bd-am2}
\|u_j u_j^\tran -  w_j w_j^\tran\|_\HS^2 \le
 2\| (\id - \Pi)u_j\|_2^2 = 2 \| V V^\tran u_j \|_2^2
 \,\le \, 2\|V^\tran u_j \|_2^2 \,.
\end{align}
Further, denoting by
$U$ the matrix of columns $u_1,\dots, u_\rnk$, as
$\{u_j\}$ are orthonormal, $U^\tran u_j =0$ for any $j>\rnk$ and
from \eqref{eq:net-dist} we deduce that 
\begin{equation}\label{eq:bd-am3}
\|V^\tran u_j\|_2 = 
\|(V-U)^\tran u_j \|_2 \le \|V-U\|_\HS \le \delta \,.
\end{equation}
Combining \eqref{eq:bd-am1}--\eqref{eq:bd-am3}, and recalling that
$\|X_>\|_\HS\le \|X\|_\HS\le \Verts$ for $X\in \ball_\HS(\Verts)$, yields
\begin{equation}\label{W.above}
\|X_>-Z\|_\HS \le \|X_>\|_\HS  \sqrt{2\Verts} \delta \le 2 \Verts^{3/2} \delta \,.
\end{equation}
Finally, by the triangle inequality and \eqref{HS.bd} we have that
\[
\|X-Y-Z\|_{\HS} \le \|X_\le - Y\|_\HS + \|X_>-Z\|_\HS \le  
3 \Verts\delta +\|X_>-Z\|_\HS \,.
\]
In view of \eqref{W.above}, we see that 
$\|X-Y-Z\|_\HS \le 5 \Verts^{3/2} \delta$ as desired for \eqref{def:By}.
\qed

\subsection{Proof of Claim \ref{claim2}}
$~$
For $Y=\Mat(\by)$ and 
$X \in \cB_\by(\eps)$, let $Z$ be as in \eqref{def:By}.
Considering \eqref{claim2.1} for matrices $X$, $Y+Z$, we get
by the monotonicity of $\ell \mapsto \|\cdot\|_{S_\ell}$ that 
\begin{align}\label{eq:err-eps}
|\Tr X^\ell - \Tr (Y+Z)^\ell| 
 & \le \ell \|X-Y-Z\|_\op ( \|X\|_{S_{\ell-1}}^{\ell-1} 
 + \|Y+Z\|_{S_{\ell-1}}^{\ell-1}) \notag \\
&  \le \ell \, \|X-Y-Z\|_\HS \, ( \|X\|_{\HS}^{\ell-1} + \|Y+Z\|_{\HS}^{\ell-1}) \,.
\end{align}
Further, $\|X-Y-Z\|_\HS\le \delta'\Verts$ by \eqref{def:By}, and
$\|X\|_\HS \le \Verts$. Thus, with $\delta' \le 1$,
\[
\|Y+Z\|_{\HS} \le \|X\|_{\HS} +  \|X-Y-Z\|_{\HS} \le \Verts+ \delta'\Verts
\le  2\Verts\,.
\]
Along with \eqref{eq:err-eps}, the preceding yields
\begin{equation}\label{TrXY+Z}
|\Tr X^\ell - \Tr (Y+Z)^\ell| \le \ell \delta'\Verts
(\Verts^{\ell-1} + (2\Verts)^{\ell-1}) \le \delta' (4\Verts)^\ell \,.
\end{equation}
Since $\Im(Z) \subseteq \ker(Y)$, with $Y,Z \in \Sym_\Verts(\R)$, 
we have that $YZ=ZY=0$, and hence
\begin{equation}	\label{TrYZ}
\Tr(Y+Z)^\ell = \Tr Y^\ell + \Tr Z^\ell.
\end{equation}
Since $\delta' \le \Verts^{-2 \ell}$, we see that
\[
|\Tr X^\ell - \Tr Y^\ell| \le |\Tr Z^\ell| + O(\Verts^{-\ell})   
\le \|Z\|_{S_\ell}^\ell +  O (\Verts^{-\ell}) \le (\eps \Verts p)^\ell + O(\Verts^{-\ell}) \,,  
\]
and the claim follows from dividing through by $(\Verts p)^\ell \ge 1$.
This concludes the proof of Claim \ref{claim2} and hence of Theorem \ref{thm:cycles.quant}.

\subsection{Proof of Proposition \ref{prop:Schatten}}
\label{sec:Schatten}
Fix $\alpha \in (2,\infty]$ and for $X\in\Sym_\Verts(\R)$ denote
\[
g_\alpha (X) = (\Verts p)^{-1} \|X\|_{S_\alpha} \,.
\]
Setting $\ell=\alpha \in (2,\infty]$ possibly non-integer, $t=q/p$, 
$\eps_{\fluct}=3 \eps$, while
replacing $h_\ell(\cdot)$ by $g_\alpha(\cdot)$,
only three items of the proof of Theorem \ref{thm:cycles} require modification.
First, since $\|X\|_{S_\alpha} \ge \|X\|_\op \ge 
\Verts_0^{-1} \1_{[\Verts_0]}^T X \1_{[\Verts_0]}$
for any $X \in \Sym_\Verts(\R)$, verifying that
$\1_{[\Verts_0]}^T X_\star \1_{[\Verts_0]} = \Verts_0 (\Verts_0-1)$ for $X_\star$ of
\eqref{X.planted}, yields 
the analog of \eqref{cycpf2}. Similarly, having 
$\|\jay_\Verts \|_{S_\alpha} 
= \Verts(1+o(1)),$ yields the analog of \eqref{cycpf3}.
Lastly, replacing Claim \ref{claim2} with the following
substitute eliminates the factor $\ell$ of \eqref{bd:large-ell}, 
thereby handling also $\alpha=\infty$.
\begin{claim}\label{claim2b}
For any $\alpha \in (2,\infty]$, $\delta' \le \Verts^{-2}$,
$\eps>0$, $\by\in \Sigma\times \cV$ and $X\in \cB_\by(\eps)$, 
\begin{equation}\label{Sell-apx}
\big| g_\alpha(X) - g_\alpha(\Mat(\by))  \big| \le \eps + \Verts^{-1}\,.
\end{equation}
\end{claim}
\begin{proof} For $Y=M(\by)$ and $Z$ as in \eqref{def:By} and 
$\alpha \in (2,\infty]$, by the triangle inequality 
\[
\big| \| X \|_{S_\alpha} - \|Y+Z\|_{S_\alpha} \big| 
\le \| X-Y-Z \|_{S_\alpha} \le \|X-Y-Z\|_{\HS} \le \delta'\Verts \le \Verts^{-1}\,.
\]
By the same reasoning, 
\[
\big| \|Y+Z\|_{S_\alpha} - \|Y\|_{S_\alpha} \big| \le  \|Z\|_{S_\alpha} \le 
\eps \Verts p \,.
\]
Adding the preceding inequalities and dividing by $\Verts p \ge 1$, yields the bound \eqref{Sell-apx}.
\end{proof}

\begin{remark}
As with Theorem \ref{thm:cycles}, our argument yields a quantitative version of Proposition \ref{prop:Schatten}, which is the same as Theorem \ref{thm:cycles.quant} but with the integer $\ell \ge 3$ in \eqref{def:eps+}--\eqref{def:Ecomplex}
and $(\Verts p)^{-\ell} \hom(\Cyc_\ell, \Gnp)$, replaced by 
$\alpha \in (2,\infty]$ and 
$h(\Anp)=(\Verts p)^{-1} \|\Anp\|_{S_\alpha}$, respectively, 
where now $\eps_{\fluct} = 3 \eps$ and
$\upphi_{\Verts,p}(C_\ell, u),$ $\uppsi_{\Verts,p}(C_\ell, u)$, are correspondingly 
replaced with \eqref{def:phigen} and \eqref{def:psigen} 
for such $h$ and $t=1 \pm u$.
\end{remark}

\section{The upper tail for largest eigenvalues}
\label{sec:PF}

\Cref{thm:lam12} is a direct consequence of the following more general, quantitative bounds.
\begin{theorem}\label{thm:lam1.quant}
For $B\in \Sym_\Verts(\R)$ (non-random), let
\[
g_B:\cX_\Verts\to \R_+, \qquad  g_B(X)= \|X+B\|_\op \,.
\]
Then, for any such $B$ and all $\Verts\in \N$, $p \in (0,1)$,
$\delta \in (0,\frac13)$ and $t\ge0$, 
\begin{align}	\label{PF:bounds}
\phigen_p\big(g_B, \,(1-3\delta)t\big) - \Verts\log(9/\delta) 
&\,\le\, -\log\pr(g_B(\Anp)\ge t)  
\,\le\, \phigen_p\big( g_B,\, t+2 \big) + \log 2 \,.
\end{align}
\end{theorem}

\begin{remark} Slight modifications in the proof of \Cref{thm:lam1.quant}
yield the same bounds on the right-most eigenvalue, namely for 
$g_B^+(X) := \sup_{u\in \sphere^{\Verts-1} } \langle u, (X+B)u\rangle$. 
\end{remark}

\begin{proof}[Proof of \Cref{thm:lam12}] 
We start with \eqref{PF.tail}.
Fix $s=q/p > 1$. 
With $\Verts p \ge \kappa \log\Verts$, for  
all $\Verts$ large enough, 
$t=\frac{s-1}{2} \Verts p \ge C_\kappa \sqrt{\Verts p}$, so we deduce 
from \eqref{approx:op} at $k=1$ and the triangle inequality that
\begin{equation}\label{log-lamd1-ut}
-\log\pr(\|\Anp\|_\op \ge \Verts q) \,\gs_s\, (\Verts p)^2 \,.
\end{equation}
Combined with the \abbr{rhs} of \eqref{PF:bounds} for $B=0$,
this implies that for $\Verts^{-1}\log\Verts\lls p\le 1/2$,
\begin{equation}\label{lbd-phi-p}
\phigen_p(\|\cdot\|_\op, \Verts q) \,\gs_s\, (\Verts p)^2 \,.
\end{equation}
In particular, the upper bound in \eqref{PF.tail} on the \abbr{lhs}  
of \eqref{log-lamd1-ut} holds for any such $p$. In case
$p \gg \Verts^{-1/2}$, we have by \eqref{lbd-phi-p} that the leading term 
on the \abbr{lhs} of \eqref{PF:bounds} (at $B=0$), is at least 
$(\Verts p)^2 \gg\Verts$. We can then set $\delta(\Verts) \to 0$ sufficiently 
slowly for it to dominate the error term $\Verts \log (9/\delta)$, 
yielding the matching lower bound in \eqref{PF.tail}.

Turning to \eqref{lam2.tail}, taking
$p \ge \frac{\kappa}{\Verts} \log\Verts$ and $t \ge C_\kappa \sqrt{\Verts p}$ (so $t \gg p$), we get 
from \eqref{approx:op}, as before, that  
\begin{equation}\label{log-lamd2-ut}
-\log\pr( \|\Anp- p \1 \1^\tran \|_\op \ge t) \ggs t^2 \,.
\end{equation}
Setting hereafter $B=-p \1 \1^\tran$, combined with the \abbr{RHS} of \eqref{PF:bounds} this yields that 
\[
\phigen_p(\|\cdot \,-p \1 \1^\tran \|_\op, \,t) = 
\phigen_p(g_B,t) \ggs t^2 \,.
\]  
Thus, the upper bound in \eqref{lam2.tail} on the \abbr{lhs} of 
\eqref{log-lamd2-ut} holds for any such $t(\Verts)$ and $p(\Verts)$. Similarly to
our proof of \eqref{PF.tail}, when $t \gg \sqrt{\Verts}$ the leading term 
in the \abbr{LHS} of \eqref{PF:bounds} is much larger than $\Verts$, so taking
$\delta(\Verts) \to 0$ sufficiently slowly yields the 
matching lower bound.
\end{proof}

To establish Theorem \ref{thm:lam1.quant} will use the following 
standard converse to Proposition \ref{prop:cvx} for 
the case that $\cK$ is a closed half-space.
\begin{lemma}
\label{lem:HP-lb}
For $s \in \R$ and non-zero $v\in \R^d$, let 
$\cH_v(s)=\{x\in \R^d: \langle v,x\rangle\ge s\}$. Then,
\begin{equation}\label{eq:HP-lb}
\mu_p(\cH_v(s)) \ge  \frac{1}{2} \expo{ - I_p(\cH_v(s+ \sqrt{2} \|v\|_2)} .
\end{equation}
\end{lemma}
\begin{proof} Let $\bx\in \{0,1\}^d$ have distribution $\mu_p$ and
$\Lambda(\beta):= \log \e e^{\beta T}$, the \abbr{cgf} of 
$T:=\langle v , \bx \rangle$.
While proving \Cref{prop:cvx} we have seen 
that for any $\beta \ge 0$ and $y \in [0,1]^d$, 
\[
I_p(y) = \Lambda_p^\star(y) 
 \ge 
\beta \langle v,y \rangle - \Lambda(\beta) \,.
\]
Consequently, 
\begin{equation}\label{st-comp}
I_p(\cH_v(t)) = \inf_{\{y:\langle v,y \rangle \ge t\}} I_p(y) \ge 
\sup_{\beta \ge 0} \{\beta t - \Lambda(\beta) \} \,.
\end{equation}
Next, with $\e_\beta$ denoting expectation  
under the tilted product measure $\mu_{p,\beta}$ such that
\[
\frac{d \mu_{p,\beta}}{d \mu_p} = e^{\beta T -\Lambda(\beta)} \,,
\]
recall that $m_\beta := \e_\beta T = \Lambda'(\beta)$ is an 
increasing function, with 
$\Lambda'(\beta) \uparrow m_\infty
<\infty$ as $\beta \to \infty$.
In particular, setting $w= 2^{-1/2} \|v\|_2$, we deduce 
from \eqref{st-comp} that whenever $s + w \ge m_\infty$ we have  
$I_p(\cH_v(s+2w))=\infty$ 
and \eqref{eq:HP-lb} trivially holds. 
Further,
$\Var_\beta (T) = \Lambda''(\beta) \le \frac{1}{4} \|v\|_2^2$
for any $\beta$. Hence, for $J_\beta:=[m_\beta-w,m_\beta+w]$ we have
from Chebychev's inequality that
\begin{equation}\label{tpge}
\pr_\beta(T \notin J_\beta) \le w^{-2} \Var_\beta(T)
\le \frac{1}{2} \,.
\end{equation}
This yields \eqref{eq:HP-lb} when $s + w \le m_0$, since
\[
\mu_p(\cH_v(s)) = 1 - \pr_0(T < s) \ge 1 - \pr_0(T \notin J_0) \ge \frac{1}{2} \,.   
\]
If $s+w \in (m_0,m_\infty)$, then $s+w = m_\beta$ for some $\beta>0$ 
with $J_\beta \subseteq [s,s+2w]$. Hence, 
\begin{equation}\label{Hp.above}
\mu_p(\cH_v(s))
\ge \pr(T \in J_\beta) = 
e^{\Lambda(\beta)} \e_\beta \big[ e^{-\beta T} \ind(T \in J_\beta) \big] 
\ge e^{\Lambda(\beta)-\beta(s+2w)} \pr_{\beta} (T \in J_\beta) \,. 
\end{equation}
Combining \eqref{st-comp} at $t=s+2w$ with  
\eqref{tpge} and \eqref{Hp.above}, we again get \eqref{eq:HP-lb}.
\end{proof}

\begin{proof}[Proof of Theorem \ref{thm:lam1.quant}]
Starting with the lower bound in \eqref{PF:bounds}, 
let $\cV\subset \sphere^{\Verts-1}$ be a Euclidean $\delta$-net of size at most $(3/\delta)^\Verts$ (for example take $\rnk=1$ in \eqref{bd:Vnet}).
Note that for all $X\in \Sym_\Verts(\R)$, 
\begin{align}\label{g-cover}
g_B(X) = & \sup_{u,v\in \sphere^{\Verts-1}}\langle u, (X+B)v\rangle 
\\
&\ge \max_{u,v\in \cV} \langle u, (X+B) v\rangle \ge (1-3\delta) g_B(X) \,.
\label{gghat}
\end{align}
Indeed, for \eqref{gghat},
supposing that $u_\star=u_\star(X),v_\star=v_\star(X)$ attain the supremum in 
\eqref{g-cover}, 
there exist $\tilde{u},\tilde{v}\in \cV$ with $\|\tilde{u}-u_\star\|_2, \|\tilde{v}-v_\star\|_2 \le \delta$, whence
\begin{align*}
\langle  \tilde{u}, (X+B) \tilde{v}\rangle 
\ge \langle u_\star, (X+B) v_\star\rangle- (2\delta+\delta^2) \|X+B\|_\op \ge (1-3\delta)g_B(X) \,.
\end{align*}
Further, from 
\eqref{g-cover}, each super-level set $\cL_\ge (g_B,s)$, $s \ge 0$, 
is the union of the closed half-spaces
$\cH_{u,v}(s):= \{ X\in \Sym_\Verts(\R): \langle u , (X+B) v\rangle \ge s\}$
over $u,v\in \sphere^{\Verts-1}$. Consequently,
\begin{equation}\label{sup-level-gB}
\phi_p(g_B,s) = I_p(\cL_\ge(g_B,s)) = 
\inf_{u,v\in \sphere^{\Verts-1}} I_p\big(\cH_{u,v}(s)\big) \,.
\end{equation}
Thus, with $s=(1-3\delta) t$, applying \eqref{gghat}, the union bound and Proposition \ref{prop:cvx}
yields
\begin{align*}
\pr(g_B(\Anp)\ge t) &
\le \sum_{u,v\in \cV} \mu_p\big( \cH_{u,v} (s) \big)
\le |\cV|^2 \max_{u,v\in \cV} \big\{ e^{-I_p(\cH_{u,v}(s))} \big\}
\le |\cV|^2 e^{ -\phi(g_B,s)} \,.
\end{align*}
The lower bound in \eqref{PF:bounds} follows from substituting the bound on $|\cV|$ and taking logarithms.

Viewing $\Sym_\Verts^0(\R)\cong \R^d$ for $d={\Verts\choose2}$, 
we see that $\cH_{u,v}(t) = \cH_{y} (t-\langle u,Bv\rangle)$ 
for $\cH_y(\cdot)$ of \Cref{lem:HP-lb}, where 
$y \ne 0$ is the upper-triangular part of $u v^\tran + v u^\tran$. It is 
easy to check that 
\[
\|y\|_2^2 \le \|u\|_2^2 \|v\|_2^2 + \langle u,v \rangle^2 \le 2 \,,
\]
hence from Lemma \ref{lem:HP-lb},  
we have that
\begin{equation}	\label{cH-uplow}
 \mu_p\big( \cH_{u,v}(t) \big) \ge \frac12 \expo{ -I_p\big(\cH_{u,v}(t+2)\big)}\,.
\end{equation}
Now from 
the identities \eqref{g-cover}, \eqref{sup-level-gB}
and the bound \eqref{cH-uplow} we have
\begin{align*}
\pr(g_B(\Anp)\ge t) &
\ge \sup_{u,v\in \sphere^{\Verts-1}} \mu_p\big( \cH_{u,v}(t) \big)
\ge \frac12 \sup_{u,v\in \sphere^{\Verts-1}}
\big\{ e^{-I_p( \cH_{u,v}(t+2) )}  \big\}
= \frac12 e^{ -\phi_p(g_B,t+2)}\,,
\end{align*}
and the upper bound in \eqref{PF:bounds} follows.
\end{proof}

\section{Lower tails: proofs of Theorems \ref{thm:Sidorenko} and 
\ref{thm:Schatten.lower}}\label{sec:lower}

In proving Theorems \ref{thm:Sidorenko} and \ref{thm:Schatten.lower} we 
set for $r \in (0,1)$, $\alpha\in[2,\infty]$ and $\eps>0$,
\begin{equation}	\label{def:Bqe}
\cB_r(\eps,\alpha) := \{ X\in \cX_\Verts: \|X - r\jay_\Verts\|_{S_{\alpha}} \le \eps r\Verts \} 
\end{equation}
and employ the following consequence of \Cref{lem:LB.ball}.
\begin{lemma}
\label{lem:tilt}
Let $\alpha\in[2,\infty]$ and $\Verts^{-1} \log\Verts \le r<p\le1/2$. If 
\begin{equation}	\label{tilt:eps}
\eps \ge \frac{C_1 \Verts^{1/\alpha}}{\sqrt{\Verts r}} 
\end{equation}
for $C_1$ of Lemma \ref{lem:gap}, then 
\begin{equation}\label{lbd-alpha-ball}
\mu_p( \cB_r(\eps,\alpha))  \ge \frac{1}{2} 
e^{-{\Verts\choose2} I_p(r)} e^{-2 \Verts \sqrt{2p}} \,.
\end{equation}
\end{lemma}
\begin{proof} Since
$\|\cdot\|_{S_\alpha} \le \Verts^{1/\alpha} \|\cdot\|_{\op}$,
thanks to condition \eqref{tilt:eps},  
\[
\{ X\in \cX_\Verts \, : \|X - r\jay_\Verts\|_\op \le C_1 \sqrt{\Verts r} \} 
\subseteq \cB_r(\eps,\alpha) \,,
\]
so \eqref{lbd-alpha-ball} is an immediate consequence of \Cref{lem:LB.ball} for 
the special case of $Q=r \jay_\Verts$.
\end{proof}

\begin{proof}[Proof of Theorem \ref{thm:Sidorenko}]
We first establish \eqref{Sido:upper}.
The event on the \abbr{lhs} 
of \eqref{Sido:upper} is $t(H,f_\Anp) \le {\widehat q}^{|E|}$, which by 
the Sidorenko property \eqref{Sidorenko} 
is contained in the event
$t_{K_2}(f_\Anp) \le {\widehat q}$. The latter is the restriction  
\[
\sum_{1 \le i<j\le \Verts} A_{ij} \le \tiny{\Verts\choose2} q \,.
\]
Since the \abbr{lhs} has the $\Bin({\Verts\choose2}, p)$ distribution, the claim follows from a classical result for tails of the binomial distribution (or one can apply \Cref{prop:lower-cvx} and follow the lines after 
\eqref{Schatten.lower-mono} in the proof of \Cref{thm:Schatten.lower} below).

Turning to the lower bound \eqref{Sido:lower}, recall that 
$\hom(F, r \jay_\Verts) \le \Verts^{\verts(F)} r^{\edges(F)}$ for any subgraph $F$
and $r\in (0,1)$. Hence, \Cref{prop:countingB}
applies with $K=1$ and
$r$ in place of $p$, 
for $\cB=\cB_r(\eps,\infty)$ and any $\eps = \frac{\eps_0}{2} r^{\Delta_\star-1}$,
$\eps_0 \le 1$. Thus, for some $C=C(\edges)$ and any such $\eps$,
\[
\sup_{X\in \cB_r(\eps, \infty)} \;
\big| \hom(H, X) - \hom(H, r\jay_\Verts) \big| \le \eps_0 C \, r^\edges \Verts^\verts \,,
\]
implying by the triangle inequality that for all $X \in \cB_r(\eps,\infty)$,
\[
\hom(H, X) \le \hom(H, r\jay_\Verts) + \eps_0 C r^\edges \Verts^\verts \le (1+ \eps_0 C) r^\edges \Verts^\verts \,.
\]
For $r= \wh{q}/(1+\eps_0 C)^{1/\edges}$ the \abbr{rhs} is at most 
$\wh{q}^{~ \edges} \Verts^\verts$, hence
\[
\{ X\in \cX_\Verts: \hom(H, X) \le \wh{q}^{~ \edges}\Verts^\verts \} \supseteq \cB_r(\eps, \infty)\,.
\]
Thanks to our assumption that $p^{\Delta_\star-1} \gg 1/\sqrt{\Verts p}$,
taking $q=sp$ for fixed $s\in (0,1)$, \Cref{lem:tilt} applies for 
$\alpha=\infty$ and some $\eps_0(\Verts) \to 0$ (such that $\eps \ge C_1/\sqrt{\Verts r}$), giving
\begin{align*}
\pr( \hom(H, \Anp) \le \wh{q}^{~ \edges}\Verts^\verts ) \ge \mu_p(\cB_r(\eps, \infty)) 
 \ge \frac{1}{2} e^{-{\Verts\choose2} I_p(r) - 2 \Verts \sqrt{2p}} \,.
\end{align*}
This completes the proof, since $I_p(r)/I_p(s p) \to 1$ and 
$p^{-1} I_p(s p)$ is bounded away from zero for such $p=p(\Verts)$ and $r=r(\Verts)$.
\end{proof}

\begin{proof}[Proof of \Cref{thm:Schatten.lower}]
We first prove \eqref{Schatten.upper}.
The first inequality is a direct consequence of  \Cref{prop:lower-cvx}. For the second inequality in \eqref{Schatten.upper}
it suffices to show that
\begin{equation}	\label{Schatten.lower:goal}
\inf\big\{ I_p(X): X\in \cX_\Verts, \, \|X\|_{S_\alpha} \le (\Verts-1)q \big\} \ge {\Verts\choose2} I_p(q).
\end{equation}
If $X\in \cX_\Verts$ is such that $ \|X\|_{S_\alpha} \le (\Verts-1) q$, then by the monotonicity of $\beta\mapsto \|\cdot \|_{S_\beta}$
\begin{equation}	\label{Schatten.lower-mono}
\frac{1}{{\Verts\choose2}}\sum_{1 \le i<j\le \Verts} X_{ij} = \frac1{\Verts(\Verts-1)}\1^\tran X\1 \le \frac1{\Verts-1} \|X\|_\op \le \frac1{\Verts-1}\|X\|_{S_{\alpha}} \le q \,.
\end{equation}
Since $I_p(\cdot)$ is convex on $[0,1]$ and decreasing on $[0,p]$, it follows
from 
\eqref{Schatten.lower-mono} that 
\[
\frac1{{\Verts\choose2}} I_p(X) =\frac1{{\Verts\choose2}}  \sum_{1 \le i < j \le \Verts} I_p(X_{ij}) \ge I_p\Big( \frac1{{\Verts\choose2}}  \sum_{1 \le i < j \le \Verts} X_{ij} \Big)  \ge I_p(q) \,,
\]
for all $X\in \cX_\Verts$ such that $ \|X\|_{S_{\alpha}} \le (\Verts-1)q$. This yields \eqref{Schatten.lower:goal} and thereby \eqref{Schatten.upper}.

Turning to the lower bound \eqref{Schatten.lower}, by the 
triangle inequality and monotonicity of $\beta \mapsto \|\cdot\|_{S_\beta}$,
we have that for any $r \in (0,1)$, $\alpha \in [2,\infty]$ and $X\in \cB_r(\eps, \alpha)$,
\[
\|X\|_{S_\alpha} \le r \|\jay_\Verts\|_{S_\alpha} + \|X-r\jay_\Verts\|_{S_\alpha} 
\le r \|\jay_\Verts\|_{\HS} +  \eps r\Verts  \le (1+\eps) r\Verts \,.
\]
For $r=q/(1+\eps)^2$ and $\eps \ge 1/(\Verts-1)$ the \abbr{rhs} is at 
most $q (\Verts-1)$, hence
\[
\{ X\in \cX_\Verts : \|X\|_{S_\alpha} \le q (\Verts-1) \} \supseteq \cB_r(\eps,\alpha) \,. 
\]
Taking $q=s p$ for fixed $s \in (0,1)$, thanks to our assumption 
that $\sqrt{\Verts p} \gg \Verts^{1/\alpha}$
(or $\Verts p \gg \log\Verts$ in case $\alpha=\infty$),
Lemma \ref{lem:tilt} 
applies for  
some $\eps=\eps(\Verts) \to 0$. 
The proof then concludes exactly as in the proof of Theorem \ref{thm:Sidorenko}.
\end{proof}

\subsection*{Acknowledgments}

$~$ We thank Sourav Chatterjee, Ronen Eldan, Ofer Zeitouni and Alex Zhai for helpful discussions and for their encouragement.  
We also thank Gady Kozma and Wojciech Samotij for providing us with an early version of their work \cite{KoSa}.
We thank Anirban Basak and the anonymous referee for feedback that improved the accuracy and clarity of presentation, and in particular for pointing out an error in our original argument for \eqref{LF.bd}, which led us to find the simpler and sharper inductive argument presented here.

\bibliographystyle{myalpha}
\bibliography{LDT.bib}

\begin{thebibliography}{NOdM19}

\bibitem[Aug]{Augeri18}
F. Augeri.
\newblock Nonlinear large deviation bounds with applications to traces of
  {W}igner matrices and cycles counts in {E}rd{\"o}s-{R}enyi graphs.
\newblock Preprint, arXiv:1810.01558.

\bibitem[Aus19]{Austin-NMF}
T. Austin.
\newblock The structure of low-complexity {G}ibbs measures on product spaces.
\newblock {\em Ann. Probab.}, 46(6):4002--4023, 2019.

\bibitem[BB]{BaBa20}
A. Basak and R. Basu.
\newblock Upper tail large deviations of regular subgraph counts in erd{\H
  o}s-r{\'e}nyi graphs in the full localized regime.
\newblock Preprint, arXiv:1912.11410.

\bibitem[BCCH17]{BCCH}
C. Borgs, J.~T. Chayes, H. Cohn, and N. Holden.
\newblock Sparse exchangeable graphs and their limits via graphon processes.
\newblock {\em J. Mach. Learn. Res.}, 18:Paper No. 210, 71, 2017.

\bibitem[BCCZ18]{BCCZ-Lp2}
C. Borgs, J.~T. Chayes, H. Cohn, and Y. Zhao.
\newblock An {$L^p$} theory of sparse graph convergence {II}: {LD} convergence,
  quotients and right convergence.
\newblock {\em Ann. Probab.}, 46(1):337--396, 2018.

\bibitem[BCCZ19]{BCCZ-LP1}
C. Borgs, J. Chayes, H. Cohn, and Y. Zhao.
\newblock An {$L^p$} theory of sparse graph convergence {I}: limits, sparse
  random graph models, and power law distributions.
\newblock {\em Trans. Amer. Math. Soc.}, 372:3019--3062, 2019.

\bibitem[BG]{BhGa18}
B.~B. Bhattacharya and S. Ganguly.
\newblock Upper tails for edge eigenvalues of random graphs.
\newblock Preprint, arXiv:1811.07554, {\em SIAM J. Disc. Math.}, to appear,
  2020.

\bibitem[BGBK]{BBK}
F. Benaych-Georges, C. Bordenave, and A. Knowles.
\newblock Spectral radii of sparse random matrices.
\newblock Preprint, arXiv:1704.02945.

\bibitem[BGLZ17]{BGLZ16}
B.~B. Bhattacharya, S. Ganguly, E. Lubetzky, and Y. Zhao.
\newblock Upper tails and independence polynomials in random graphs.
\newblock {\em Adv. Math.}, 319:313--347, 2017.

\bibitem[BGSZ20]{BGSZ16}
B.~B. Bhattacharya, S. Ganguly, X. Shao, and Y. Zhao.
\newblock Upper tails for arithmetic progression in a random set.
\newblock {\em Int. Math. Res. Not.}, 2020(1):167--213, 2020.

\bibitem[BS]{BaSz18}
A. Backhausz and B. Szegedy.
\newblock Action convergence of operators and graphs.
\newblock Preprint, arXiv:1811.00626.

\bibitem[CD16]{ChDe14}
S. Chatterjee and A. Dembo.
\newblock Nonlinear large deviations.
\newblock {\em Advances in mathematics}, 299:396--450, 2016.

\bibitem[CFS10]{CFS:Sidorenko}
D. Conlon, J. Fox, and B. Sudakov.
\newblock An approximate version of {S}idorenko's conjecture.
\newblock {\em Geom. Funct. Anal.}, 20(6):1354--1366, 2010.

\bibitem[Cha12]{Chatterjee12-triangles}
S. Chatterjee.
\newblock The missing log in large deviations for triangle counts.
\newblock {\em Random Structures Algorithms}, 40(4):437--451, 2012.

\bibitem[Cha16]{Chatterjee:survey}
S. Chatterjee.
\newblock An introduction to large deviations for random graphs.
\newblock {\em Bull. Amer. Math. Soc. (N.S.)}, 53(4):617--642, 2016.

\bibitem[CKLL18]{CKLL}
D. Conlon, J.~H. Kim, C. Lee, and J. Lee.
\newblock Some advances on {S}idorenko's conjecture.
\newblock {\em J. Lon. Math. Soc.}, 98(3):593--608, 2018.

\bibitem[CV11]{ChVa11}
S. Chatterjee and S.~R.~S. Varadhan.
\newblock The large deviation principle for the {E}rd{\H{o}}s--{R}\'enyi random
  graph.
\newblock {\em European J. Combin.}, 32(7):1000--1017, 2011.

\bibitem[DK12a]{DeKa12:cliques}
B. Demarco and J. Kahn.
\newblock Tight upper tail bounds for cliques.
\newblock {\em Random Structures Algorithms}, 41(4):469--487, 2012.

\bibitem[DK12b]{DeKa12:triangles}
B. DeMarco and J. Kahn.
\newblock Upper tails for triangles.
\newblock {\em Random Structures Algorithms}, 40(4):452--459, 2012.

\bibitem[EG18]{ElGr-decomp}
R. Eldan and R. Gross.
\newblock Decomposition of mean-field {G}ibbs distributions into product
  measures.
\newblock {\em Electron. J. Probab.}, 23:Paper No. 35, 24, 2018.

\bibitem[Eld18]{Eldan-NMF}
R. Eldan.
\newblock Gaussian-width gradient complexity, reverse log-{S}obolev
  inequalities and nonlinear large deviations.
\newblock {\em Geom. Funct. Anal.}, 28(6):1548--1596, 2018.

\bibitem[FK81]{FuKo81}
Z. F{\"u}redi and J. Koml{\'o}s.
\newblock The eigenvalues of random symmetric matrices.
\newblock {\em Combinatorica}, 1(3):233--241, 1981.

\bibitem[FK99]{FrKa99}
A. Frieze and R. Kannan.
\newblock A simple algorithm for constructing {S}zemer\'{e}di's regularity
  partition.
\newblock {\em Electron. J. Combin.}, 6:Research Paper 17, 7, 1999.

\bibitem[Fre18]{Frenkel}
P.~E. Frenkel.
\newblock Convergence of graphs with intermediate density.
\newblock {\em Trans. Amer. Math. Soc.}, 370(5):3363--3404, 2018.

\bibitem[GH]{GuHu18}
A. Guionnet and J. Husson.
\newblock Large deviations for the largest eigenvalue of rademacher matrices.
\newblock Preprint, arXiv:1810.01188. {\em Ann. Probab.}, to appear, 2020.

\bibitem[Hat10]{Hatami10}
H. Hatami.
\newblock Graph norms and {S}idorenko's conjecture.
\newblock {\em Israel J. Math.}, 175:125--150, 2010.

\bibitem[HMS]{HMS19}
M. Harel, F. Mousset, and W. Samotij.
\newblock Upper tails via high moments and entropic stability.
\newblock Preprint, arXiv:1904.08212.

\bibitem[Jan]{Janson16:sigma-finite}
S. Janson.
\newblock Graphons and cut metric on $\sigma$-finite measure spaces.
\newblock Preprint, arXiv:1608.01833.

\bibitem[JOR04]{JOR04}
S. Janson, K. Oleszkiewicz, and A. Ruci{\'n}ski.
\newblock Upper tails for subgraph counts in random graphs.
\newblock {\em Israel J. Math.}, 142:61--92, 2004.

\bibitem[JR02]{JaRu02}
S. Janson and A. Ruci{\'n}ski.
\newblock The infamous upper tail.
\newblock {\em Random Structures Algorithms}, 20(3):317--342, 2002.
\newblock Probabilistic methods in combinatorial optimization.

\bibitem[JW16]{JaWa}
S. Janson and L. Warnke.
\newblock The lower tail: {P}oisson approximation revisited.
\newblock {\em Random Structures Algorithms}, 48(2):219--246, 2016.

\bibitem[KS]{KoSa}
G. Kozma and W. Samotij.
\newblock Private communication.

\bibitem[KV04]{KiVu04}
J.~H. Kim and V.~H. Vu.
\newblock Divide and conquer martingales and the number of triangles in a
  random graph.
\newblock {\em Random Structures Algorithms}, 24(2):166--174, 2004.

\bibitem[LHY18]{LHY}
R. Lata{\l}a, R.~V. Handel, and P. Youssef.
\newblock The dimension-free structure of nonhomogeneous random matrices.
\newblock {\em Inventiones Math.}, 214(3):1031--1080, 2018.

\bibitem[Lov12]{Lovasz-book}
L. Lov{\'a}sz.
\newblock {\em Large networks and graph limits}, volume~60 of {\em American
  Mathematical Society Colloquium Publications}.
\newblock American Mathematical Society, Providence, RI, 2012.

\bibitem[LZ15]{LuZh12}
E. Lubetzky and Y. Zhao.
\newblock On replica symmetry of large deviations in random graphs.
\newblock {\em Random Structures Algorithms}, 47(1):109--146, 2015.

\bibitem[LZ17]{LuZh14}
E. Lubetzky and Y. Zhao.
\newblock On the variational problem for upper tails in sparse random graphs.
\newblock {\em Random Structures Algorithms}, 50(3):420--436, 2017.

\bibitem[NOdM19]{NeOs18}
J. Ne\v{s}et\v{r}il and P. Ossona~de Mendez.
\newblock Local-global convergence, an analytic and structural approach.
\newblock {\em Comment. Math. Univ. Carolin.}, 60(1):97--129, 2019.

\bibitem[Sid93]{Sidorenko}
A. Sidorenko.
\newblock A correlation inequality for bipartite graphs.
\newblock {\em Graphs Combin.}, 9(2):201--204, 1993.

\bibitem[Sim84]{Simonovits84}
M. Simonovits.
\newblock Extremal graph problems, degenerate extremal problems, and
  supersaturated graphs.
\newblock In {\em Progress in graph theory ({W}aterloo, {O}nt., 1982)}, pages
  419--437. Academic Press, Toronto, ON, 1984.

\bibitem[Sim05]{Sim05b}
B. Simon.
\newblock {\em Trace ideals and their applications}, volume 120 of {\em
  Mathematical Surveys and Monographs}.
\newblock American Mathematical Society, Providence, RI, second edition, 2005.

\bibitem[Sio58]{Sio58}
M. Sion.
\newblock On general minimax theorems.
\newblock {\em Pac. J. Math.}, 8:171--176, 1958.

\bibitem[{\v S}W19]{SiWa18}
M. {\v S}ileikis and L. Warnke.
\newblock A counterexample to the {D}e{M}arco--{K}ahn upper tail conjecture.
\newblock {\em Random Structures Algorithms}, 55(4):775--794, 2019.

\bibitem[Szea]{Szegedy:Sidorenko}
B. Szegedy.
\newblock An information theoretic approach to {S}idorenko's conjecture.
\newblock Preprint, arXiv:1406.6738.

\bibitem[Szeb]{Szegedy15:sparse}
B. Szegedy.
\newblock Sparse graph limits, entropy maximization and transitive graphs.
\newblock Preprint. arXiv:1504.00858.

\bibitem[Sze11]{Szegedy11}
B. Szegedy.
\newblock Limits of kernel operators and the spectral regularity lemma.
\newblock {\em European Journal of Combinatorics}, 32(7):1156 -- 1167, 2011.
\newblock Homomorphisms and Limits.

\bibitem[Tal96]{Talagrand-newlook}
M. Talagrand.
\newblock A new look at independence.
\newblock {\em Ann. Probab.}, 24(1):1--34, 1996.

\bibitem[Tao12]{Tao-blog-regularity-lemma}
T. Tao.
\newblock The spectral proof of the szemeredi regularity lemma.
\newblock URL:
  https://terrytao.wordpress.com/2012/12/03/the-spectral-proof-of-the-szemeredi-regularity-lemma/,
  December 2012.

\bibitem[Zha17]{Zhao:lower}
Y. Zhao.
\newblock On the lower tail variational problem for random graphs.
\newblock {\em Combin. Probab. Comput.}, 26(2):301--320, 2017.

\end{thebibliography}

\end{document}